\documentclass[reqno]{amsart}
\usepackage{amssymb}
\usepackage{amscd}
\usepackage[all,cmtip,rotate]{xy}
\usepackage{tikz}
\usetikzlibrary{matrix,shapes,arrows,positioning,decorations.pathmorphing}
\usepackage{tikzscale}
\tikzstyle{block} = [rectangle, draw, fill=blue!20,
    text width=6em, text centered, rounded corners, minimum height=6em]
\tikzstyle{line} = [draw, -latex']
\usepackage{comment}

\theoremstyle{plain}
\newtheorem{theorem}{Theorem}[section]
\newtheorem{proposition}[theorem]{Proposition}
\newtheorem{lemma}[theorem]{Lemma}
\newtheorem{corollary}[theorem]{Corollary}

\theoremstyle{definition}

\newtheorem{example}[theorem]{Example}
\newtheorem{question}[theorem]{Question}
\newtheorem{remark}[theorem]{Remark}

\DeclareMathOperator{\Cl}{{\it Cl}}
\DeclareMathOperator{\Pic}{{Pic}}
\DeclareMathOperator{\reg}{{reg}}

\DeclareMathOperator{\rdim}{reg-dim}
\DeclareMathOperator{\Hom}{Hom}
\DeclareMathOperator{\GV}{GV}
\DeclareMathOperator{\rGV}{rGV}
\DeclareMathOperator{\h}{ht}
\DeclareMathOperator{\Max}{Max}
\DeclareMathOperator{\Inv}{Inv}
\DeclareMathOperator{\Prin}{Prin}

\begin{document}

\title[Prime Factorization of Ideals]{Prime Factorization of ideals in commutative rings, with a focus on Krull rings}

\author[G.W. Chang]{Gyu Whan Chang}
\address{Gyu Whan Chang \\ Department of Mathematics Education \\ Incheon National University \\ Incheon 22012 \\ Republic of Korea}
\email{whan@inu.ac.kr}

\author[J.S. Oh]{Jun Seok Oh}
\address{Jun Seok Oh \\ Department of Mathematics Education \\ Jeju National University \\ Jeju 63243 \\ Republic of Korea}
\email{junseok.oh@jejunu.ac.kr}


\maketitle

\begin{abstract}
Let $R$ be a commutative ring with identity.
The structure theorem says that $R$ is a PIR (resp., UFR, general ZPI-ring, $\pi$-ring) if and only if $R$ is a finite direct product of PIDs (resp., UFDs, Dedekind domains, $\pi$-domains) and special primary rings.
All of these four types of integral domains are Krull domains, so motivated by the structure theorem, we study the prime factorization of ideals in a ring that is a finite direct product of Krull domains and special primary rings.
Such a ring will be called a general Krull ring.
It is known that Krull domains can be characterized by the star operations $v$ or $t$ as follows:
An integral domain $R$ is a Krull domain if and only if every nonzero proper principal ideal of $R$ can be written as a finite $v$- or $t$-product of prime ideals.
However, this is not true for general Krull rings.
In this paper, we introduce a new star operation $u$ on $R$, so that $R$ is a general Krull ring if and only if every proper principal ideal of $R$ can be written as a finite $u$-product of prime ideals.
We also study several ring-theoretic properties of general Krull rings including  Kaplansky-type theorem, Mori-Nagata theorem, Nagata rings, and Noetherian property.
\end{abstract}

\vspace{.5cm}
\setcounter{tocdepth}{1}
\tableofcontents

\section{Introduction}

Multiplicative Ideal Theory is a branch of Commutative Ring Theory in which (unique or non-unique) ideal factorization properties have been studied.
The class of integral domains with ideal factorization properties includes principal ideal domains (PIDs), Dedekind domains, unique factorization domains (UFDs), $\pi$-domains, and Krull domains, and the following diagram shows the relationship of these five types of integral domains;

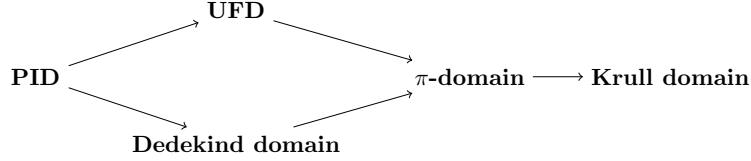
\begin{figure}[h]
\resizebox{.8\textwidth}{!}{%
\begin{tikzpicture}[node distance=2cm]
\title{Untergruppenverband der}
\node(PID)    at (0,0)     {\bf PID};
\node(UFD)    at (3,1)     {\bf UFD};
\node(Dd)     at (3,-1)    {\bf Dedekind domain};
\node(PiD)    at (6.5,0)   {\bf $\pi$-domain};
\node(KD)     at (9.5,0)   {\bf Krull domain};

\draw[->]  (PID)  -- (UFD);
\draw[->]  (PID)  -- (Dd);
\draw[->]  (UFD)  -- (PiD);
\draw[->]  (Dd)   -- (PiD);
\draw[->]  (PiD)  -- (KD);
\end{tikzpicture}
}
\end{figure}

A rank-one discrete valuation ring (DVR) is just a PID with a unique nonzero prime ideal.
A Krull domain is defined by a locally finite intersection of DVRs,
a UFD is a Krull domain with trivial divisor class group, and
a PID that is not a field is a one-dimensional UFD.
A Dedekind domain (resp., $\pi$-domain) is an integral domain in which each ideal
(resp. principal ideal) can be written as a finite product of prime ideals.
A Krull domain also has a prime ideal factorization property, which is similar to those of Dedekind domains and $\pi$-domains.
But, in order to characterize Krull domains by a prime ideal factorization property,
we need the notion of star operations called the $v$- and $t$-operation.
Then an integral domain $D$ is a Krull domain if and only if every nonzero principal ideal $aD$ of $D$ can be written as a finite $v$-product of prime ideals of $D$, i.e., $aD = (P_1 \cdots P_n)_v$ for some prime ideals $P_1, \ldots, P_n$, if and only if every nonzero principal ideal of $D$ is a finite $t$-product of prime ideals.

The five types of integral domains with prime ideal factorization properties can be generalized to rings with zero divisors by at least two ways:
\begin{itemize}
\item PIR, UFR, general ZPI-ring, $\pi$-ring,

\item regular PIR, factorial ring, Dedekind ring, regular $\pi$-ring, Krull ring.
\end{itemize}
For example, $R$ is a UFR (resp., factorial ring) if and only if every element (resp., regular element) of $R$ can be written as a finite product of prime elements.
A PIR is a special primary ring (SPR) if it has only one prime ideal.
Then $R$ is a PIR (resp., UFR, general ZPI-ring, $\pi$-ring) if and only if $R$ is a finite direct product of PIDs (resp., UFDs, Dedekind domains, $\pi$-domains) and SPRs.
Now, we will call a ring $R$ {\em a general Krull ring} if $R$ is a finite direct product of Krull domains and SPRs.
Hence, $\pi$-rings are general Krull rings and general Krull rings are Krull rings.

Recall that $R$ is a Krull ring if and only if every regular principal ideal of $R$ can be written
as a finite $v$-product (or $t$-product) of prime ideals \cite[Theorem 13]{7}.
However, the next example shows that this is not true of general Krull rings.

\begin{example} \label{ex0.1}
Let $\mathbb{Z}$ be the ring of integers, $\mathbb{Q}$ be the field of rational numbers, and
$R = \mathbb Z \times \mathbb Q$ be the direct product of $\mathbb{Z}$ and $\mathbb{Q}$.
Then $\mathbb Z$ and $\mathbb Q$ are Krull domains, so $R$ is a general Krull ring.
However, if $\langle(1,0)\rangle$ is the ideal of $R$ generated by $(1,0)$,
then $\langle(1,0)\rangle_t = \langle(1,0)\rangle_v = R$. Hence, $\langle(1,0)\rangle$
cannot be written as a finite $t$- nor $v$-product of prime ideals.
\end{example}

Let $D$ be a Krull domain. It is easy to see that $D$ is a Krull domain if and only if
there is a star operation $*$ on $D$ such that each nonzero proper principal ideal of
$D$ can be written as a finite $*$-product of prime ideals \cite[Theorem 1.1]{l72}.
Hence, we have the following natural question.

\begin{question} \label{quest}
Is there a star operation $*$ on a ring so that a general Krull ring can be characterized as a ring
in which each principal ideal can be written as a finite $*$-product of prime ideals?
\end{question}

The purpose of this paper is to give an answer to Question \ref{quest} affirmatively.
That is, in this paper, we introduce a new star operation $u$ on a ring $R$ and
we show that $R$ is a general Krull ring if and only if every proper principal ideal of $R$ is written as a finite $u$-product of prime ideals.

This paper consists of eight sections including introduction.
In Section~\ref{1}, we review the definition and basic properties of star operations for easy reference of the reader.
Section~\ref{2} is devoted to the historical overview of PIDs, UFDs, Dedekind domains, $\pi$-domains, Krull domains,
and their generalizations to rings with zero divisors. In Section 3,
we study the $w$-operation on integral domains and its two generalizations to rings with zero divisors.
We also review some basic properties of the finite direct product of rings,
which are very useful for the study of general Krull rings.
In Section~\ref{3}, we introduce and study
a new star operation $u$ and its relationship with two other star operations which are
reviewed in Section 3.
In Section~\ref{4}, which is the main section of this paper, we characterize
a general Krull ring via the star operation $u$ introduced in Section~\ref{3}.
Among others, we show that $R$ is a general Krull ring if and only if every proper principal ideal of
$R$ can be written as a finite $u$-product of prime ideals, if and only if
every nonprincipal prime ideal of $R$ contains a $u$-invertible prime ideal.
In Section~\ref{5}, we study the Nagata ring of general Krull rings,
which is a generalization of Krull domains that $R$ is a Krull domain if and only if
the $t$-Nagata ring of $R$ is a PID. We prove that $R$ is a $u$-Noetherian ring if and only if
the polynomial ring $R[X]$ is a $u$-Noetherian ring, if and only if the $u$-Nagata ring of $R$ is a Noetherian ring.
It is shown, in Section~\ref{4}, that a general Krull ring is a $u$-Noetherian ring.
Finally, in Section~\ref{6}, we study when a $u$-Noetherian ring is a general Krull ring.
For example, we show that $R$ is a general Krull ring if and only if
$R$ is a $u$-Noetherian ring such that $R_P$ is a DVR or an SPR for all maximal $u$-ideals $P$ of $R$.
We also study a ring $R$ in which $R_P$ is a DVR or an SPR for all maximal $u$-ideals $P$ of $R$.

\section{Star operations on rings} \label{1}

All rings considered in this paper are commutative rings with (nonzero multiplicative) identity.
Let $R$ be a ring with total quotient ring $T(R)$.
An overring of $R$ means a subring of $T(R)$ containing $R$.
Let $Z(R)$ be the set of zero divisors in $R$ containing the zero element of $R$.
A regular element is an element that is not a zero divisor.
For a subset $E$ of $R$, let $\reg (E)$ be the set of regular elements of $R$ in $E$,
so $\reg(R) = R \setminus Z(R)$ and $T(R) = R_{\reg(R)}$.

\subsection{Star operations}

An $R$-submodule of $T(R)$ is called a {\it Kaplansky fractional ideal} \cite[page 37]{8}.
A Kaplansky fractional ideal of $R$ is regular if it contains a regular element of $R$.
Let $\mathsf{K} (R)$ be the set of Kaplansky fractional ideals of $R$,
$F(R)$ be the set of fractional ideals of $R$ (i.e., $I \in F(R)$ if and only if
$I \in \mathsf K(R)$ and $dI \subseteq R$ for some $d \in \reg (R)$), and $F^{\reg} (R)$ be
the set of regular fractional ideals of $R$; so $F^{\reg} (R) \subseteq F(R) \subseteq \mathsf K(R)$.
An (integral) ideal of $R$ is a fractional ideal of $R$ that is contained in $R$.

A mapping $* \colon \mathsf K (R) \to \mathsf K (R)$, given by $I \mapsto I_*$, is called a {\it star operation} on $R$
if the following four conditions are satisfied for all $I, J \in \mathsf K (R)$ and $a \in T(R)$:
\begin{enumerate}
\item $R_* = R$,

\item $aI_* \subseteq (aI)_*$, and equality holds when $a$ is regular.

\item $I \subseteq I_*$, and $I \subseteq J$ implies that $I_* \subseteq J_*$.

\item $(I_*)_* = I_*$.
\end{enumerate}
For all $I \in \mathsf K (R)$, let
\[
  I_{*_f} = \bigcup \{ J_* \mid J \in \mathsf K (R) \mbox{ is finitely generated and } J \subseteq I \} \,.
\]
Then $*_f$ is also a star operation on $R$. The star operation $*$ is said to be {\it of finite type} if $* = *_f$,
and $*$ is said to be {\it reduced} if $(0)_* = (0)$.
Clearly, $*_f$ is of finite type, and $*$ is reduced if and only if $*_f$ is reduced.

\smallskip
\begin{lemma}{\cite[page 94]{e19}} \label{star}
Let $*$ be a star operation on a ring $R$. Then $(IJ_* )_* = (IJ)_*$ for all $I, J \in \mathsf K (R)$.
\end{lemma}

\begin{proof}
If $I, J \in \mathsf K (R)$, then $IJ \subseteq IJ_* \subseteq (IJ)_*$ by the properties (2) and (3) of star operations.
Thus, $(IJ_*)_* = (IJ)_*$ by the property (4).
\end{proof}

\smallskip
An $I \in \mathsf K (R)$ is said to be a {\it $*$-ideal} if $I_* = I$.
A $*$-ideal $I$ is {\it of finite type} if $I = J_*$ for some finitely generated subideal $J$ of $I$.
Clearly, a Kaplansky fractional $*$-ideal of finite type is a fractional ideal.
A $*$-ideal is a {\it maximal $*$-ideal} if it is maximal among proper integral $*$-ideals.
Let $*$-$\Max (R)$ denote the set of maximal $*$-ideals of $R$.
The following proposition shows that if $*$ is of finite type and $R \subsetneq T(R)$,
then $*$-$\Max (R) \neq \emptyset$. The results of Proposition \ref{prop1.2}
have been noted in the literature (for example, see \cite{7})
but we have been unable to find proofs, so we include the proofs for easy reference.

\begin{proposition} \label{prop1.2}
Let $*$ be a star operation of finite type on a ring $R$.
Then the following statements hold.
\begin{enumerate}
\item A prime ideal minimal over an integral $*$-ideal is a $*$-ideal.

\item A proper integral $*$-ideal is contained in a maximal $*$-ideal.

\item A maximal $*$-ideal is a prime ideal.
\end{enumerate}
\end{proposition}

\begin{proof}
(1) Let $I$ be an integral $*$-ideal of $R$ and $P$ be a prime ideal of $R$ that is minimal over $I$.
Let $J$ be a finitely generated subideal of $P$. Since $P$ is minimal over $I$,
there are an integer $n \geq 1$ and an element $y \in R \setminus P$
such that $yJ^n \subseteq I$ (cf. \cite[Theorem 2.1]{4}).
Hence, by Lemma~\ref{star},
$$y(J_*)^n \subseteq y(J^n)_* \subseteq I_* = I \subseteq P,$$
whence $J_* \subseteq P$. Thus, $P_* = P$ because $*$ is of finite type.

(2) Let $I$ be an integral $*$-ideal of $R$.
Clearly, if $\{P_{\alpha}\}$ is a chain of prime $*$-ideals of $R$ containing $I$, then
$\bigcup_{\alpha}P_{\alpha}$ is a prime $*$-ideal of $R$  containing $I$. So, by
Zorn's lemma, there is at least one maximal $*$-ideal of $R$ containing $I$.

(3) Let $P$ be a maximal $*$-ideal of $R$ and $a, b \in R \setminus P$.
Then $(P + aR)_* = (P+bR)_* = R$ by assumption. Hence, by Lemma~\ref{star},
\begin{eqnarray*}
R &=& (P + aR)_*(P + bR)_* \\ &\subseteq& ((P + aR)_*(P + bR)_*)_* = ((P + aR)(P + bR))_* \\ &\subseteq& (P+abR)_* \subseteq R \,,
\end{eqnarray*}
so $(P+abR)_* = R$. Thus, $ab \not\in P$, which shows that $P$ is a prime ideal.
\end{proof}

For a prime ideal $P$ of $R$ and an ideal $I$ of $R$, let
$R_{(P)} =\{ \frac{a}{b} \mid a \in R$ and $b \in \reg(R \setminus P)\}$,
 $R_{[P]} =\{ z \in T(R) \mid zs \in R$ for some $s \in R \setminus P\}$, and
 $[I]R_{[P]} =\{ z \in T(R) \mid zs \in I$ for some $s \in R \setminus P\}$.
Then $R_{(P)}$ and $R_{[P]}$ are overrings of $R$, $R_{(P)} \subseteq R_{[P]}$,
$[I]R_{[P]}$ is an ideal of $R_{[P]}$, and if $I$ is a prime
ideal of $R$ with $I \subseteq P$, $[I]R_{[P]}$ is also a prime ideal of $R_{[P]}$.
The ring $R_{[P]}$ is called the large quotient ring of $R$ with respect to $P$,
which was introduced by Griffin \cite{m70} in order to study the Pr\"ufer rings with zero divisors.
Obviously, $P \subseteq Z(R)$ if and only if $R_{[P]} = R_{(P)} = T(R)$.
It is also clear that if $R$ is an integral domain,
then $R_P = R_{(P)} = R_{[P]}$. Moreover,
if $R$ is a Marot ring, which is a ring whose regular
ideals are generated by their subsets of regular elements, then $R_{[P]}= R_{(P)}$  and $[P]R_{[P]}= PR_{(P)}$ \cite[Theorem 7.6]{4}.

\begin{corollary}
Let $*$ be a star operation of finite type on a ring $R$
and $\Lambda = \{P \in *$-$\Max(R) \mid P$ is regular$\}$.
Then $R = \bigcap_{P \in *\textnormal{-} \hspace{-2pt}\Max(R)}R_{[P]} = \bigcap_{P \in \Lambda}R_{[P]}
= \bigcap_{P \in *\textnormal{-} \hspace{-2pt}\Max(R)}R_{(P)} = \bigcap_{P \in \Lambda}R_{(P)}$.
\end{corollary}

\begin{proof}
It is clear that $R \subseteq \bigcap_{P \in *\textnormal{-} \hspace{-2pt}\Max(R)}R_{(P)} \subseteq \bigcap_{P \in \Lambda}R_{(P)}
 \subseteq \bigcap_{P \in \Lambda}R_{[P]}$
and $R \subseteq \bigcap_{P \in *\textnormal{-} \hspace{-2pt}\Max(R)}R_{[P]}
 \subseteq \bigcap_{P \in \Lambda}R_{[P]}$, so it suffices to show that $\bigcap_{P \in \Lambda}R_{[P]} \subseteq R$.
For $z \in  \bigcap_{P \in \Lambda}R_{[P]}$, let $I = \{a \in R \mid az \in R\}$.
Then $I$ is a regular ideal of $R$ and $I \nsubseteq P$ for all $P \in \Lambda$. Hence,
$I_* = R$ by Proposition \ref{prop1.2}(2). Now, note that $zI \subseteq R$, so $zI_* \subseteq (zI)_* \subseteq R_* = R$, and hence
$I_* = I$. Thus, $z \in R$.
\end{proof}

Let $*_1$ and $*_2$ be star operations on $R$.
We say that {\it $*_1 \le *_2$} if $I_{*_1} \subseteq I_{*_2}$ for all $I \in \mathsf K (R)$.
It is clear that $*_1 \le *_2$ if and only if $(I_{*_1})_{*_2} = I_{*_2}$,
if and only if $(I_{*_2})_{*_1} = I_{*_2}$ for all $I \in \mathsf K (R)$,
i.e., $*_2$-ideals are $*_1$-ideals.
A ring is called a {\it $*$-Noetherian ring}
if it satisfies the ascending chain condition on its integral $*$-ideals.
It is routine to check that $R$ is $*$-Noetherian if and only if every $*$-ideal of $R$ is of finite type, and in this case, $* = *_f$.
Clearly, if $*_1 \le *_2$, then $*_1$-Noetherian rings are $*_2$-Noetherian.
Hence, a Noetherian ring $R$ is $*$-Noetherian for all star operations $*$ on $R$.

\begin{remark}
(1)  Krull introduced the concept of a star operation,
which was denoted by $'$-operation, on an integral domain \cite[page 118]{k35}.
Gilmer used the notation $*$-operation for the first time \cite[Section 32]{3}. After Gilmer's book \cite{3},
the star operation has become an indispensable means of studying the
ideal factorization properties of commutative rings.

(2) The star operation of this paper is called a (unital) semistar operation in \cite[Definition 2.4.1]{e19}
and the star operation of \cite[Definition 2.4.32]{e19} is just a function from $F^{\reg} (R)$ into $F^{\reg} (R)$ satisfying the properties (1) - (4) above.
Moreover, if $R$ is an integral domain, then $F^{\reg} (R)$ is just the set of nonzero fractional ideals of $R$ and the ``classical" star operation on $R$ is a function from $F^{\reg} (R)$ into itself \cite[Section 32]{3}.

(3) Okabe and Matsuda, in \cite{om94}, introduced the notion of a semistar operation on an integral domain $R$, which is a function $*$ from $\mathsf K_0 (R):=\mathsf K (R) \setminus \{(0)\}$ into itself such that
(i) $(aI)_* = aI_*$, (ii) $I \subseteq I_*$, and $I \subseteq J$ implies that $I_* \subseteq J_*$,
and (iii) $(I_*)_* = I_*$ for all $I, J \in \mathsf K_0 (R)$ and $0 \neq a \in T(R)$. Then, in \cite{fl03}, Fontana and Loper said that $*$ is a (semi)star operation if $R_*=R$, in order to emphasize the fact that
$*$ is an extension of a star operation on $R$ to $\mathsf{K}_0(R)$.
\end{remark}

\subsection{The $d$-, $v$- and $t$-operations}

For $I \in \mathsf K (R)$, let $I^{-1} = (R :_{T(R)} I) = \{ x \in T(R) \mid xI \subseteq R \}$,
then $I^{-1} \in \mathsf K (R)$.
Now, the $d$-, $v$-, and $t$-operation are the star operations on $R$ defined by
\[
  I_d = I \,, \quad I_v = (I^{-1})^{-1} \mbox{ for all } I \in \mathsf K (R) \,, \mbox{ and } \,\, t = v_f \,.
\]
It is well known and easy to see that $d \le *_f \le *$, $*_f \le t \le v$, and $* \le v$ for any star operation $*$ on $R$ \cite[Proposition 2.4.10]{e19}.

Assume that $R$ is an integral domain.
Then $(0)$ is a prime ideal of $R$,
\[
  (0)_v = (0) \quad \Leftrightarrow \quad R \neq T(R) \quad \Leftrightarrow \quad (0)_t = (0) \,,
\]
if $I \in \mathsf K (R)$ and $R \subsetneq T(R)$, then $I_v \subsetneq T(R)$ if and only if $I$ is a fractional ideal of $R$,
and $* \leq v$ for any star operation $*$ on $R$.
By this fact, it is reasonable that the star operation
on an integral domain $R$ is defined on the set of nonzero fractional ideals of $R$.

\begin{remark}
The idea of localization comes from algebraic geometry. The localization is a crucial tool for studying varieties locally near a point $p$, and so it allows us to only focus on rational functions that are well-defined at the point $p$. A star operation is a similar tool for studying commutative rings in the sense that we are just interested in ideals that we certainly have in mind. For instance, in Krull domains, every nonzero proper principal ideal is a unique finite $v$-product of height-one prime ideals, whence it is enough to look into the height-one prime ideals.
\end{remark}

\subsection{The $*$-invertibility}

Let $*$ be a star operation on a ring $R$.
An $I \in \mathsf K(R)$ is said to be {\it invertible} if $II^{-1} = R$.
Clearly, if $I$ is invertible, then $I$ is a finitely generated fractional ideal of $R$ \cite[Theorem 7.1]{3}.
As the $*$-operation analog, $I \in K(R)$ is said to be {\it $*$-invertible} if $(II^{-1})_* = R$.
Hence, if $*$ is of finite type, then $I$ is $*$-invertible
if and only if $II^{-1} \nsubseteq P$ for all $P \in *$-Max$(R)$ by Proposition \ref{prop1.2}(2).
It is well known that an invertible ideal is regular \cite[Theorem 7.1]{3},
while a $*$-invertible ideal need not be regular (see, for example, Example~\ref{ex0.1}).

\begin{proposition}\label{prop1.5} {\em  (cf. \cite[Propositions 2.6, 2.8(3) and Lemma 3.17]{k89-1} for an integral domain)}
If $*$ is a star operation of finite type, then
every $*$-invertible Kaplansky fractional $*$-ideal is of finite type
and a $t$-invertible $t$-ideal.
\end{proposition}

\begin{proof}
Let $I$ be a $*$-invertible Kaplansky fractional $*$-ideal of a ring $R$. Then
$$I_t = I_t (II^{-1})_* \subseteq ( I_t (II^{-1})_{*} )_{*} = ( I(I_t I^{-1}))_{*} \subseteq I_* = I \subseteq I_t,$$
and thus $I_t = I$. Next, note that $R = (II^{-1})_* \subseteq (II^{-1})_t \subseteq R$. Thus, $(II^{-1})_t = R$.
Finally, since $*$ is of finite type, there is a finitely generated ideal $J$ of $R$ such that
 $J \subseteq I$ and $(JA)_* = R$ for some Kaplansky fractional ideal $A$ of $R$ with $A \subseteq I^{-1}$.
 Hence, $I = I (JA)_* \subseteq (I(JA)_* )_* = (J (IA))_* \subseteq J_* \subseteq I_* = I$.
 Thus, $I = J_*$.
\end{proof}

It is known that if $P$ is a $t$-invertible prime $t$-ideal of an
integral domain $D$, then $P$ is a maximal $t$-ideal of $D$ \cite[Proposition 1.3]{hz89}.
The next corollary shows that this is true for rings with zero divisors.
The proof of Corollary \ref{coro1.6} is similar to that of \cite[Proposition 1.3]{hz89}.

\begin{corollary} \label{coro1.6}
If $*$ is a star operation of finite type, then every $*$-invertible prime
$*$-ideal is a maximal $t$-ideal.
\end{corollary}

\begin{proof}
Let $P$ be a $*$-invertible prime $*$-ideal of a ring $R$.
Then,  by Proposition \ref{prop1.5}, $P$ is a $t$-invertible $t$-ideal and $P = I_t$ for some
finitely generated ideal $I$ of $R$.
Choose $a \in R \setminus P$ and $x \in (I+aR)^{-1}$. Then
$xaI \subseteq I \Rightarrow xaP = xaI_t \subseteq (xaI)_t \subseteq I_t = P,$
and since $P$ is a prime ideal, $xP \subseteq P$. Hence,
$$x \in xR = x(PP^{-1})_t \subseteq (xPP^{-1})_t \subseteq (PP^{-1})_t = R.$$
Thus, $(I+aR)^{-1} = R$, which implies $(I+aR)_v = (I+aR)_t = R$.
Therefore, $P$ is a maximal $t$-ideal.
\end{proof}

Let $\Inv (R)$ be the group of invertible fractional ideals of $R$ under the usual multiplication
and $\Prin (R)$ be its subgroup of regular principal fractional ideals of $R$.
Then the factor group $\Pic (R) = \Inv (R) / \Prin (R)$ is called the {\it Picard group} (or {\it ideal class group}) of $R$.
Now, let $\Inv^* (R)$ (resp., $\Inv_* (R)$) be the group of $*$-invertible Kaplansky fractional
(resp., $*$-invertible regular fractional) $*$-ideals of $R$ under $I \times J = (IJ)_*$. Then
\[
  \Prin (R) \subseteq \Inv (R) \subseteq \Inv_* (R) \subseteq \Inv^* (R) \,,
\]
so if we let $\Cl_* (R) = \Inv_* (R) / \Prin (R)$ (resp., $\Cl^* (R) = \Inv^* (R) / \Prin (R)$)
be the factor group of $\Inv_* (R)$ (resp., $\Inv^* (R)$) modulo $\Prin (R)$, then
\[
  \Pic (R) \subseteq \Cl_* (R) \subseteq \Cl^* (R)
\]
(cf. \cite[pp. 109-110]{e19}).
We say that $\Cl_* (R)$ is the {\it $*$-class group} of $R$ and $\Cl^* (R)$ is the
{\it semistar $*$-class group} of $R$ as in \cite[Definition 2.5.21]{e19}.
By Proposition \ref{prop1.5}, $\Cl_* (R) \subseteq \Cl_t (R)$ and $\Cl^* (R) \subseteq \Cl^t (R)$.
The $t$-class group $\Cl_t (R)$ is usually denoted by $\Cl (R)$ and called the {\it class group} of $R$.
If $R$ is a Krull domain, then $\Cl (R)$ is the usual divisor class group of $R$.

For more on basic properties of star operations, the reader can refer
to \cite[Chapter 2]{e19}, \cite[Chapter 3]{kk14}, \cite[Section 32]{3}, and \cite[Lemma 3.3]{hp80}.

\subsection{Terminology}
Let $\h P$ denote the height of a prime ideal $P$, $\dim (R)$ be the Krull dimension of $R$,
$\rdim (R)$ be the regular dimension of $R$,
and Max$(R)$ be the set of maximal ideals of $R$.
Let $N(R)$ be the nilradical of $R$, i.e., $N(R) = \{x \in R \mid x^n = 0$ for some $n \geq 1\}$.
Then $N(R)$ is the intersection of all prime ideals of $R$, $N(R) \subseteq Z(R)$,
and $N(R) = \{0\}$ if and only if $R$ is reduced.

A ring $R$ has few zero divisors if $Z(R)$ is a finite union of prime ideals,
$R$ is said to be {\it additively regular} if for each $z \in T(R)$, there is an $x \in R$
such that $z+x$ is regular, and $R$ is a {\it Marot ring} if each regular ideal of $R$ is generated by a set of regular elements in $R$.
Then a Noetherian ring has few zero divisors,
a ring with few zero divisors is additively regular, and an additively regular
ring is a Marot ring \cite[Theorem 7.2]{4}, but the reverse implications do not hold in general. In particular,
if dim$(T(R)) = 0$, then $R$ is additively regular \cite[Theorem 7.4]{4}.

A ring $R$ satisfies {\it Property(A)} if each finitely generated ideal $I \subseteq Z (R)$ has a nonzero annihilator.
Then $R$ has Property(A) if and only if $T(R)$ has Property(A) \cite[Corollary 2.6]{4}.
The class of rings with Property(A) includes Noetherian rings \cite[Theorem 82]{8},
the polynomial ring \cite[Corollary 2.9]{4}, and integral domains.
A ring $R$ is an {\it r-Noetherian ring} if each regular ideal of $R$ is finitely generated.
Clearly, Noetherian rings are r-Noetherian, r-Noetherian rings need not be Noetherian
(see, for example, \cite[Section 5]{ck21-1}), and the polynomial ring $R[X]$ is r-Noetherian if
and only if $R$ is Noetherian.

\section{Historical overview} \label{2}

A principal ideal domain (PID) is an integral domain all of whose ideals are principal;
a unique factorization domain (UFD) is an integral domain in which each (nonzero)
nonunit can be written as a finite product of prime elements; and
a Dedekind domain (resp., $\pi$-domain) is an integral domain
whose (nonzero) ideals (resp., principal ideals)
can be written as a finite product of prime ideals.

\subsection{Krull domains}

Let $D$ be an integral domain with quotient field $K$ and $X^1(D)$ be the set of nonzero minimal
(i.e., height-one) prime ideals of $D$. Recall from \cite[$\S$ 2]{n55} that $D$ is a Krull domain if
\begin{enumerate}
\item $D = \bigcap_{P \in X^1(D)}D_P$,

\item $D_P$ is a DVR for all $P \in X^1(D)$, and

\item each nonzero nonunit of $D$ is contained in only a finitely many
prime ideals in $X^1(D)$.
\end{enumerate}
In \cite[Proposition 3]{n55},
Nagata proved that $D$ is a Krull domain if and only if there exists a family $\Delta$
of DVRs with quotient field $K$ such that (i) $D$ is the intersection of all rings in $\Delta$
and (ii) every nonzero element of $D$ is a unit in all but a finite number of rings in $\Delta$.

The theory of Krull domains was originated by Krull \cite{k31, k31-1}.
Also, in \cite[43. $V$-ideale]{k35}, Krull stated (without proof) that $D$ is a Krull domain if and only if
each $v$-ideal $I$ of $D$ is a unique finite $v$-product of height-one prime ideals of $D$,
i.e., $I = (P_1^{e_1} \cdots P_n^{e_n})_v$ for some distinct height-one prime ideals $P_1,
\dots , P_n$ and positive integers $e_1, \dots, e_n$ such that the expression $I = (P_1^{e_1} \cdots P_n^{e_n})_v$ is unique. Then,
in \cite{n63}, Nishimura showed that $D$ is a Krull domain if and only if
each $v$-ideal of $D$ is a unique finite $v$-product of height-one prime ideals of $D$,
if and only if $D$ is a completely integrally closed $v$-Noetherian domain (i.e., Mori domain).
In 1968, Tramel showed that $D$ is a Krull domain if and only if each nonzero proper principal ideal of $D$
can be written as a finite $v$-product of prime ideals \cite[Theorem 3.1]{t68}, which
also shows that the uniqueness of Nishimura's result is superfluous. In \cite[Theorem]{n73},
Nishimura showed that $D$ is a Krull domain if and only if
each $t$-ideal of $D$ is a finite $t$-product of height-one prime ideals of $D$
under the additional condition that the expression $I = (P_1^{e_1} \cdots P_n^{e_n})_t$ is unique.
Then, by \cite[Theorem 1.1.]{l72}, $D$ is a Krull domain if and only if
each nonzero proper principal ideal of $D$
can be written as a finite $t$-product of prime ideals, if and only if
each nonzero $t$-ideal of $D$ is a finite $t$-product of height-one prime ideals of $D$.
In \cite[Theorem 3.5]{k89}, Kang showed that if each nonzero prime ideal of $D$ contains a
$t$-invertible prime ideal, then $D$ is a Krull domain.

\subsection{PIR, UFR, general ZPI-ring, and $\pi$-ring}

A ring $R$ is said to be a {\it special primary ring} (SPR) or a {\it special principal ideal ring}
(SPIR) if $R$ is a local ring with maximal ideal $M$
such that each proper ideal of $R$ is a power of $M$, equivalently, $M$ is principal and $M^n = (0)$ for some
integer $n \geq 1$. It is worth noting that a field is a Krull domain and an SPR.

A {\it principal ideal ring} (PIR) is a ring all of whose ideals are principal.
The structure theorem for PIRs says that (i) a finite direct product of PIRs is a PIR
and (ii) $R$ is a PIR if and only if $R$ is a finite direct product of PIDs and SPRs \cite[Theorem 33, page 245]{zs60}.
Moreover, Hungerford showed that each direct summand of a PIR is the homomorphic image of a PID \cite[Theorem 1]{h68}.

A {\it general ZPI-ring} is a ring in which each ideal can be expressed as a finite product of prime ideals.
The letters ZPI stand for Zerlegung Primideale. Mori first studied the general ZPI-ring \cite{m40-1}
and Asano proved that $R$ is a general ZPI-ring if and only if $R$ is a finite direct product of Dedekind domains
and SPRs \cite[Satz 21]{a51}.
For more properties of general ZPI-rings, see \cite[Section 39]{3} or \cite[Chapter 9]{10}.

In \cite{f69}, Fletcher introduced the notion of a {\it unique factorization ring} (UFR),
which is just a UFD in case of integral domains.
He showed that a finite direct product of UFRs is a UFR, a PIR is a UFR \cite[Theorem 4]{f69},
and a UFR is a finite direct product of UFDs and SPRs \cite[Theorem 19]{f70}.
Later, in \cite[Theorem 4]{f71}, Fletcher proved that $R$ is a UFR
if and only if every element of $R$ is a finite product of prime elements of $R$.

A {\it $\pi$-ring} is a ring in which each principal ideal can be written as a finite product of prime ideals.
Mori gave a complete description of $\pi$-domains in \cite{m39},
where it was noted that every nonzero minimal prime ideal of a $\pi$-domain is invertible.
Mori also proved that $R$ is a $\pi$-ring if and only if $R$ is a finite direct product of $\pi$-domains and SPRs \cite{m40}.
A proof of Mori's result may also be found in \cite[Section 46]{3}.
Mori showed that the polynomial ring $R[ \{ X_{\alpha} \} ]$ over a $\pi$-domain $R$ is also a $\pi$-domain \cite{m40-1}.
As in the case of integral domains, we have the following implications except the counter part of Krull domains;

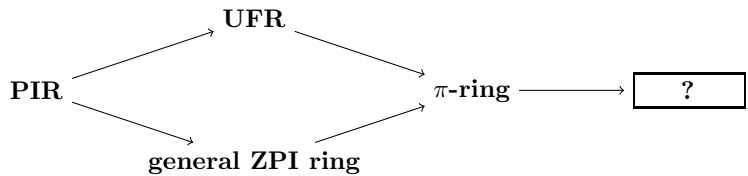
\begin{figure}[h]
\resizebox{.8\textwidth}{!}{%
\begin{tikzpicture}[node distance=2cm]
\title{Untergruppenverband der}
\node(PIR)    at (0,0)   {\bf PIR};
\node(UFR)    at (3,1)   {\bf UFR};
\node(ZPI)    at (3,-1)  {\bf general ZPI ring};
\node(PiR)    at (6,0)   {\bf $\pi$-ring};
\node(?)      at (9,0)   {\bf $\framebox{ \quad ? \quad}$};

\draw[->]  (PIR)  -- (UFR);
\draw[->]  (PIR)  -- (ZPI);
\draw[->]  (UFR)  -- (PiR);
\draw[->]  (ZPI)  -- (PiR);
\draw[->]  (PiR)  -- (?);
\end{tikzpicture}
}
\end{figure}

\subsection{regular PIR, factorial ring, Dedekind ring, and regular $\pi$-ring}

A {\it regular PIR} (resp., {\it Dedekind ring}) is a ring in which each regular ideal is principal (resp., invertible).
An invertible ideal is finitely generated, so a regular PIR and a Dedekind ring
are both r-Noetherian rings.
It is known that a ring $R$ is a Dedekind ring if and only if $R/I$ is a PIR for all regular ideals $I$ of $R$ \cite[Proposition 7.9]{m85},
if and only if each regular ideal of $R$ is a (unique) finite product of prime ideals \cite[Theorem 10]{m63}.
We say that $R$ is a multiplication ring if, given any two ideals $A$ and $B$ of $R$
with $A \subseteq B$, there exists an ideal $Q$ such that $A = BQ$.
Clearly, an integral domain is a multiplication ring if and only if it
is a Dedekind domain \cite[Proposition 9.13]{10}.
Moreover, $R$ is a multiplication ring if and only if $T(R)$ is a multiplication ring, $R$ is a Dedekind ring,
and non-maximal prime ideals of $R$ are idempotent \cite[Theorem 13]{g74}.

A ring $R$ is called a {\it factorial ring} if every nonunit regular element of $R$ can be written as a finite product of prime elements.
Then the polynomial ring $R[X]$ is a factorial ring if and only if $R[X]$ is a UFR,
if and only if $R$ is a finite direct product of UFDs \cite[Corollary 5.8]{aam85}.
A ring is said to be a {\it regular $\pi$-ring} if every regular principal ideal is a finite product of prime ideals.
It is known that $R$ is a regular $\pi$-ring if and only if
every regular prime ideal of $R$ contains an invertible prime ideal \cite[Lemma 2]{k91}.

\subsection{Rank-one discrete valuation rings}

A valuation on a ring $R$ is a mapping $v$ from $R$ onto a totally
ordered abelian group with $\infty$ adjoined such that
(i) $v(ab) = v(a) + v(b)$ and (ii) $v(a+b) \geq \min\{v(a), v(b)\}$ for all $a,b \in R$.
It is clear that if $v$ is a valuation on $R$, then $v(1) = 0$ and $v(0) = \infty$.
Moreover, if $R_v = \{x \in R \mid v(x) \geq 0\}$ and $P_v = \{x \in R \mid v(x) > 0\}$,
then $R_v$ is a subring of $R$, $P_v$ is a prime ideal of $R_v$, and $(R_v, P_v)$
is called a valuation pair of $R$ \cite[page 193]{m69}.
The valuation $v$ on $R$ was first studied by Manis \cite{m69} when $R$ is a ring with zero divisors.

Let $\mathbb{Z}$ be the additive group of integers.
Then a valuation $v$ from $T(R)$ onto $\mathbb{Z} \cup \{\infty\}$ is called a {\it rank-one discrete valuation} on $T(R)$,
and in this case, if $V = \{x \in T(R) \mid v(x) \geq 0\}$   and  $P = \{x \in T(R) \mid v(x) > 0\}$,
then $(V, P)$ is called a {\it rank-one discrete valuation ring} (rank-one DVR).
It is easy to see that if $(V, P)$ is a rank-one DVR, then $V = V_{[P]}$ and $PV_P$ is principal.
Moreover, if $PV_P = pV_P$, then $[P]V_{[P]} = [pV]V_{[P]}$. Also,
if $I$ is an ideal of $V$ with $P \subsetneq I$, then $[I]V_{[P]} = V_{[P]}$.
It should be noted that if $T(R)$ is a field, then $P$
is the maximal ideal of $V$, but this is not true in general
(see, for example, \cite[Example 5.4]{ck21-1} or \cite[Example 7 on page 182]{4}),
which also shows that a rank-one DVR need not be a regular PIR.

\subsection{Krull rings with zero divisors}

A ring $R$ is a {\em Krull ring} if there exists a family
$\{(V_{\alpha}, P_{\alpha}) \mid \alpha \in \Lambda\}$ of rank-one DVRs such that (i) $R =\bigcap_{\alpha \in \Lambda}
V_{\alpha}$, (ii) for each $a \in \reg(R)$, $aV_{\alpha} = V_{\alpha}$ for almost all $\alpha \in \Lambda$,
and (iii) $P_{\alpha}$ is a regular ideal for all $\alpha \in \Lambda$ \cite[page 132]{9}.
In particular, if $\{V_{\alpha} \} = \emptyset$, then $R= T(R)$, so $T(R)$
is regarded as a Krull ring. If $R =\bigcap_{\alpha \in \Lambda} V_{\alpha}$ is irredundant, then
$\{(V_{\alpha}, P_{\alpha}) \mid \alpha \in \Lambda\}$ is said to be a defining family for a Krull ring $R$.
Let $X^1_r(R)$ be the set of minimal regular prime ideals of $R$.
Then, as in the case of Krull domains \cite[Corollary 43.9]{3}, $\{(R_{[P]}, [P]R_{[P]}) \mid P \in X^1_r(R)\}$
is the unique defining family for a Krull ring $R$, which was proved by Potelli and Spangher \cite[Proposition 40]{12}
in the Marot Krull ring case and by Alajbegovi\'c and Osmanagi\'c \cite[Propositions 2.10 and 2.11]{ao90} and
Chang and Kang \cite[Theorem 4.3]{ck21} in the general case.

Krull rings with zero divisors were first introduced by Marot \cite[page 27]{ma68} under the condition that the rings are Marot.
A complete generalization of Krull domains to rings with zero divisors
was done by Kennedy \cite{k73, 9} and Huckaba \cite[Definition 2.1]{h76} respectively.
Kennedy showed that if $R$ is a Krull ring, then $R$ is completely integrally closed and
$R$ satisfies the ascending chain condition on integral regular $v$-ideals \cite[Proposition 2.2]{9}.
Matsuda proved that the converse of Kennedy's result is also true \cite[Theorem 5]{m82} as in the case of integral domains.
Huckaba proved that the integral closure of a Noetherian ring is a Krull ring \cite[Theorem 2.3]{h76}.
Portelli and Spangher also studied the notion of Krull rings, which is just a Marot Krull ring \cite{12}.
In \cite{k91}, Kang gave some characterizations of a Marot Krull ring.
Later, in \cite{7}, Kang completely extended his results of \cite{k91}
to rings with zero divisors without the Marot condition. Among them, he
showed that $R$ is a Krull ring if and only if
every regular prime ideal of $R$ contains a $t$-invertible regular prime ideal,
if and only if every proper regular principal ideal of $R$ can be written as a finite $t$-product of prime ideals.
Then a regular $\pi$-ring is a Krull ring, so we have the following implications;

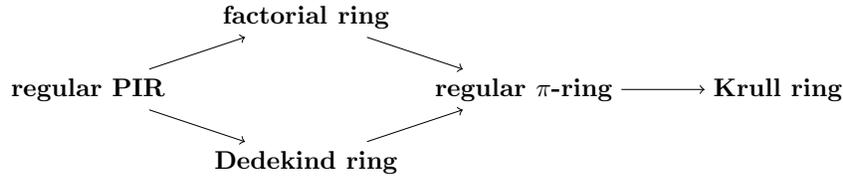
\begin{figure}[h]
\resizebox{.9\textwidth}{!}{%
\begin{tikzpicture}[node distance=2cm]
\title{Untergruppenverband der}
\node(rPIR)         at (0,0)   {\bf regular PIR};
\node(factorial)    at (3,1)   {\bf factorial ring};
\node(Dr)           at (3,-1)  {\bf Dedekind ring};
\node(rPiR)         at (6,0)   {\bf regular $\pi$-ring};
\node(Krull)        at (9.5,0)   {\bf Krull ring};

\draw[->]  (rPIR)       -- (factorial);
\draw[->]  (rPIR)       -- (Dr);
\draw[->]  (factorial)  -- (rPiR);
\draw[->]  (Dr)         -- (rPiR);
\draw[->]  (rPiR)       -- (Krull);
\end{tikzpicture}
}
\end{figure}

The next result shows when a Krull ring is a regular PIR, a factorial ring, a Dedekind ring,
or a regular $\pi$-ring. We recall that (i) $\rdim (R) = 0$ if and only if $R = T(R)$; (ii)
$\rdim (R) = 1$ if and only if $R \neq T(R)$ and each regular prime ideal of $R$ is a maximal ideal; and
(iii) if $R$ is an integral domain, then $\dim (R) = \rdim (R)$.

\begin{theorem} \label{Krull ring property}
Let $R$ be a Krull ring, which is not a total quotient ring, i.e., $R \subsetneq T(R)$.
Then the following statements hold.
\begin{enumerate}

\item $R$ is a regular $\pi$-ring if and only if $\Cl (R) = \Pic (R)$.

\item $R$ is a factorial ring if and only if $\Cl (R) = \{ 0 \}$.

\item $R$ is a Dedekind ring if and only if $\rdim (R) = 1$.

\item $R$ is a regular PIR if and only if $\rdim (R) = 1$ and $\Cl (R) = \{ 0 \}$.
\end{enumerate}
\end{theorem}

\begin{proof}
(1) $\Cl (R) = \Pic (R)$ means that each regular $t$-invertible $t$-ideal of $R$ is invertible.
Thus, the result follows directly from the definition of regular $\pi$-rings.

(2) \cite[Theorem]{am95-2} (cf. \cite[Proposition 7.10]{m85} for the Marot Krull ring).

(3) It is clear that if $R$ is a Dedekind ring, then $\rdim (R) = 1$ \cite[Remark (7) on page 24]{m63}.
Conversely, assume that $\rdim (R) = 1$.
Then each regular prime $t$-ideal of $R$ is a maximal ideal,
and hence every regular ideal of $R$ is invertible. Thus, $R$ is a Dedekind ring.

(4) This follows from (1) and \cite[page 112]{e19}.
\end{proof}

It is worthwhile to note that the prime factorization of regular ideals
in Krull rings, regular PIRs, Dedekind rings, factorial rings, and regular $\pi$-rings
is unique. However, the prime factorization in PIRs, general ZPI-rings, UFRs,
and $\pi$-rings need not be unique (for example, if $R = \mathbb{Z}/4\mathbb{Z}$
and $P = 2\mathbb{Z}/4\mathbb{Z}$, then $R$ is an SPR and $P^2 = P^{2+n}$ for all integers $n \geq 1$).

\subsection{Nagata rings}

Let $X$ be an indeterminate over a ring $R$ and $R[X]$ be the
polynomial ring over $R$. For $f \in R[X]$, let $c(f)$ denote the ideal of $R$
generated by the coefficients of $f$. McCoy's theorem says that
$f \in R[X]$ is a zero divisor if and only if $af = 0$ for some $0 \neq a \in R$.
Hence, $f \in R[X]$ is regular if and only if $c(f)$ is semiregular,
i.e., $xc(f) = (0)$ implies that $x = 0$ for all $x \in R$.
It is clear that if $R$ satisfies Property(A), then
$f \in R[X]$ is regular if and only if $c(f)$ is regular, hence
$I \subseteq Z(R)$ if and only if $IR[X] \subseteq Z(R[X])$.

If $f,g \in R[X]$, then
$c(f)^{n+1}c(g) = c(f)^nc(fg)$ for some integer $n \geq 1$ by Dedekind-Mertens lemma
\cite[Corollary 28.3]{3}.
Hence, if $S = \{f \in R[X] \mid c(f) = R\}$, then $S$ is a (regular)
saturated multiplicative set of $R[X]$, and hence $R[X]_S$
is an overring of $R[X]$ \cite[page 18]{n62}.
The ring $R[X]_S$ is denoted by $R(X)$ and called the {\it Nagata ring} of $R$.
In \cite{a69}, Arnold related the ideal theory of $R(X)$ to that of $R$.
For example, he showed that $R$ is a Pr\"ufer domain (resp.,
Dedekind domain) if and only if $R(X)$ is a Pr\"ufer domain (resp., PID).
Anderson showed that every finitely generated locally principal ideal of $R(X)$ is principal \cite[Theorem 2]{a77},
and hence Pic$(R(X)) = \{0\}$.

Now, let $R$ be an integral domain and $N_v = \{f \in R[X] \mid f \neq 0$ and $c(f)_v = R\}$. Then
$N_v$ is a saturated multiplicative set of $R[X]$ and $R(X) \subseteq R[X]_{N_v}$,
which is called the {\it $t$-Nagata ring} of $R$.
Gilmer showed that $R$ is a Krull domain if and only if $R[X]_{N_v}$ is a PID,
and if $R$ is a P$v$MD, then $R[X]_{N_v}$ is a Bezout domain \cite{g70}.
($R$ is a Pr\"ufer $v$-multiplication domain (P$v$MD) if every nonzero
finitely generated ideal of $R$ is $t$-invertible.) The $t$-Nagata ring of $R$
was generalized via an arbitrary star-operation $*$ on $R$ by Kang \cite{k89-1}, i.e.,
he studied the ring $R[X]_{N_*}$, where $N_* = \{f \in R[X] \mid f \neq 0$ and $c(f)_* = R\}$,
and, among other things, he showed that $\Pic (R[X]_{N_*}) = \{0\}$
and $\Cl (R[X]_{N_v}) = \{0\}$.
Huckaba and Papick considered the $t$-Nagata ring in the case of a ring with zero divisors
under the additional condition that it is an additively regular ring with Property(A) \cite{hp80}.
Among other things, they showed that $R$ is a P$v$MR (i.e., each finitely generated regular ideal
is $t$-invertible) if and only if $R[X]_{U_2}$ is a Pr\"ufer ring,
where $U_2 = \{f \in R[X] \mid c(f)_v = R$ and $f$ is regular$\}$ \cite[Theorem 3.6]{hp80}.
In \cite{7}, Kang studied two generalizations of the $t$-Nagata ring of an
integral domain to a ring $R$ with zero divisors, i.e., $N_v(R) = \{f \in R[X] \mid c(f)_v = R\}$
and $N_v^r(R) = \{f \in R[X] \mid c(f)$ is regular and $c(f)_v = R\}$.

\section{The $w$-operation and the finite direct product of rings} \label{w-oper}

In this section, we review the $w$-operation on integral domains and its
two generalizations to rings with zero divisors, which motivate
the new star operation of this paper for the prime factorization property
of general Krull rings. We also study some basic properties of the finite
direct product of rings.

\subsection{The $w$-operation on integral domains}

Let $R$ be an integral domain.
A finitely generated ideal $I$ of $R$ is called a {\it GV-ideal} if $I^{-1} = R$,
where $\GV$ stands for Glaz and Vasconcelos,
and we denote by $\GV (R)$ the set of all $\GV$-ideals of $R$.
In \cite{gv77}, Glaz and Vasconcelos said that a torsion free $R$-module $M$ is {\it semi-divisorial} if $M = \{ a \in M \otimes_R T(R) \mid aJ \subseteq M \mbox{ for some } J \in \GV (R) \}$.
The $w$-operation on $R$ is a star operation defined by
\[
  I_w = \{ x \in T(R) \mid xJ \subseteq I \mbox{ for some } J \in \GV (R) \}
\]
for all $I \in \mathsf K (R)$. Then $w$ is  of finite type, $w \leq t$,
$t$-$\Max(R) = w$-$\Max(R)$, and $I_w = \bigcap_{P \in t\textnormal{-} \hspace{-2pt}\Max(R)}IR_P$ for
all $I \in F(R)$ \cite[Corollaries 2.13 and 2.17]{ac00}.
The $w$-operation was introduced by Hedstrom and Houston \cite{hh80}
and it has begun to be studied extensively after Wang and McCasland's paper \cite{wm97},
where the notation $w$ was first used, was published.

 The $w$-operation is very useful when we study the ideal theory of integral domains.
 For example, it is known that (i) $R$ is a $w$-Noetherian domain if and only if $R[X]$ is a
 $w$-Noetherian domain \cite[Theorem 1.13]{wm99}, if and only if $R[X]_{N_v}$ is a Noetherian domain
 (\cite[Corollary 2.8]{w02} and \cite[Theorem 2.2]{c05});
 (ii) $R$ is a Krull domain if and only if every nonzero (prime) ideal of $R$ is
 $w$-invertible \cite[Theorem 5.4]{wm97}, if and only if
 $R$ is an integrally closed $w$-Noetherian domain \cite[Theorem 2.8]{wm99};
 (iii) $R$ is a P$v$MD if and only if $R$ is integrally closed and $t= w$ on $R$ \cite[Theorem 3.5]{k89-1};
 and (iv) if $Q$ is a nonzero primary ideal of $R$, then
 $Q_w = Q$ if and only if $Q_w \subsetneq R$ \cite[Proposition 1.1]{wm99}.

\subsection{Two generalizations of the $w$-operation to rings with zero divisors}

The $w$-operation has been generalized to rings with zero divisors by two ways.
The first one was by Yin, Wang, Zhu, and Chen \cite{ywzc11} as follows:
Let $J$ be a finitely generated ideal of $R$.
Then $J$ is called a {\it GV-ideal}, denoted by $J \in \textnormal{GV}(R)$,
if the homomorphism $\varphi \colon R \to \Hom_R (J,R)$ given by $\varphi(r)(a) = ra$ is an isomorphism \cite[Definition 1.1]{ywzc11}.
Clearly, if $J$ is regular, then $\Hom_R (J,R) = J^{-1}$, so $\varphi$ is an isomorphism if and only if $J^{-1} = R$.
The $w$-operation on a ring $R$ with zero divisors is a star operation of finite type defined by
\[
  A_w = \{ x \in T(R) \mid xJ \subseteq A \mbox{ for some } J \in \GV (R) \}
\]
for all $A \in \mathsf K (R)$ \cite[Section 3]{ywzc11}.

Let $\Lambda = \{ J \in \mathsf K (R) \mid J \mbox{ is  finitely generated and } J^{-1} = R \}$.
Later, in \cite[Definition 2.4.21]{e19}, Elliott introduced another type of a $w$-operation as follows; $A_{w'} = \{ x \in T(R) \mid xJ \subseteq A \mbox{ for some } J \in \Lambda \}$ for all $A \in \mathsf K (R)$.

An ideal $I$ of $R$ is {\it semiregular} if $I$ contains a finitely generated ideal $I_0$ which has no nonzero annihilator,
i.e., $xI_0 = (0)$ implies that $x = 0$ for all $x \in R$.
Clearly, a regular ideal is semiregular. Also, a $\GV$-ideal is semiregular,
and hence the $w$-operation is reduced (i.e., $(0)_w = (0)$).
Let $S^{0}$ be the set of finitely generated semiregular ideals of $R$, $R[X]$ be the polynomial ring over $R$, and
\[
  Q_0 (R) = \{ u \in T(R[X]) \mid Iu \subseteq R \mbox{ for some } I \in S^{0} \} \,.
\]
Then $Q_0 (R)$, called the {\it ring of finite fractions}, is a commutative ring with identity; $T(R) \subseteq Q_0 (R) \subseteq T(R[X])$.
Moreover, if $R$ satisfies Property(A) (i.e., each semiregular ideal of $R$ is regular),
 then $Q_0 (R) = T(R)$ \cite{l93}.
It is known that a finitely generated semiregular ideal $J$ of $R$ is a $\GV$-ideal if and only if $(R :_{Q_0 (R)} J) = R$ \cite[Lemma 4.7]{wk15}.
Note that $(R :_{Q_0 (R)} J) = R$ implies that $(R :_{T(R)} J)= R$, so if $J \in \GV (R)$, then $J^{-1} = R$.
Hence $A_w \subseteq A_{w'}$ for all $A \in \mathsf K(R)$.
Clearly, if $R$ is an integral domain but not a field, then $w = w'$ on $R$.
But, $w \neq w'$ in general (see Example~\ref{ex2.7}).

Let $N_w = \{f \in R[X] \mid c(f)_w =R\}$. Then $N_w$ is a saturated multiplicative set of
$R[X]$ such that each element of $N_w$ is a regular element of $R[X]$
\cite[Lemma 3.2 and Proposition 3.3(2)]{wk15},
and hence $R[X]_{N_w}$ is an overring of $R[X]$. Moreover,
each ideal of $R[X]_{N_w}$ is a $w$-ideal \cite[Theorem 6.6.8]{wk16},
Max$(R[X]_{N_w}) = \{P[X]_{N_w} \mid P \in w$-Max$(R)\}$ \cite[Proposition 3.3(4)]{wk15},
and $R$ is a $w$-Noetherian ring if and only if $R[X]$ is $w$-Noetherian,
if and only if $R[X]_{N_w}$ is Noetherian \cite[Theorem 6.6.8]{wk16}.

\subsection{The direct product of finitely many rings}

Let $R_1$ and $R_2$ be rings. Then the direct product of $R_1$ and $R_2$,
denoted by $R_1 \times R_2$, is a commutative ring with identity.
All of the results in this subsection are stated for $R_1 \times R_2$.
However, by mathematical induction, it is easy to see that
the results of this section hold for the direct product of finitely many rings.

\begin{lemma} \label{direct product}
Let $R = R_1 \times R_2$ be the direct product of rings $R_1$ and $R_2$.
Then the following conditions are satisfied.
\begin{enumerate}
\item $T(R) = T(R_1) \times T(R_2)$.

\item Every $R$-submodule of $T(R)$ is of the form $I_1 \times I_2$ for
an $R_i$-submodule $I_i$ of $T(R_i)$, $i =1, 2$.

\item $Q$ is a prime (resp., primary) ideal of $R$ if and only if $Q = P_1 \times R_2$ or $Q = R_1 \times P_2$ for some prime (resp., primary) ideal $P_i$ of $R_i$, $i=1,2$, and in this case, if $Q$ is a prime ideal, then $R_Q = (R_i)_{P_i}$.

\item $\dim (R) = \max \{ \dim (R_1), \dim (R_2) \}$.

\item $\reg(R) = \reg(R_1) \times \reg(R_2)$. Thus,
$Z(R) = \big( Z(R_1) \times R_2 \big) \cup \big( R_1 \times Z(R_2) \big)$.

\item $N(R) = N(R_1) \times N(R_2)$.
\end{enumerate}
\end{lemma}

\begin{proof}
This is well known and an easy exercise.
\end{proof}

The next proposition can be applied so that the ideals of $R_1 \times R_2$ with
prescribed properties can be constructed by using the ideals of $R_1$ and $R_2$.

\begin{proposition} \label{star direct product}
Let $R = R_1 \times R_2$ be the direct product of rings $R_1$ and $R_2$,
$J_i$ be an $R_i$-submodule of $T(R_i)$ for $i =1, 2$, and $J = J_1 \times J_2$.
Then the following conditions are satisfied.
\begin{enumerate}
\item $J$ is fractional (resp., regular, finitely generated, principal)
if and only if both $J_1$ and $J_2$ are fractional (resp., regular, finitely generated, principal).

\item $J^{-1} = J_1^{-1} \times J_2^{-1}$.

\item $J_v = (J_1)_v \times (J_2)_v$.

\item $J_t = (J_1)_t \times (J_2)_t$.

\item $J_{w'} = (J_1)_{w'} \times (J_2)_{w'}$.

\item $J \in \GV (R)$ if and only if $J_i \in \GV (R_i)$ for $i = 1, 2$.

\item $J_w = (J_1)_w \times (J_2)_w$.

\item $J$ is invertible (resp., $v$-invertible, $t$-invertible, $w'$-invertible, $w$-invertible)
 if and only if $J_i$ is invertible (resp., $v$-invertible, $t$-invertible, $w'$-invertible, $w$-invertible) for $i = 1, 2$.
\end{enumerate}
\end{proposition}

\begin{proof}
(1), (2), (3), (4), and (5) \cite[Exercise 9 on page 169]{e19}.

(6) \cite[Proposition 1.2]{ywzc11}.

(7) Let $x = (a_1, a_2) \in R$. Then $x \in J_w$ if and only if $xA \subseteq J$ for some $A \in \GV (R)$,
 if and only if $(a_1, a_2) \big( A_1 \times A_2 \big) \subseteq J_1 \times J_2$ for some $A_i \in \GV (R_i)$, $i=1, 2$,
 by (6) and Lemma~\ref{direct product}, if and only if $a_i A_i \subseteq J_i$ for some $A_i \in \GV (R_i)$,
 if and only if $(a_1, a_2) \in (J_1)_w \times (J_2)_w$.
Thus, $J_w = (J_1)_w \times (J_2)_w$.

(8) This follows directly from (2), (3), (4), (5), and (7).
\end{proof}

\begin{corollary} \label{coro3.3}
Let $R = R_1 \times R_2$ be the direct product of rings $R_1$ and $R_2$. Then
the following conditions are satisfied.
\begin{enumerate}
\item $\Pic (R) = \Pic (R_1) \times \Pic (R_2)$.

\item $\Cl_* (R) = \Cl_* (R_1) \times \Cl_* (R_2)$ for $* = v, t, w', w$.
\end{enumerate}
\end{corollary}

\begin{proof}
This follows directly from Proposition~\ref{star direct product}(1) and (8).
\end{proof}

The following corollary seems to be well known.
But, we couldn't find it in the literature, so we include it.

\begin{corollary} \label{direct product Property(A)}
Let $R = R_1 \times R_2$ be the direct product of rings $R_1$ and $R_2$.
Then $R$ satisfies Property(A) if and only if both $R_1$ and $R_2$ satisfy Property(A).
\end{corollary}

\begin{proof}
($\Rightarrow$) For a finitely generated
ideal $J_1$ of $R_1$ with $J_1 \subseteq Z (R_1)$, let $J = J_1 \times R_2$.
Then $J$ is finitely generated and $J \subseteq Z (R)$.
Hence, by assumption, there is a nonzero element $(a,b) \in R$ such that $(a,b)J = \langle (0,0) \rangle$.
Note that $1 \in R_2$, so $b = 0$. Hence, $(a,b) \neq (0,0)$ implies that $a \neq 0$.
Note also that $aJ_1 = (0)$. Thus, $R_1$ satisfies Property(A).
Similarly, $R_2$ also satisfies Property(A).

($\Leftarrow$) Let $J = J_1 \times J_2$ be a finitely generated ideal of $R$ with $J \subseteq Z (R)$.
Then $J_i$ is finitely generated for $i=1,2$ by Proposition~\ref{star direct product}(1) and
either $J_1 \subseteq Z (R_1)$ or $J_2 \subseteq Z (R_2)$ by Lemma \ref{direct product}(5).
For convenience, assume that $J_1 \subseteq Z (R_1)$.
Then there is an element $0 \neq a \in R_1$ such that $aJ_1 = (0)$.
Hence $(a,0) \in R$, $(a,0) \neq (0,0)$, and $(a,0)J = \langle (0,0) \rangle$.
Thus, $R$ satisfies Property(A).
\end{proof}

\begin{corollary}
Let $R$ be a PIR, a general ZPI-ring, an UFR, or a $\pi$-ring.
Then $R$ has few zero divisors and satisfies Property(A).
\end{corollary}

\begin{proof}
By the structure theorem, $R$ is a finite direct product of integral domains
and SPRs. Hence, $R$ has few zero divisors by Lemma \ref{direct product}(3) and (5).
Note that both integral domains and SPRs satisfy Property(A).
Thus, $R$ satisfies Property(A) by Corollary \ref{direct product Property(A)}.
\end{proof}

We end this section with an example of rings $R$ that are finite
direct product of rings with a unique minimal prime ideal, and in this case,
each prime ideal of $R$ contains a unique minimal prime ideal.

\begin{example}
(1) Let $R$ be an integrally closed ring such that $T(R)$ is a zero dimensional Noetherian ring.
Then the zero ideal of $T(R)$ has an irredundant primary decomposition,
say, $(0) = Q_1 \cap \cdots \cap Q_n$ in $T(R)$, and each pair of $Q_i$'s is comaximal.
Hence, $T(R) = T(R)/Q_1 \times \cdots \times  T(R)/Q_n$ and
$$R \hookrightarrow R/(Q_1 \cap R) \times \cdots \times  R/(Q_n \cap R) \hookrightarrow T(R).$$
Note that $R$ is integrally closed and $R/(Q_1 \cap R) \times \cdots \times  R/(Q_n \cap R)$ is a finitely generated $R$-module, so
$$R = R/(Q_1 \cap R) \times \cdots \times  R/(Q_n \cap R).$$
Note also that each $Q_i \cap R$ is a primary ideal of $R$. Thus, each $R/(Q_i \cap R)$
has a unique minimal prime ideal. Moreover, each prime ideal of $R$ contains a unique minimal prime ideal by
Lemma \ref{direct product}(3).

(2) Let $R$ be an integrally closed reduced ring. By the proof of (1) above,
$R$ has a finite number of
minimal prime ideals if and only if $R$ is a finite direct product of integral domains
\cite[Lemma 8.14]{4}.
In particular, if $R$ is an integrally closed reduced Noetherian ring,
then $R$ is a finite direct product of integrally closed Noetherian domains.
\end{example}

\section{A new star operation} \label{3}

Let $R$ be a ring with total quotient ring $T(R)$.
In this section, we introduce a new star operation $u$ on $R$, which is
an extension of the $w$-operation on integral domains to rings with zero divisors,
i.e., $u=w$ on an integral domain, and we study basic properties of the $u$-operation.
Now, let
\[
  \rGV (R) = \{ J \in F^{\reg} (R) \mid J \mbox{ is finitely generated and } J^{-1} = R \} \,.
\]
Then $\rGV (R)$ is a multiplicative set of regular ideals
of $R$, $R \in \rGV (R)$, and $\rGV (R) = \GV (R) \cap F^{\reg} (R)$.

\begin{proposition} \label{u-oper}
Let $\mathsf K (R)$ be the set of Kaplansky fractional ideals of a ring $R$.
For each $I \in \mathsf K (R)$, let
\[
  I_u = \{ x \in T(R) \mid xJ \subseteq I \mbox{ for some } J \in \rGV(R)\} \,.
\]
Then the following conditions hold for all $a \in T(R)$ and $I,J \in \mathsf K (R)$:
\begin{enumerate}
\item $I_u \in \mathsf K (R)$.

\item $aI_u \subseteq (aI)_u$, and equality holds when $a$ is regular.

\item $I \subseteq I_u$, and $I \subseteq J$ implies $I_u \subseteq J_u$.

\item $(I_u)_u = I_u$.

\item $R_u = R$.

\item $(0)_u = (0)$.

\item $I_u = \bigcup\{(I_0)_u \mid I_0 \in \mathsf K (R)$, $I_0 \subseteq I$, and $I_0$ is finitely generated$\}$.
\end{enumerate}
Thus, the map $u \colon \mathsf K (R) \to \mathsf K (R)$, given by $I \mapsto I_u$, is a reduced star operation of finite type on $R$.
\end{proposition}

\begin{proof}
(1) It is a routine to check that $I_u \in \mathsf K (R)$, so the proof is omitted.

(2) If $x \in I_u$, there exists a $J \in \rGV (R)$ such that $xJ \subseteq I$, so
$axJ \subseteq aI$, and hence $ax \in (aI)_u$.
Thus, $aI_u \subseteq (aI)_u$.
Moreover, if $a$ is regular, then
\[
  I_u = \frac{1}{a} (aI_u) \subseteq \frac{1}{a}(aI)_u \subseteq (\frac{1}{a} (aI))_u = I_u \,,
\]
and so $\frac{1}{a}(aI)_u = I_u$. Thus, $aI_u = (aI)_u$.

(3) Clear.

(4) Let $x \in (I_u)_u$. Then $xJ \subseteq I_u$ for some $J \in \rGV (R)$.
Note that $J$ is finitely generated and $\rGV (R)$ is a multiplicative set of ideals.
Hence, there is a $J_1 \in \rGV (R)$ such that $xJJ_1 \subseteq I$. Thus, $JJ_1 \in \rGV (R)$ implies $x \in I_u$.
Therefore, $(I_u)_u \subseteq I_u$, so by (3), $(I_u)_u = I_u$.

(5) If $x \in R_u$, then $xJ \subseteq R$ for some $J \in \rGV (R)$.
Hence, $x \in xR = xJ_v \subseteq (xJ)_v \subseteq R_v = R$. Thus,
$R_u \subseteq R$, and hence $R_u = R$ by (3).

(6) Let $x \in (0)_u$. Then $xJ \subseteq (0)$ for some $J \in \rGV (R)$,
and since $J$ is regular, $x = 0$. Thus, $(0)_u \subseteq (0)$, so $(0)_u = (0)$.

(7) Let $x \in I_u$. Then $xJ \subseteq I$ for some $J \in \rGV (R)$.
Note that $xJ$ is finitely generated, so there is a finitely generated subideal $I_0$ of $I$
such that $xJ \subseteq I_0$. Hence, $x \in (I_0)_u$. This implies
$I_u = \bigcup\{(I_0)_u \mid I_0 \in \mathsf K (R)$ is finitely generated and $I_0 \subseteq I\}$.
Thus, $u$ is of finite type.
\end{proof}

As in \cite[Definition 2.4.16]{e19},
a star operation $*$ on a ring $R$ is said to be {\it stable} if
\begin{itemize}
\item $(I \cap J)_* = I_* \cap J_*$ for all $I, J \in \mathsf K (R)$.

\item $(I :_{T(R)} J)_* = (I_* :_{T(R)} J)$ for all $I, J \in \mathsf K (R)$ with $J$ finitely generated.
\end{itemize}
It is known that $*$ is stable if and only if $I_* = \{x \in T(R) \mid xJ \subseteq I$ for some $J \in \mathsf K (R)$
with $J_* = R\}$ for all $I \in \mathsf K (R)$ \cite[Proposition 2.4.23]{e19}

\begin{corollary} \label{u-stable}
The $u$-operation on a ring $R$ is stable.
\end{corollary}

\begin{proof}
Let $A, B \in \mathsf K (R)$.

Claim 1. $( A \cap B)_u = A_u \cap B_u$. ({\em Proof.}  Clearly, $( A \cap B)_u \subseteq A_u \cap B_u$.
For the reverse containment, let $x \in A_u \cap B_u$.
Then there exist some $J_1, J_2 \in \rGV (R)$ such that $xJ_1 \subseteq A$ and $xJ_2 \subseteq B$.
Hence, $xJ_1 J_2 \subseteq A \cap B$, and since $J_1 J_2 \in \rGV (R)$, $x \in (A \cap B)_u$.
Thus, $A_u \cap B_u \subseteq (A \cap B)_u$.)

Claim 2. $(A_u :_{T(R)} B) = (A :_{T(R)} B)_u$ for $B$ finitely generated.
({\em Proof.}  Let $x \in (A :_{T(R)} B)_u$.
Then $xJ \subseteq (A :_{T(R)} B)$ for some $J \in \rGV (R)$. Hence, $xJB \subseteq A$, so
$xJB \subseteq (xJB)_u \subseteq A_u$, which implies $xB \subseteq (A_u)_u = A_u$, so that $x \in (A_u :_{T(R)} B)$.
Thus, $(A :_{T(R)} B)_u \subseteq (A_u :_{T(R)} B)$.
Conversely, let $z \in (A_u :_{T(R)} B)$.
Then $zB \subseteq A_u$, and since $B$ is finitely generated, there exists a $J \in \rGV (R)$ such that $zJB \subseteq A$.
Hence, $zJ \subseteq (A :_{T(R)} B)$, which means that $z \in (A :_{T(R)} B)_u$.
Thus, $(A_u :_{T(R)} B) \subseteq (A :_{T(R)} B)_u$.)
\end{proof}

A star operation $*$ on a ring $R$ is said to be {\it spectral} if
there exists a set $\Lambda$ of prime ideals of $R$ such that
$I_* = \underset{P \in \Lambda}{\bigcap} [I] R_{[P]}$ for all $I \in \mathsf K (R)$
\cite[Definition 3.11.2]{e19}.
The next corollary shows that the $u$-operation on a ring $R$ is spectral.

\begin{corollary}
$I_u = \underset{P \in u \textnormal{-} \hspace{-2pt}\Max (R)}{\bigcap} [I] R_{[P]}$ for all $I \in \mathsf K (R)$.
\end{corollary}

\begin{proof}
Let $x \in I_u$. Then $xJ \subseteq I$ for some $J \in \rGV (R)$.
Note that $J_u = R$, so $J \nsubseteq P$ for all $P \in u$-$\Max (R)$.
Thus, $x \in \underset{P \in u \textnormal{-} \hspace{-2pt}\Max (R)}{\bigcap} [I] R_{[P]}$.
Conversely, assume that $z \in \underset{P \in u \textnormal{-} \hspace{-2pt}\Max (R)}{\bigcap} [I] R_{[P]}$, and let $A = \{ a \in R \mid az \in I \}$.
Then $A \nsubseteq P$ for all $P \in u$-$\Max (R)$, so $A_u = R$,
which implies that $J \subseteq A$ for some $J \in \rGV (R)$.
Hence, $zJ \subseteq zA \subseteq I$, so that $z \in I_u$.
Thus, the proof is completed.
\end{proof}

\begin{remark}
Let $\Lambda$ be a multiplicative set of ideals generated by a set of finitely generated ideals $J$ of a ring $R$ with $J^{-1} = R$.
For $I \in \mathsf K (R)$, let
\[
  I_{*_{\Lambda}} = \{ x \in T(R) \mid xJ \subseteq I \mbox{ for some } J \in \Lambda \} \,.
\]
Then the following statements hold.
\begin{enumerate}
\item $*_{\Lambda}$ is a star operation of finite type on $R$.

\item Each ideal in $\Lambda$ is semiregular if and only if $*_{\Lambda}$ is reduced.

\item Let $N_{*_{\Lambda}} = \{ f \in R[X] \mid c(f)_{*_{\Lambda}} = R \}$.
Then each element in $N_{*_{\Lambda}}$ is a regular element of $R[X]$ if and only if each ideal in $\Lambda$ is semiregular.

\item $*_{\Lambda}$ is stable and spectral.
\end{enumerate}
The proof is the same as that of the $u$-operation case.
Moreover, if $\Lambda = \{ R \}$, then $*_{\Lambda} = d$, if $\Lambda = \rGV (R)$, then $*_{\Lambda} = u$,
if $\Lambda = \GV (R)$, then $*_{\Lambda} = w$, and if $\Lambda$ is the set of all finitely generated ideals $J$ of $R$ with $J^{-1} = R$, then $*_{\Lambda} = w'$.
\end{remark}

It is obvious that if $R$ is an integral domain that is not a field,
then every nonzero ideal of $R$ is regular, hence $u = w = w'$ on $R$.
We next investigate the exact relationship of $u$, $w$, $w'$, and $t$,
which is useful in the subsequent arguments.

\begin{proposition} \label{u-w-w'-oper}
Let $w$ and $w'$ be the star operations on a ring $R$ as in Section~$\ref{w-oper}$.
\begin{enumerate}
\item $I_u \subseteq I_w \subseteq I_{w'}$ for all $I \in \mathsf K (R)$,
and equalities hold when $I$ is regular.
\item If $R$ satisfies Property(A), then $u=w$ on $R$.
\end{enumerate}
\end{proposition}

\begin{proof}
(1) Let $\Lambda = \{ J \in \mathsf K (R) \mid J \mbox{ is finitely generated and } J^{-1} = R \}$.
Then, by definition, $\rGV (R) \subseteq \GV (R) \subseteq \Lambda$, so $I_u \subseteq I_w \subseteq I_{w'}$.

Suppose that $I$ is regular. We have only to show that $I_{w'} \subseteq I_u$.
Since $I$ is regular, $I$ contains a regular element, say, $a \in R$.
Choose $x \in I_{w'}$. Then there is a finitely generated ideal $J'$ of $R$ such that $x J' \subseteq I$ and $(J')_v = R$.
Now, let $d \in \reg(R)$ be such that $dx \in R$
and $J = J' + daR$. Then $J$ is finitely generated, regular, $J_v = R$, and $xJ \subseteq I$. Hence, $x \in I_u$.
Thus, $I_{w'} \subseteq I_u$.

(2) Let $J \in \GV (R)$. Then $J$ is semiregular, and since $R$ satisfies Property(A), $J$ is regular.
Hence $J \in \rGV (R)$. Thus $\GV (R) = \rGV (R)$, which implies that $u = w$ on $R$.
\end{proof}

Let $I$ be a regular ideal of $R$. Then $I \subseteq II^{-1} \subseteq R$,
so $II^{-1}$ is a regular ideal of $R$.
Hence, the next corollary also shows that $I$ is $u$-invertible if and only if $I$ is $t$-invertible
by Proposition \ref{prop1.2}.

\begin{corollary} \label{u-maxiaml}
Let $P$ be a regular ideal of a ring $R$.
Then $P$ is a maximal $t$-ideal if and only if $P$ is a maximal $u$-ideal.
\end{corollary}

\begin{proof}
This follows from Proposition~\ref{u-w-w'-oper}(1) and \cite[Corollary 3.14]{ywzc11},
which states that a regular ideal is a maximal $t$-ideal if and only if it is a maximal $w$-ideal.
\end{proof}

\begin{corollary} \label{coro4.6}
$\Cl_t (R) = \Cl_w (R) = \Cl_{w'} (R) = \Cl_u (R)$.
\end{corollary}

\begin{proof}
Let $I$ be a regular fractional ideal of $R$.
Then $II^{-1}$ is regular, and hence $(II^{-1})_u = (II^{-1})_w = (II^{-1})_{w'}$
by Proposition~\ref{u-w-w'-oper}(1).
Thus, $\Cl_u (R) = \Cl_w (R) = \Cl_{w'} (R)$.
Furthermore, $(II^{-1})_u = R$ if and only if $(II^{-1})_t = R$ by Corollary ~\ref{u-maxiaml}.
Thus, $\Cl_u (R) = \Cl_t (R)$.
\end{proof}

A Krull ring is characterized by the $t$- or $v$-operation, i.e.,
$R$ is a Krull ring if and only if each regular
principal ideal of $R$ is a finite $t$-product (or $v$-product) of prime ideals.
The following corollary shows that a Krull ring can be characterized by the $u$, $w$, or $w'$-operation.

\begin{corollary} \label{coro4.7}
Let $* = u, t, v, w$, or $w'$ and $X^1_r(R)$ be the set of minimal
regular prime ideals of a ring $R$. Then
the following statements are equivalent.
\begin{enumerate}
\item $R$ is a Krull ring.

\item $I_*$ is a finite $*$-product of prime ideals for all regular ideals $I$ of $R$.

\item Every regular principal ideal of $R$ is a finite $*$-product of prime ideals.

\item Every regular ideal of $R$ is $*$-invertible.
\end{enumerate}
In this case, $X^1_r(R) = t$-$\Max (R) \cap F^{\reg} (R)$ and
$I_u = I_v = I_t = I_w = I_{w'}$ for all $I \in F^{\reg} (R)$.
\end{corollary}

\begin{proof}
For $* = v$ or $t$, see \cite[Theorem 13]{7}.
For $* = u, w$, or $w'$, we may assume that $* = u$ by Proposition~\ref{u-w-w'-oper}.

(1) $\Rightarrow$ (2) Let $I$ be a regular ideal of $R$.
Then $I$ is $t$-invertible and $I_t = (P_1 \cdots P_n)_t$ for some prime ideals $P_1, \ldots, P_n$ of $R$ \cite[Theorem 13]{7},
so $I$ and $P_1 \cdots P_n$ are $u$-invertible by Corollary~\ref{u-maxiaml}, and hence
$$I_u = (I_u)_t = I_t = (P_1 \cdots P_n)_t = ((P_1 \cdots P_n)_u)_t =(P_1 \cdots P_n)_u$$
by Proposition~\ref{prop1.5}. Thus, $I_u = (P_1 \cdots P_n)_u$.

(2) $\Rightarrow$ (3) Clear.

(3) $\Rightarrow$ (1) It is known that a regular principal ideal is a $t$-ideal and $u \le t$.
Hence, every regular principal ideal of $R$ is a finite $t$-product of prime ideals by (3).
Thus, $R$ is a Krull ring \cite[Theorem 13]{7}.

(1) $\Leftrightarrow$ (4)  Recall from \cite[Theorem 13]{7} that $R$ is a Krull ring if and only if every regular
ideal of $R$ is $t$-invertible. Thus, the result follows directly from Proposition \ref{u-w-w'-oper}
and Corollary \ref{u-maxiaml}.

\vspace{.12cm}
Moreover, $P \in X^1_r(R)$ if and only if $P$ is a regular $t$-invertible
prime $t$-ideal of $R$ \cite[Lemma 3.2]{ck21}, if and only if $P$ is a
regular maximal $t$-ideal by Corollary \ref{coro1.6} and \cite[Theorem 13]{7}.
\end{proof}

As in \cite{d62}, we say that an ideal $I$ of a ring $R$ is a {\it $Z$-ideal} if $I \subseteq Z (R)$.
Hence, $I$ is a $Z$-ideal if and only if $I$ is not regular.

\begin{proposition} \label{regLocal}
Let $S = \reg (R)$ for a ring $R$, so $T(R) = R_S$, and $I$
be an ideal of $R$. Then the following conditions are satisfied.
\begin{enumerate}
\item  $I_u R_S = IR_S$.

\item If $IR_S \cap R = I$, then $I_u = I$. Hence, every prime $Z$-ideal is a $u$-ideal.

\item If $Q$ is a primary ideal of $R$ such that $Q_u \subsetneq R$, then $Q_u = Q$.

\item $u = d$ on $T(R)$, i.e., every ideal of $T(R)$ is a $u$-ideal.
\end{enumerate}
\end{proposition}

\begin{proof}
(1) If $x \in I_u$, then there is a $J \in \rGV(R)$ such that $xJ \subseteq I$, so
$$x \in xR_S = xJR_S \subseteq IR_S.$$
Hence, $I_u R_S \subseteq IR_S$. The reverse containment is clear.
Thus, $I_uR_S = IR_S$.

(2) By (1), $I_u \subseteq I_u R_S \cap R = IR_S \cap R = I$. Thus, $I_u = I$.

(3) Clearly, $Q \subseteq Q_u$. For the reverse containment,
let $x \in Q_u$. Then $xJ \subseteq Q$ for some $J \in rGV(R)$.
Note that $Q_u \subsetneq R$ by assumption, so $J \nsubseteq \sqrt{Q}$, and
since $Q$ is primary, $x \in Q$.
Thus, $Q_u \subseteq Q$.

(4) Note that $\rGV ( T(R) ) = \{ T(R) \}$, so if $I$ is an ideal of $T(R)$, then $x \in I_u$
if and only if $xT(R) \subseteq I$, if and only if $x \in I$.
Thus, $I_u = I$.
\end{proof}

\begin{corollary} \label{T(R)_Noetherian}
If $R$ is a $u$-Noetherian ring, then $T(R)$ is Noetherian.
\end{corollary}

\begin{proof}
Let $S = \reg (R)$, so $T(R) = R_S$ and every ideal of $T(R)$ is of the form $IR_S$ for some ideal $I$ of $R$.
Now, let $I$ be a $Z$-ideal of $R$.
Then, by assumption, $I_u = J_u$ for some finitely generated ideal $J$ of $R$, whence $IR_S = I_u R_S = J_u R_S = JR_S$ by Proposition~\ref{regLocal}(1). Thus, $R_S$ is a Noetherian ring.
\end{proof}

\begin{corollary} \label{u-inv regular}
A fractional ideal $I$ of $R$ is $u$-invertible if and only if $I$ is a $t$-invertible regular ideal.
\end{corollary}

\begin{proof}
Without loss of generality, we may assume that $I$ is an ideal of $R$, i.e., $I \subseteq R$.
Furthermore, by the remark before Corollary \ref{u-maxiaml},
it suffices to show that if $I$ is $u$-invertible, then $I$ is regular.
Suppose that $I$ is $u$-invertible, and let
$S = \reg (R)$, so $T(R) = R_S$.
Then, by Proposition~\ref{regLocal}(1),
 $$R_S = (II^{-1})_u R_S = (II^{-1})R_S = (IR_S)(I^{-1}R_S) = IR_S,$$
  where the last equality follows by $R \subseteq I^{-1} \subseteq R_S$.
Hence, $IR_S = R_S$, which implies that $I \cap S \neq \emptyset$. Thus, $I$ is regular.
\end{proof}

We end this section with a construction of a ring on which $u \lneq w \lneq w'$ by a series of examples.
For the construction of such a ring, we first need the following lemma,
which is the $u$-operation analog of Proposition \ref{star direct product}(3)-(7).

\begin{lemma} \label{u-diect product}
Let $R = R_1 \times R_2$ be the direct product of rings $R_1$ and $R_2$, $J_i$ be
an $R_i$-submodule of $T(R_i)$ for $i = 1, 2$, and $J = J_1 \times J_2$.
\begin{enumerate}
\item $J \in \rGV (R)$ if and only if $J_i \in \rGV (R_i)$ for $i = 1, 2$.

\item $J_u = (J_1)_u \times (J_2)_u$.

\item $J$ is $u$-invertible if and only if $J_i$ is $u$-invertible for $i = 1, 2$.
\end{enumerate}
\end{lemma}

\begin{proof}
(1) This follows directly from Proposition~\ref{star direct product}(1) and (2).

(2) This is an immediate consequence of (1) (cf. the proof of Proposition~\ref{star direct product}(7)).

(3) $J$ is $u$-invertible if and only if $(JJ^{-1})_u = R$, if and only if $(J_i J^{-1}_i)_u = R_i$ for $i = 1, 2$
by (2) and Proposition \ref{star direct product}(2), if and only if each $J_i$ is $u$-invertible.
\end{proof}

The first construction is a ring on which $u \lneq w$.
This example also shows that semiregular ideals need not be regular.

\begin{example} \label{ex2.6}
Let $D$ be a UFD with $\dim (D) \ge 2$, $X^{1} (D)$ be the set of height-one prime ideals of $D$,
and $M = \underset{P \in X^{1} (D)}{\bigoplus} T(D/P)$.
Then $M$ is a $D$-module defined by $\alpha(\overline{x_p}) = (\overline{\alpha x_p})$ for $\alpha \in D$,
where $\overline{x_p} = x_p + P \in T(D/P)$ for all $P \in X^{1} (D)$.
Let $R = D (+) M$ be the idealization of $M$ in $D$, i.e., $R$ is a ring under the usual addition of $R \times M$
and the multiplication defined by $(a, x)(b, y) = (ab, ay + bx)$ for all
$(a, x), (b, y) \in R \times M$ \cite[Section 25]{4}.
Note that if $P$ is a nonzero prime ideal of $D$, then $P^{-1} \supsetneq D$ if and only if $P \in X^{1} (D)$.
Note also that $\bigcup_{P \in X^{1} (D)} P$ is the set of nonunits of $D$,
so $Z (R) = \{ (a,m) \in R \mid a \in P$ for some $P \in X^{1} (D) \}$ \cite[Lemma 4.1(1)]{zw15},
and thus $T(R)= R$.
Now, let $A$ be a nonzero finitely generated proper ideal of $D$ with $A^{-1} = D$
(such an ideal $A$ exists because $D$ is a UFD with dim$(D) \geq 2$).
Then $A(+)M \in$ GV$(R)$ \cite[Lemma 4.1(5)]{zw15}, so
$(A(+)M)_w = R$. However, since rGV$(R) = \{R\}$, $(A(+)M)_u = A(+)M \subsetneq R$.
Therefore, $u \lneq w$.
\end{example}

\begin{example} \label{ex2.7}
Let $R = \mathbb Z \times \mathbb Q$ and $I = d \mathbb Z \times \{0\}$ for a nonzero nonunit $d \in \mathbb Z$.
Then $I_{w'} = d \mathbb Z \times \mathbb Q \supsetneq I = I_u = I_w$.
Thus, $u = w \lneq w'$ on $R$.
\end{example}

\begin{proof}
Clearly, $R$ is a PIR, so if $J$ is a finitely generated semiregular ideal of $R$, then $J$ is principal, whence $J$ is regular.
Thus, $u = w$ on $R$.
Moreover, if $J \in \rGV (R)$, then $J = \langle (1,1) \rangle = R$, and hence $I_u = I$.

Now, note that $(\mathbb Z \times \{ 0 \} )_v = \langle (1,0) \rangle_v = R$.
Hence, if $J$ is a finitely generated ideal of $R$ with $J_v = R$, then $J = \mathbb Z \times \{ 0 \}$ or $J = R$.
Note also that $(a,q)(1,0) \in I$ if and only if $a \in d\mathbb Z$ and $q \in \mathbb Q$.
Thus, $I_{w'} = d\mathbb Z \times \mathbb Q$. Hence, $u = w \lneq w'$ on $R$.
\end{proof}

We now state a ring on which $u \lneq w \lneq w'$.

\begin{example}
Let $R_1$ be the ring of Example \ref{ex2.6},
$R_2$ be the ring of Example \ref{ex2.7}, and $R= R_1 \times R_2$
be the direct product of $R_1$ and $R_2$.
For $i =1,2$, let $J_i$ be the ideal of $R_i$ in Examples \ref{ex2.6} and \ref{ex2.7}
respectively and $J = J_1 \times J_2$. Then $J$ is an ideal of $R$ and
$(J_1)_u \times (J_2)_u \subsetneq (J_1)_w \times (J_2)_w \subsetneq
(J_1)_{w'} \times (J_2)_{w'}$ by Lemma~\ref{u-diect product}
and Proposition \ref{star direct product}, whence
$J_u \subsetneq J_w \subsetneq J_{w'}$. Thus, $u \lneq w \lneq w'$ on $R$.
\end{example}

\section{General Krull rings} \label{4}

Inspired by the name of general ZPI-rings, we will say that $R$ is a {\it general Krull ring}
if $R$ is a finite direct product of Krull domains and SPRs.
Then a general Krull ring satisfies Property(A) by Corollary~\ref{direct product Property(A)}
and it is additively regular \cite[Theorems 7.2 and 7.4]{4}.
Moreover, since a finite direct product of Krull rings is a Krull ring \cite[page 134]{9},
a general Krull ring is a Krull ring.
In this section, we develop the theory of general Krull rings with a focus on the prime factorization of ideals via the $u$-operation that is introduced in Section 4.

Let $R$ be a ring, $T(R)$ be the total quotient ring of $R$,
$X^1(R)$ be the set of height-one prime ideals of $R$, and
$X_r^1(R)$ be the set of minimal regular prime ideals of $R$.
It is clear that if $P \in X^1(R)$ is regular, then $P \in X^1_r(R)$.

\subsection{Basic properties of a general Krull ring}

A $\pi$-domain is a Krull domain, so $\pi$-rings are general Krull rings.
 Our first result summarizes the exact relationship of PIRs, general ZPI-rings, UFRs, $\pi$-rings, and
general Krull rings.

\begin{theorem} \label{chracter-cg}
Let $R$ be a general Krull ring, which is not a total quotient ring, i.e., $R \subsetneq T(R)$.
Then the following statements hold.
\begin{enumerate}
\item $R$ is a PIR if and only if $R$ is a UFR with $\dim (R) = 1$.

\item $R$ is a general ZPI-ring if and only if $\dim (R) = 1$.

\item $R$ is a UFR if and only if $\Cl (R) = \{ 0 \}$.

\item $R$ is a $\pi$-ring if and only if $\Cl (R) = \Pic (R)$.
\end{enumerate}
\end{theorem}

\begin{proof}
Let $R = R_1 \times \cdots \times R_n$ for some Krull domains and SPRs $R_1, \ldots, R_n$.
It is clear that if $R_i$ is an SPR, then $\dim (R_i) = 0$ and $\Cl (R_i) = \{ 0 \}$.
Note that

\vspace{.1cm}

$\cdot$ $\dim (R) = \max \{ \dim (R_i) \mid i = 1, \ldots, n \}$ by Lemma~\ref{direct product}(4),

\vspace{.1cm}
$\cdot$ $\Cl (R) = \Cl (R_1) \times \cdots \times \Cl (R_n)$, and

\vspace{.1cm}
$\cdot$ $\Pic (R) = \Pic (R_1) \times \cdots \times \Pic (R_n)$ by Corollary~\ref{coro3.3}.

\vspace{.1cm}
\noindent
Thus, the result follows directly from the following observation: If $R$ is a Krull domain,
then (1) $R$ is a PID if and only if $R$ is a UFD with $\dim (R) = 1$,
(2) $R$ is a Dedekind domain if and only if $\dim (R) = 1$ (cf. \cite[Theorem 37.8]{3}),
(3) $R$ is a UFD if and only if $\Cl (R) = \{ 0 \}$ \cite[Proposition 6.1]{f73}, and
(4) $R$ is a $\pi$-domain if and only if $\Cl (R) = \Pic (R)$ (cf. \cite[Theorem 46.7]{3}). 
\end{proof}

Let $R$ be a general Krull ring and $P$ be a prime ideal of $R$. Then
$R_P$ is either a Krull domain or an SPR by Lemma \ref{direct product}(3)
and $P$ contains a unique minimal prime ideal.
We next give some basic properties of general Krull rings.

\begin{theorem} \label{thm5.2}
Let $R$ be a general Krull ring, which is not a total quotient ring, i.e., $R \subsetneq T(R)$.
Then the following statements hold.
\begin{enumerate}
\item If $p \in Z (R)$ is a prime element of $R$, then $p^{m} R = p^{m+1}R$ for some $m \geq 1$.

\item If $\alpha \in Z (R)$, then $\alpha = ab$ for some $a \in \reg (R)$ and $b \in Z(R)$ that is a finite product of prime elements of $R$ in $Z (R)$.

\item If $I$ is a $Z$-ideal of $R$, then $I = \alpha J$ for some $\alpha \in Z (R)$ and a regular ideal $J$ of $R$.

\item If $\alpha \in R$ and $J$ is a regular ideal of $R$, then $( \alpha J)_u = \alpha J_u = \alpha J_v = \alpha J_t$.
In particular, each principal fractional ideal of $R$ is a $u$-ideal.

\item  $X_r^1(R) = X^1(R)$.

\item $u$-$\Max(R) = X^1(R) \cup \{P \in \Max(R) \mid P \mbox{ is minimal} \}$.

\item Each minimal prime ideal of $R$ is principal and $\dim (T(R)) =0$.

\item If $R$ has a unique minimal prime ideal, then $R$ is either a Krull domain or an SPR.
\end{enumerate}
\end{theorem}

\begin{proof}
Let $R = R_1 \times \cdots \times R_k \times R_{k+1} \times \cdots \times R_n$ be a finite direct product of Krull domains $R_1, \ldots, R_k$ and SPRs $R_{k+1}, \ldots, R_n$.
Now, let $P_i$ be the maximal ideal of $R_i$ for $i = k+1, \ldots, n$. Then, by Lemma~\ref{direct product}(5),
\[
\begin{aligned}
Z (R) & = \big( \{ 0 \} \times R_2 \times \cdots \times R_n \big) \cup \cdots \cup \big( R_1 \times \cdots \times \{ 0 \} \times R_{k+1} \times \cdots \times R_n \big) \\
           & \quad \cup \big( R_1 \times \cdots \times R_k \times P_{k+1} \times \cdots \times R_n \big) \cup \cdots \cup \big( R_1 \times \cdots \times R_{n-1} \times P_n \big),
\end{aligned}
\]
 so $Z (R)$ is a finite union of minimal prime ideals.
Next, let $q_i$ be the prime element of $R_i$ for $i = k+1, \ldots, n$,
so $q_i R_i$ is the maximal ideal of $R_i$ and $q_i^{m_i}R_i = (0)$ for some integer $m_i \geq 2$.
Let
\[
  p_i = \left\{ \begin{array}{ll}
                      (1, \ldots, 1, \underset{i{th}}{0}, 1, \ldots, 1) & \hbox{for $i = 1, \ldots, k$} \\
                      (1, \ldots, 1, \underset{i{th}}{q_i}, 1, \ldots, 1) & \hbox{for $i = k+1, \ldots, n$} \,.
                    \end{array} \right.
\]
Clearly, $\{ p_i R \mid i = 1, \ldots, n \}$ is the set of minimal prime ideals of $R$,
$\{p_iR \mid i=k+1, \dots , n\} \subseteq$ Max$(R)$,
$p_i R = p^{2}_i R$ for $i = 1, \ldots, k$, $p_i^{m_i}R = p_i^{m_i+1}R$ for $i = k+1, \ldots, n$,
and $Z (R) = \bigcup_{i=1}^{n} p_i R$.\\

(1) Note that $pR$ is a minimal prime ideal of $R$, so $pR = p_i R$ for some $i \in \{ 1, \ldots, n \}$.
Thus, if $i \in \{ 1, \ldots, k \}$, then $pR = p^{2}R$; and if $i \in \{ k+1, \ldots, n \}$,
then $p^mR = p^{m+1}R$ for some integer $m \geq 2$.

(2) Let $\alpha = (x_1, \ldots, x_k, y_{k+1}, \ldots, y_n)$. Then
\[
  \alpha = (x_1,1, \ldots, 1) (1, x_2, 1, \ldots, 1) \cdots (1, \ldots, 1, y_{k+1}, \ldots, 1) \cdots (1, \ldots, 1, y_n) \,,
\]
so if we let $\hat{x}_i = (1, \ldots, 1, x_i, 1, \ldots, 1)$ and $\hat{y}_i = (1, \ldots, 1, y_i, 1, \ldots, 1)$, then
$$\alpha = \hat{x}_1 \cdots  \hat{x}_k  \hat{y}_{k+1} \cdots  \hat{y}_n.$$
It is easy to see that $x_i = 0$ (resp., $x_i \neq 0$) if and only if $\hat{x}_i = p_i$ (resp., $\hat{x}_i$ is regular).
Also, $y_i R_i \subseteq q_i R_i$ (resp., $y_i R_i = R_i$) if and only if
$\hat{y}_i R = p^{z_i}_i R$ for some $z_i \ge 1$ (resp., $\hat{y}_i$ is regular).
Note that if $\hat{y}_i R = p^{z_i}_i R$ for some $z_i \ge 1$, then $\hat{y}_i = u_i p^{z_i}_i$ for some unit $u_i$ of $R$.
Thus, $\alpha = ab$ for some $a \in \reg (R)$ and $b \in Z (R)$ that is a finite product of $p_1, \ldots, p_n$.

(3) Let $I = I_1 \times \cdots \times I_n$, where $I_i$ is an ideal of $R_i$ for $i = 1, \ldots, n$.
Next, let $A_i = R_1 \times \cdots \times R_{i-1} \times I_i \times R_{i+1} \times \cdots \times R_n$ for $i = 1, \ldots, n$.
Then each $A_i$ is an ideal of $R$ and $I = A_1 \cdots A_n$.
Note that $A_i = p_i R$ or $A_i$ is regular for $i = 1, \ldots, k$.
Note also that $A_i = p^{z_i'}_i R$ for some $z_i' \ge 1$ or $A_i = R$ for $i = k+1, \ldots, n$. Now, let
\[
  e_i = \left\{ \begin{array}{ll}
                      0 & \hbox{if $A_i$ is regular} \\
                      1 & \hbox{if $A_i = p_i R$}
                    \end{array} \right.
\]
for $i = 1, \ldots, k$,
\[
  z_i = \left\{ \begin{array}{ll}
                      0     & \hbox{if $A_i = R$} \\
                      z_i' & \hbox{if $A_i = p^{z_i'}_i R$}
                    \end{array} \right.
\]
for $i = k+1, \ldots, n$, $\alpha = p^{e_1}_1 \cdots p^{e_k}_k p^{z_{k+1}}_{k+1} \cdots p^{z_n}_n$, and $J = A_{i_1} \cdots A_{i_j}$
for $1 \le i_1 \lneq \cdots \lneq i_j \le k$ with $e_{i_j} = 0$.
Then $I = \alpha J$, $\alpha$ is a finite product of prime elements of $R$ in $Z (R)$, and $J$ is regular.

(4) Let $\alpha = (a_1, \ldots, a_n)$ and $J = J_1 \times \cdots \times J_n$.
Then $\alpha J = (a_1 J_1) \times \cdots \times (a_n J_n)$, and hence $(\alpha J)_u = (a_1 J_1)_u \times \cdots \times (a_n J_n)_u$
by Lemma~\ref{u-diect product}(2).
For $i = 1, \ldots, k$, $R_i$ is a Krull domain, whence $(a_i J_i)_u = a_i (J_i)_u$.
Also, by Proposition~\ref{regLocal}(4), $(a_i J_i)_u = a_i J_i = a_i (J_i)_u$ for $i = k+1, \ldots, n$. Thus,
\begin{eqnarray*}
(\alpha J)_u & = & a_1 (J_1)_u \times \cdots \times a_n (J_n)_u \\
                    & = & \alpha \big( (J_1)_u \times \cdots \times (J_n)_u \big)
                     = \alpha(J_1 \times \cdots \times J_n)_u \\
                    & = & \alpha J_u \,.
\end{eqnarray*}
Note that $R$ is a Krull ring and $J$ is regular. Thus, $J_u = J_v = J_t$ by Corollary~\ref{coro4.7}.
In particular, if $I = \alpha R$ for some $\alpha \in T (R)$, then $d \alpha \in R$ for some $d \in \reg (R)$,
and hence $d I_u = (dI)_u = (d\alpha R)_u = d \alpha R$. Thus, $I_u = \alpha R = I$.

(5) Let $P \in X_r^1(R)$. Then $P$ contains at least one of $p_j$ for $j = 1, \dots, k$,
and hence
$P = R_1 \times \cdots \times P_j \times \cdots \times R_n$ for some nonzero minimal prime
ideal $P_j$ of $R_j$. Thus, $\h P = \h P_j = 1$. Conversely, if $Q \in X^1(R)$, then
$Q \nsubseteq Z(R)$ because $Z (R) = \bigcup_{i=1}^{n} p_i R$, so
$Q$ is regular. Thus, $Q \in X_r^1(R)$.

(6) If $P \in X^1(R)$, then $P \in X_r^1(R)$ by (5), and since
$R$ is a Krull ring, $P$ is a maximal $t$-ideal by Corollary \ref{coro4.7}.
Thus, $P$ is a maximal $u$-ideal by Corollary \ref{u-maxiaml}. Next, if $P$ is a maximal ideal of $R$
that is also minimal, then $P$ is a $u$-ideal by Proposition \ref{regLocal}(2),
and thus $P$ is a maximal $u$-ideal. For the reverse containment, let $Q \in u$-$\Max(R)$.
If $Q$ is regular, then $Q$ is a maximal $t$-ideal by Corollary \ref{u-maxiaml}, and since
$R$ is a Krull ring, $P \in X_r^1(R)$ by Corollary \ref{coro4.7}. Thus, $Q \in X^1(R)$ by (5).
Next, assume that $Q$ is a $Z$-ideal. Then $Q$ is minimal, and
since $Q$ is a maximal $u$-ideal of $R$, $Q$ is a maximal ideal.

(7) Recall that $\{p_iR \mid i = 1, \dots, n\}$ is the set of minimal prime ideals of $R$
and $Z(R) = \bigcup_{i=1}^np_iR$. Hence, each minimal prime ideal of $R$ is principal
and $\dim (T(R)) = 0$.

(8) If $R = R_1 \times \cdots \times R_n$, then $R$ has exactly $n$ minimal prime ideals
by the first paragraph of this proof. Thus, if $R$ has a unique minimal prime ideal,
then $n=1$, which implies that $R$ is either a Krull domain or an SPR.
\end{proof}

The divisor class group of a Krull domain is very important
because it measures how far away the Krull domain is from being a UFD.
Let $R$ be a Krull ring and Inv$_t(R)$ be the group of $t$-invertible regular
fractional $t$-ideals of $R$. Then Inv$_t(R)$ is a free abelian group with a basis $X^1(R)$,
and hence $Cl(R)$ is generated by the set of ideal classes in $Cl(R)$ containing a height-one prime ideal.
The next corollary with Theorem \ref{thm5.2} shows that
the class group of a general Krull ring reflects its ideal factorization property
(cf. Theorem \ref{chracter-cg}(3)).

\begin{corollary}
Let $R$ be a general Krull ring, $\mathfrak{I}_u(R)$ be the set of
integral $u$-ideals of $R$, and $\mathfrak{I}Prin(R)$
be the semigroup of integral principal ideals of $R$ under the usual ideal multiplication.
Then the following statements hold.
\begin{enumerate}
\item $\mathfrak{I}_u(R)$ is a commutative semigroup with identity under $I*J = (IJ)_u$ for all $I,J \in \mathfrak{I}_u(R)$ and
$\mathfrak{I}Prin(R)$ is a subsemigroup of $\mathfrak{I}_u(R)$.
\item $\mathfrak{I}_u(R)/\mathfrak{I}Prin(R)$ is an abelian group and $Cl(R) =  \mathfrak{I}_u(R)/\mathfrak{I}Prin(R)$.
\end{enumerate}
\end{corollary}

\begin{proof}
(1) It is routine to check that $\mathfrak{I}_u(R)$ is a commutative semigroup with
identity $R$ (as an element of $\mathfrak{I}_u(R)$).
Moreover, each principal ideal of $R$ is a $u$-ideal by Theorem \ref{thm5.2}(4)
and $R \in \mathfrak{I}Prin(R)$.
Thus, $\mathfrak{I}Prin(R)$ is a subsemigroup of $\mathfrak{I}_u(R)$.

(2) Let $I$ be an integral $u$-ideal of $R$. Then $I = \alpha J$ for some $\alpha \in R$
and a regular ideal $J$ of $R$ and $(IJ^{-1})_u = (\alpha JJ^{-1})_u = \alpha (JJ^{-1})_u$
by Theorem \ref{thm5.2}. Now, note that $R$ is a Krull ring, so $(JJ^{-1})_u = R$ by Corollary \ref{coro4.7}.
Hence, if $d \in \reg(J)$, then $dJ^{-1} \in \mathfrak{I}_u(R)$
and $I*(dJ^{-1}) = \alpha dR$. Thus, $\mathfrak{I}_u(R)/\mathfrak{I}Prin(R)$ is an abelian group.

Now, let $\mathfrak{I}_t(R)$ be the set of regular integral $t$-ideals of $R$.
A similar argument shows that $\mathfrak{I}_t(R)$ is the semigroup
under $I*_tJ = (IJ)_t$. Note that $R$ is a Krull ring, so $I_t = I_u$ for all
regular integral ideals $I$ of $R$ by Corollary \ref{coro4.7},
whence $\mathfrak{I}_t(R)$ is a subsemigroup of $\mathfrak{I}_u(R)$. Moreover,
if $\mathfrak{I}_rPrin(R)$ is the set of regular integral principal ideals of $R$,
then $\mathfrak{I}_rPrin(R)$ is a subsemigroup of both $\mathfrak{I}_t(R)$
and $\mathfrak{I}Prin(R)$ and $Cl(R) = \mathfrak{I}_t(R)/\mathfrak{I}_rPrin(R)$.
Note that the composition of the maps
$\mathfrak{I}_t(R) \hookrightarrow   \mathfrak{I}_u(R) \twoheadrightarrow \mathfrak{I}_u(R)/\mathfrak{I}Prin(R)$
is a semigroup homomorphism, which is also surjective and the kernel is $\mathfrak{I}_rPrin(R)$.
Thus, $Cl(R) =  \mathfrak{I}_u(R)/\mathfrak{I}Prin(R)$.
\end{proof}

An element $a$ of a ring $R$ is primary if $aR$ is a primary ideal.
Hence, a prime element is primary.
In \cite{s67}, Storch said that an integral domain is an almost factorial ring if it is a Krull domain with torsion divisor class group.
He showed that a Krull domain $D$ is an almost factorial ring if and only if some power of each nonzero element in $D$ is a
product of primary element, if and only if, for $a,b \in D$,
there is an integer $n = n(a,b)>0$ such that $a^nD \cap b^nD$ is principal \cite[Proposition 6.8]{f73}.
We will say that $R$ is an {\it almost UFR} if $R$ is a general Krull ring with $\Cl(R)$ torsion.
It is clear that a UFR is an almost UFR.

\begin{corollary}
The following statements are equivalent for a ring $R$.
\begin{enumerate}
\item $R$ is an almost UFR.
\item $R$ is a general Krull ring in which
some power of each nonunit element is a finite product of primary elements.

\item $R$ is a finite direct product of almost factorial domains and SPRs.
\end{enumerate}
\end{corollary}

\begin{proof}
$R$ is a general Krull ring in all of (1), (2), and (3) cases.
Hence, $R = R_1 \times \cdots \times R_k \times \cdots \times R_n$
for some Krull domains $R_1, \dots , R_k$ and SPRs $R_{k+1}, \dots, R_n$.

(1) $\Leftrightarrow$ (3) By Corollaries~\ref{coro3.3} and \ref{coro4.6},
$$\Cl (R) = \Cl_u(R) = \Cl_u (R_1) \times \cdots \times \Cl_u (R_k) = \Cl (R_1) \times \cdots \times \Cl (R_k).$$
Thus, the proof is completed by observing that
$\Cl(R)$ is torsion if and only if $\Cl(R_i)$ is torsion for $i=1, \dots, k$.

(2) $\Leftrightarrow$ (3) We first note that
$R$ is an almost UFR if and only if $R_i$ is an almost factorial domain for $i=1, \dots , k$
by the proof of (1) $\Leftrightarrow$ (3).
Let $x = (x_1, \dots, x_k, \dots,  x_n) \in R$, so if we let
$\hat{x}_i = (1, \dots , x_i, \dots , 1)$ for $i =1, \dots ,n$, then
$x = \hat{x}_1 \cdots \hat{x}_k \cdots \hat{x}_n.$ Clearly,
$\hat{x}_i$  is a finite product of prime elements for $i= k+1, \dots , n$.
Note that for an integer $q \geq 1$, $\hat{x}_i^q$ is a finite product of
primary elements if and only if $x_i^q$ is (cf. Lemma \ref{direct product}(3)).
Hence, if $R$ is an almost UFR,
then there is an integer $m_i \geq 1$ such that $\hat{x_i}^{m_i}$ is a finite product of primary elements
for $i=1, \dots, k$ \cite[Proposition 6.8]{f73}.
Thus, if $m = m_1 \cdots m_k$, then $x^m = \hat{x}_1^m \cdots \hat{x}_k^m \cdots \hat{x}_n^m$ is a finite product of primary elements.
Conversely, if some power of $\hat{x}_i$ is a finite product of primary elements, then some power of $x_i$ is
a finite product of primary elements. Hence, $R_i$ is an almost factorial domain for $i=1, \dots , k$.
\end{proof}

\subsection{The prime factorization of ideals in general Krull rings}

In this subsection, we state the
prime factorization of ideals in general Krull rings via the
$u$-operation and apply the result to give a couple of characterizations
of general Krull rings. To do so,
the following lemma is needed.

\begin{lemma} \label{local-properties}
Let $P$ be a prime $u$-ideal of a ring $R$ and $A, B \in \mathsf K (R)$.
\begin{enumerate}
\item  $I_uR_P = IR_P$ for all $I \in \mathsf K (R)$.

\item $A_u = B_u$ if and only if $AR_Q = BR_Q$ for all $Q \in u$-$\Max (R)$.
\end{enumerate}
\end{lemma}

\begin{proof}
(1) Let $x \in I_u$. Then $xJ \subseteq I$ for some $J \in \rGV (R)$.
Note that $J \nsubseteq P$, so $x \in xR_P = xJR_P \subseteq IR_P$.
Hence, $I_u \subseteq IR_P$. Thus, $I_uR_P = IR_P$.

(2) If $A_u = B_u$, then by (1), $AR_Q = A_u R_Q = B_u R_Q = BR_Q$ for all $Q \in u$-$\Max (R)$.
Conversely, assume that $AR_Q = BR_Q$ for all $Q \in u$-$\Max (R)$.
By symmetry, it suffices to show that $A_u \subseteq B_u$.
For $x \in A_u$, let $J = \{ a \in R \mid ax \in B \}$.
Then, by (1) and assumption, $J \nsubseteq Q$ for all $Q \in u$-$\Max (R)$, so $J_u = R$.
Hence, $x \in xR = xJ_u \subseteq (xJ)_u \subseteq B_u$.
Thus, $A_u \subseteq B_u$.
\end{proof}

We are now ready to give the prime factorization property of general Krull rings
which are the motivation of this paper.

\begin{theorem} \label{s-krull}
The following statements are equivalent for a ring $R$.
\begin{enumerate}
\item $R$ is a general Krull ring.

\item Each principal ideal of $R$ is a finite $u$-product of prime ideals.

\item For each $a \in R$, $(aR)_u$ is a finite $u$-product of prime ideals.

\item Each integral $u$-ideal of $R$ is a finite $u$-product of prime ideals.

\item $R$ is a Krull ring, each minimal prime ideal of $R$ is a principal ideal, and $\dim (T(R)) = 0$.

\item $R$ is a Krull ring, the zero element of $R$ is a finite product of prime elements, and $\dim (T(R)) = 0$.
\end{enumerate}
\end{theorem}

\begin{proof}
(1) $\Rightarrow$ (4) Let $R = R_1 \times \cdots \times R_n$ be a finite direct product of Krull domains and SPRs $R_1, \ldots, R_n$.
Let $I$ be a $u$-ideal of $R$. Then $I = I_1 \times \cdots \times I_n$, where $I_i$ is an ideal of $R_i$ for $i = 1, \ldots, n$
by Lemma~\ref{direct product}(2).
Note that $I_i$ is a $u$-ideal of $R_i$ by Lemma~\ref{u-diect product}(2) and $R_i$ is either a Krull domain or an SPR, so
\[
  I_i = \big( P^{e_{i1}}_{i1} \cdots P^{e_{ik_i}}_{ik_i} \big)_u
\]
for some prime ideals $P_{ij}$ of $R_i$ and integers $e_{ij} \ge 1$ by Corollary \ref{coro4.7}
and Proposition \ref{regLocal}(4).
Now, let
\[
  Q_{ij} = R_1 \times \cdots \times P_{ij} \times \cdots \times R_n
\]
for all $i, j$.
Then $Q_{ij}$ is a prime ideal of $R_i$ by Lemma~\ref{direct product}(3) and
\[
    \begin{aligned}
      I \quad & = \quad \big( P^{e_{11}}_{11} \cdots P^{e_{1 k_{1}}}_{1 k_{1}} \big)_u \times \cdots \times \big( P^{e_{n1}}_{n1} \cdots P^{e_{n k_{n}}}_{n k_{n}} \big)_u \\
        \quad & = \quad \Big( \big( P^{e_{11}}_{11} \cdots P^{e_{1 k_{1}}}_{1 k_{1}} \big) \times \cdots \times \big( P^{e_{n 1}}_{n1} \cdots P^{e_{n k_{n}}}_{n k_{n}} \big) \Big)_u \\
        \quad & = \quad \Big( \prod_{i=1}^{n} \big( \prod_{j=1}^{k_i} Q_{ij}^{e_{ij}} \big) \Big)_u
    \end{aligned}
\]
by Lemma~\ref{u-diect product}.

(1) $\Rightarrow$ (2) Let $a \in R$. Then, by Theorem \ref{thm5.2}(4),
$aR$ is a $u$-ideal. Thus, $aR$ is a finite $u$-product of prime ideals by (1) $\Rightarrow$ (4) above.

(2) $\Rightarrow$ (3) Clear.

(4) $\Rightarrow$ (3) Clear.

(3) $\Rightarrow$ (1) Let $S = \reg (R)$ and $a \in R$.
Then $(aR)_u = (P_1' \cdots P_m')_u$ for some prime ideals $P_1', \ldots, P_m'$ of $R$, whence by Proposition~\ref{regLocal},
\[
  aR_S = (aR)_u R_S = (P_1' \cdots P_m')_u R_S = (P_1' \cdots P_m')R_S = (P_1' R_S) \cdots (P_m' R_S) \,.
\]
Clearly, either $P_i' R_S = R_S$ or $P_i' R_S$ is a prime ideal of $R_S$.
Note that each principal ideal of $R_S$ is of the form $aR_S$ for some $a \in R$.
Thus, $R_S$ is a $\pi$-ring, and since $R_S$ is a total quotient ring,
$R_S$ is a finite direct product of SPRs, say, $R_S = T_1 \times \cdots \times T_n$ \cite[Theorem 46.11]{3}.
Hence, $\dim (R_S) = 0$ and $R_S$ has exactly $n$ prime ideals,
which are all maximal ideals, by Lemma~\ref{direct product}. 

Let $P_1, \ldots, P_n$ be the prime ideals of $R$ such that $\{ P_1 R_S, \ldots, P_n R_S \}$ is the set of prime ideals of $R_S$.
Note that $P^{e_i}_i R_S + P^{e_j}_j R_S = R_S$ for $i \neq j$ and positive integers $e_i, e_j$, so it follows that
\[
  (0) = (P^{e_1}_1 R_S) \cdots (P^{e_n}_n R_S) = P^{e_1}_1 R_S \cap \cdots \cap P^{e_n}_n R_S
\]
for some integers $e_i \ge 1$.
We may assume that $(P^{k_1}_1 R_S) \cdots (P^{k_n}_n R_S) = (0)$ if and only if $k_i \ge e_i$ for all $i = 1, \ldots, n$.
Let $P^{e_i}_i R_S \cap R = Q_i$ for $i = 1, \ldots, n$.
Then $Q_i$ is a $P_i$-primary ideal of $R$, $Q_1 \cap \cdots \cap Q_n = (0)$,
and $Q_i R_S = P^{e_i}_i R_S$ for $i = 1, \ldots, n$.
Note that $(bR)_u = bR$ for all $b \in \reg (R)$ by Proposition~\ref{u-oper}(2) and $u \le v$,
so every proper regular principal ideal of $R$ is a finite
$v$-product of prime ideals by assumption.
Thus, $R$ is a Krull ring by Corollary \ref{coro4.7}, so $R$ is integrally closed,
and hence $R = R/Q_1 \times \cdots \times R/Q_n$ \cite[Lemma 3.1]{ch01},
each $R/Q_i$ is a Krull ring \cite[Proposition 3.2]{am95},
and $T(R) = T(R/Q_1) \times \cdots \times T(R/Q_n)$ by Lemma~\ref{direct product}.

We first assume that $P_i$ is not a maximal ideal of $R$.
Then there is a height-one prime ideal $P$ of $R$ such that $P_i \subsetneq P$.
Thus, $R_P$ is a $\pi$-ring because $I_u R_P = IR_P$ for all ideals $I$ of $R$
by Lemma \ref{local-properties}(1), whence $R_P$ is an integral domain \cite[Theorem 46.8]{3}.
Then $P_i R_P = (0) = Q_i R_P$, so $Q_i = P_i$.
Thus, $R/Q_i = R/P_i$ is a Krull domain.
Next, if $P_i$ is a maximal ideal of $R$, then $T(R/Q_i) = R/Q_i$, and hence $R/Q_i = R/P^{e_i}_i = R_S/P_i^{e_i} R_S$ is an SPR.

(1) $\Rightarrow$ (5) A general Krull ring is a Krull ring. Thus,
the result follows directly from Theorem \ref{thm5.2}(7).

(5) $\Rightarrow$ (1) $T(R)$ is a zero-dimensional PIR by assumption, so
the proof of (3) $\Rightarrow$ (1) above
shows that $R = R/Q_1 \times \cdots \times R/Q_n$ for some primary ideals $Q_i$ of $R$ such that $Q_1 \cap \cdots \cap Q_n = (0)$.
For $Q_i$, let $P_i = \sqrt{Q_i}$. If $P_i$ is maximal, then $R/Q_i$ is an SPR
with a unique prime ideal $P_i/Q_i$.
Next, we assume that $P_i$ is not maximal.
Then $P_i$ is contained in a height-one prime ideal $P$ of $R$, which is regular, so $P$ is $t$-invetible.
Thus, $PR_P$ is principal.
Now, let $P_i R_P = qR_P$ and $PR_P = pR_P$ for some $p, q \in R$.
Then $q = pq_1$ for some $q_1 \in R_P$, and since $p \notin P_i R_P$, $q_1 \in P_i R_P$.
Thus, $q_1 = qx$ for some $x \in R_P$, whence $q = pq_1 = pqx$, implying that $q (1-px) = 0$ in $R_P$.
Since $1-px \in R_p \setminus PR_P$, it follows that $q = 0$ in $R_P$.
Thus, $R_P$ is an integral domain, and so $Q_i = P_i$.
Therefore, $R/Q_i$ is a Krull domain.

(5) $\Leftrightarrow$ (6) This follows from \cite[Theorem 10]{ac11}
that each minimal prime ideal of $R$ is  principal if and only if
the zero element of $R$ is a finite product of prime elements.
\end{proof}

If $R$ is a general Krull ring, then every principal ideal of $R$ can be written
as a finite $u$-product of prime ideals by Theorem~\ref{s-krull}.
The proof of Theorem~\ref{s-krull} shows that such prime ideals can
be chosen so that the heights are at most one.
Moreover, the height of such a prime ideal is one if and only if it is regular.

\begin{remark} \label{remark5.8}
Let $R$ be a general Krull ring and $a \in R$.
Then $aR$ is a finite $u$-product of prime ideals by Theorem \ref{s-krull} and
$a =bc$ for some $b \in \reg(R)$ and $c \in Z(R)$ that is a finite product of prime elements of $R$
in $Z(R)$ by Theorem \ref{thm5.2}(2). Note that $R$ is a Krull ring, so $bR = (P_1^{e_1} \cdots P_n^{e_n})_u$
for some $P_i \in X^1(R)$ and integers $e_i \geq 1$ and this expression is unique up to the order.
Next, $cR = (Q_1^{k_1} \cdots Q_m^{k_m})_u$ for some prime $Z$-ideals $Q_i$ and integers $k_i \geq 1$.
Note that either $Q_i^2 = Q_i$ or $Q_i \supsetneq Q_i^2 \supsetneq \cdots \supsetneq Q_i^{q_i}
= Q_i^{q_i+1}$ for some integer $q_i \geq 2$ by the proof of Theorem \ref{thm5.2}.
Hence, if we choose $k_i$ so that $k_i = 1$ or $k_i \leq q_i$, then
the expression $cR = (Q_1^{k_1} \cdots Q_m^{k_m})_u$ is unique up to the order.
So,
\begin{eqnarray*}
aR &=& bcR\\ &=& ((bR)(cR))_u
 = ((P_1^{e_1} \cdots P_n^{e_n})_u(Q_1^{k_1} \cdots Q_m^{k_m})_u)_u\\
&=& (P_1^{e_1} \cdots P_n^{e_n}Q_1^{k_1} \cdots Q_m^{k_m})_u
\end{eqnarray*}
and this expression is unique up to the order as in the case of Krull domains.
\end{remark}

Let $R$ be a general Krull ring. Then $u=w$ on $R$ because $R$ satisfies Property(A). Hence, by Theorem \ref{s-krull},
each principal ideal of $R$ is a finite $w$-product of prime ideals.
We next show that the converse is also true, so a general Krull ring
can be characterized via the $w$-operation.

\begin{corollary} \label{coro5.7}
The following statements are equivalent for a ring $R$.
\begin{enumerate}
\item $R$ is a general Krull ring.
\item Each principal ideal of $R$ is a finite $w$-product of prime ideals.
\item $(aR)_w$ is a finite $w$-product of prime ideals for all $a \in R$.
\item Each integral $w$-ideal of $R$ is a finite $w$-product of prime ideals.
\end{enumerate}
\end{corollary}

\begin{proof}
(1) $\Rightarrow$ (2) and (4) A general Krull ring satisfies Property(A), so
$u = w$ on $R$ by Proposition \ref{u-w-w'-oper}. Thus,
these follow directly from Theorem \ref{s-krull}.

(2) $\Rightarrow$ (3) Clear.

(4) $\Rightarrow$ (3) Clear.

(3) $\Rightarrow$ (1) If $R$ satisfies Property(A), then
$u=w$ by Proposition \ref{u-w-w'-oper}(2). Hence, by Theorem \ref{s-krull},
it suffices to show that $R$ satisfies Property(A). Recall that if dim$(T(R)) = 0$, then $R$
satisfies Property(A) \cite[Corollaries 2.6 and 2.12]{4}. Hence,
we will show that dim$(T(R))=0$. Since the $w$-operation is reduced, by (3),
$$(0) = (P_1^{e_1} \cdots P_n^{e_n})_w = P_1^{e_1} \cdots P_n^{e_n}$$
for some distinct minimal prime ideals $P_i$ of $R$ and integers $e_i \geq 1$.

Let $M$ be a maximal $w$-ideal of $R$ and $a \in R$. Then
$(aR)_w = (Q_1 \cdots Q_m)_w$ for some prime ideals $Q_1, \dots, Q_m$ of $R$,
and hence, by \cite[Corollary 3.10]{ywzc11}, $$aR_M = (Q_1 \cdots Q_m)R_M = (Q_1R_M) \cdots (Q_mR_M),$$
which implies that $R_M$ is a $\pi$-ring, so $R_M$
is a DVR or an SPR. Thus, $M$ contains exactly one of $P_i$'s and
$(P_1^{e_1} \cdots P_n^{e_n})R_M   = ((P_1^{e_1})_w \cap \cdots \cap (P_n^{e_n})_w)R_M$.
Again, by \cite[Corollary 3.10 and Proposition 3.11]{ywzc11},
\begin{eqnarray*}
(0) &=& P_1^{e_1} \cdots P_n^{e_n}\\
&=& ((P_1^{e_1})_w \cap \cdots \cap (P_n^{e_n})_w)_w \\ &=& (P_1^{e_1})_w \cap \cdots \cap (P_n^{e_n})_w.
\end{eqnarray*}
Clearly, each $(P_i^{e_i})_w$ is $P_i$-primary. Hence,
$Z(R) = \bigcup_{i=1}^nP_i$ and $\Max \big( T(R) \big) = \{P_i T(R) \mid i =1, \dots , n\}$.
Thus, $\dim(T(R))=0$.
\end{proof}

An integral domain is a Krull domain if and only if it is an integrally
closed $u$-Noetherian domain \cite[Theorem 2.8]{wm97}, while
a Krull ring need not be a $u$-Noetherian ring (e.g., a non-Noetherian
total quotient ring). The next corollary shows that a general Krull ring
is $u$-Noetherian.

\begin{corollary} \label{T(R) PIR}
The following statements hold for a general Krull ring $R$.
\begin{enumerate}
\item $T(R)$ is a zero-dimensional PIR.
\item $R$ satisfies the ascending chain condition on integral $u$-ideals. Hence,
$R$ is a $u$-Noetherian ring.
\end{enumerate}
\end{corollary}

\begin{proof}
(1) The proof of Theorem~\ref{s-krull} shows that $T(R)$ is a finite direct product of SPRs.
Thus, $T(R)$ is a zero-dimensional PIR.

(2) Let $I$ be a $u$-ideal of $R$.
Then $I_u = (P^{e_1}_1 \cdots P^{e_n}_n)_u$ for some distinct prime ideals $P_1, \ldots, P_n$ and positive integers $
e_1, \ldots, e_n$ by Theorem~\ref{s-krull}.
If $J$ is a $u$-ideal of $R$ with $I \subseteq J$, then $J$ is a finite $u$-product of prime ideals,
which is of the form $J = (P^{k_1}_1 \cdots P^{k_n}_n)_u$ for some integers $k_i$ with $0 \le k_i \le e_i$ for $i = 1, \ldots, n$
by Remark \ref{remark5.8}.
Thus, if $I = I_1 \subsetneq \cdots \subsetneq I_m$ is a chain of integral $u$-ideals of $R$,
then $m \le \sum_{i=1}^{n} e_i < \infty$.
\end{proof}

The following corollary is a general Krull ring analog of \cite[Proposition 3.2]{am95}
that if $R$ is a finite direct product of rings, then
$R$ is a Krull ring if and only if each direct summand of $R$ is a Krull ring.

\begin{corollary} \label{SKR direct product}
Let $R = R_1 \times R_2$ be the direct product of rings $R_1$ and $R_2$.
Then $R$ is a general Krull ring if and only if both $R_1$ and $R_2$ are general Krull rings.
\end{corollary}

\begin{proof}
($\Leftarrow$) Clear.
($\Rightarrow$) For $a \in R_1$, let $I = (a,1)R$.
Then $I = aR_1 \times R_2$, and since $R$ is a general Krull ring,
\[
  aR_1 \times R_2 = (Q_1 \cdots Q_k)_u \,
\]
for some prime ideals $Q_1, \ldots, Q_k$ of $R$.
Clearly, each $Q_i = P_i \times R_2$ for some prime ideal $P_i$ of $R_1$ by Lemma \ref{direct product}. Hence,
\[
  aR_1 \times R_2 = \big( (P_1 \cdots P_k) \times R_2 \big)_u = (P_1 \cdots P_k)_u \times R_2 \,,
\]
and thus $aR_1 = (P_1 \cdots P_k)_u$.
Thus, by Theorem~\ref{s-krull}, $R_1$ is a general Krull ring.
Similarly, $R_2$ is also a general Krull ring.
\end{proof}

A general Krull ring is a Krull ring whose total quotient ring is a PIR. However,
we next give a Krull ring $R$ such that $T(R)$ is an SPR but $R$ is not a general Krull ring.
This example shows that the two conditions of Theorem \ref{s-krull}(5) are the best for
a Krull ring to be a general Krull ring.

\begin{example} \label{ex5.10}
Let $V$ be a rank-two discrete valuation ring, $Q$ be a primary ideal of $V$ such that
$\h (\sqrt{Q}) = 1$ and $Q \subsetneq \sqrt{Q}$, and $R = V/Q$ be the factor ring of $V$ modulo $Q$.
Then the following conditions hold.
\begin{enumerate}
\item $T(R)$ is an SPR.

\smallskip
\item $R$ is a Krull ring.

\smallskip
\item $R$ is not a general Krull ring.
\end{enumerate}
\end{example}

\begin{proof}
(1)  Let $P = \sqrt{Q}$ and $\overline{x} = x + Q \in R$ for all $x \in V$.
For $a, b \in V$, assume that $\overline{a} \cdot \overline{b} = \overline{0}$ in $R$
and $\overline{a} \notin P$. Then $ab \in Q$
and $a \notin P$, and since $Q$ is primary, $b \in Q$, whence $\overline{b} = \overline{0}$ in $R$.
Hence, $\reg (R) = R \setminus P/Q$ and $$T(R) = (V/Q)_{P/Q} \cong V_P / QV_P.$$
Thus, $T(R)$ is an SPR because $V_P$ is a DVR.

\smallskip
(2) Let $M$ be the maximal ideal of $V$. Then $M/Q$ is a regular prime
ideal of $R$ by the proof of (1), reg-ht$(M/Q) = 1$, and $M/Q$ is principal.
Hence, $R$ is a rank-one discrete valuation ring. Thus, $R$ is a Krull ring.

\smallskip
(3) Assume that $R$ is a general Krull ring. Then the minimal prime ideals of $R$
are principal by Theorem \ref{s-krull}, so
$\overline{P} = P/Q$ is principal, which implies that $P$ is principal, a contradiction.
Thus, $R$ is not a general Krull ring.
\end{proof}

A Krull ring $R$ such that $T(R)$ is a PIR
need not be a general Krull ring by Example \ref{ex5.10}.
We next characterize when $R$ is a general Krull ring.

\begin{theorem} \label{Krull and strong Krull}
The following statements are equivalent for a ring $R$.
\begin{enumerate}
\item $R$ is a general Krull ring.

\item $R$ is a Krull ring, $T(R)$ is a PIR, and $R_P$ is a DVR for all $P \in X^{1}_r (R)$.

\item $R$ satisfies the following conditions.
      \begin{enumerate}
      \item $R = \underset{P \in X^{1}_r (R)}{\bigcap} R_{[P]}$.

      \item $R_P$ is a DVR for all $P \in X^{1}_r (R)$.

      \item Each regular element of $R$ is contained in only finitely many prime ideals in $X^{1}_r (R)$.

      \item $T(R)$ is a PIR.
      \end{enumerate}
\end{enumerate}
\end{theorem}

\begin{proof}
(1) $\Rightarrow$ (2) Let $R = R_1 \times \cdots \times R_n$ be a direct product of Krull domains and SPRs.
So if $Q \in X^{1}_r (R)$, then $Q = R_1 \times \cdots \times P_i \times \cdots \times R_n$
for some height-one prime ideal $P_i$ of a Krull domain $R_i$.
Hence, $R_Q \cong (R_i)_{P_i}$, and since $R_i$ is a Krull domain, $(R_i)_{P_i}$ is a DVR.
The other properties follow from Theorem~\ref{s-krull} and Corollary~\ref{T(R) PIR}.

(2) $\Leftrightarrow$ (3) Assume that $R_P$ is a DVR for all $P \in X^{1}_r (R)$.
Then $( R_{[P]}, [P]R_{[P]} )$ is a rank-one DVR \cite[Theorem 1]{ck02}.
Thus, $R$ is a Krull ring if and only if
$R$ satisfies the conditions (a) and (c) \cite[Theorem 3.5]{ck21}.

(2) $\Rightarrow$ (1) Since $T(R)$ is a PIR, $(0)$ has an irredundant primary decomposition in $R$, say, $(0) = Q_1 \cap \cdots \cap Q_n$. Then
\[
  R \hookrightarrow R/Q_1 \times \cdots \times R/Q_n \hookrightarrow T( R/Q_1 ) \times \cdots \times T (R/Q_n ) \cong T(R) \,.
\]
Note that $R$ is integrally closed and $R/Q_1 \times \cdots \times R/Q_n$ is a finitely generated $R$-module,
so $R = R/Q_1 \times \cdots \times R/Q_n$.
Now, if $\sqrt{Q_i}$ is not a maximal ideal, then $(R/Q_i)_{P_i/Q_i} \cong R_{P_i}/Q_i R_{P_i}$
for some $P_i \in X^1_r(R)$ containing $Q_i$.
Note that $R_{P_i}$ is a DVR by assumption, so $Q_i R_{P_i} = (0)$. Hence, $R/Q_i$ is an integral domain,
and since $R$ is a Krull ring, $R/Q_i$ is a Krull domain.
Next, if $\sqrt{Q_i}$ is a maximal ideal, then $R/Q_i = T(R/Q_i)$, and since $T(R)$ is a PIR, $R/Q_i$ is an SPR.
Thus, $R$ is a general Krull ring.
\end{proof}

\begin{corollary} \label{krull overring}
Let $R$ be a general Krull ring and $S$ be an overring of $R$. Then
$S$ is a Krull ring if and only if $S$ is a general Krull ring.
\end{corollary}

\begin{proof}
By Theorem \ref{s-krull}, a general Krull ring is a Krull ring. Conversely,
assume that $S$ is a Krull ring. Note that $T(S) = T(R)$, so $T(S)$
is a PIR by Corollary \ref{T(R) PIR}. Now, let $Q \in X^1_r(S)$ and $P = Q \cap R$.
Then $P$ is regular, and since $R$ is a general Krull ring, $R_P$ is an integral domain.
Note that $$R_P \hookrightarrow S_{R \setminus P} \hookrightarrow T(R)_{R \setminus P} \hookrightarrow T(R_P),$$
so $S_{R \setminus P}$ is an integral domain, whence $S_Q$ is also an integral domain.
By assumption, $S$ is a Krull ring, so $Q \in X^1_r(S)$ implies $Q \subsetneq QQ^{-1}$. Hence,
$QS_Q$ is principal so that $S_Q$ is a DVR. Thus,
$S$ is a general Krull ring by Theorem \ref{Krull and strong Krull}.
\end{proof}

Let $R$ be a general Krull ring.
It is easy to see that $R$ is reduced
if and only if $R$ is a finite direct product of Krull domains
(cf. Proposition \ref{polynomial}). Hence, the following corollary shows
that $R/N(R)$ is a finite direct product of Krull domains.

\begin{corollary}
Let $R$ be a general Krull ring and $A$ be an ideal of $R$ contained in $N(R)$.
Then $R/A$ is a general Krull ring.
\end{corollary}

\begin{proof}
Note that $R$ is a finite direct product of Krull domains and SPRs,
so $x \in \reg (R)$ if and only if $x + A \in \reg (R/A)$, and $A T(R) = A$ by Lemma~\ref{direct product}(1) and (6).
Moreover, $X^{1}_r (R/A) = \{ P/A \mid P \in X^{1}_r (R) \}$. Hence, the following statements hold.
\begin{enumerate}
\item[(b)] $(R/A)_{(P/A)} \cong R_P / A R_P$ is a DVR for all $P \in X^{1}_r (R)$.
\smallskip
\item[(c)] Each regular element of $R/A$ is contained in only finitely many prime ideals in $X^{1}_r (R/A)$.
\smallskip
\item[(d)] $T(R/A) \cong T(R) / A T(R)$ is a PIR.
\smallskip
\item[(a)] Now, let $\alpha \in \underset{ P \in X^{1}_r (R)}{\bigcap} (R/A)_{[P/A]}$.
           Note that $T(R/A) = T(R)/ A T(R) = T(R)/A$, so $\alpha = x + A$ for some $x \in T(R)$.
           Let $I = \{ a \in R \mid (a + A) (x + A) \in R/A \}$.
           Then $I$ is a regular ideal of $R$ such that $A \subsetneq I \nsubseteq P$ for all $P \in X^{1}_r (R)$, whence $I_u = R$.
           Note that $(I/A) (x + A) \subseteq R/A$, so $xI \subseteq R$, and hence $x I_u \subseteq R$.
           Thus, $I = I_u = R$, which implies that $\alpha = x + A \in R/A$.
           Therefore, $R/A = \underset{P \in X^{1}_r (R)}{\bigcap} (R/A)_{[P/A]}$.
\end{enumerate}
Thus, by Theorem~\ref{Krull and strong Krull}, $R/A$ is a general Krull ring.
\end{proof}

Let $S$ be a multiplicative set of a ring $R$.
It is known that if $R$ is a Krull ring,
then $R_{[S]}$ is a Krull ring by \cite[Lemmas 4.1 and 4.2]{ck21}
(\cite[Corollary 43.6]{3} for the integral domain case and \cite[Proposition 37]{12} for the Marot Krull ring case),
while $R_S$ need not be a Krull ring by the following example.

\begin{example} \label{ex5.12}
Let $V$ be a valuation domain with prime ideals $(0) \subsetneq P \subsetneq M$
such that $V_P$ is not a DVR and $M$ is the maximal ideal of $V$. Let $V/M$ be the factor ring of $V$
modulo $M$ that is also a $V$-module, $R_0 = V(+)V/M$
be the idealization of $V/M$ in $V$ (see Example \ref{ex2.6}), and
$S = (V \setminus P)(+)V/M$. Then $T(R_0) = R_0$, $N(R_0) = (0)(+)V/M$, which is the minimal
prime ideal of $R_0$ \cite[Theorems 25.1(3) and 25.3]{4}), and $(R_0)_S = V_P(+)(V/M)_P$ \cite[Lemma 25.4]{4}.
Note that $(V/M)_P = (0)$ because every element of $V$ in $M \setminus P$ is a zero divisor of $V/M$,
so $(R_0)_S = V_P$. Now, let $D$ be a Krull domain with quotient field $K$,
$R = R_0 \times D$ be the direct product of $R_0$ and $D$, and assume $D \subsetneq K$.
Then $R$ is a Krull ring and $R \subsetneq T(R) = R_0 \times K$. Moreover, $S \times D$
is a multiplicative subset of $R$ and $R_{S \times D} = (R_0)_S = V_P$. Thus, $R_S$ is not a Krull ring.
\end{example}

It is worth noting that $R_{[S]}$ is an overring of $R$ but $R_S$ need not be an overring of $R$.
We next show that if $R$ is a general Krull ring, then
both $R_S$ and $R_{[S]}$ are general Krull rings. We first need a lemma.

\medskip
\begin{lemma} \label{lemma5.12}
Let $S$ be a multiplicative set of a ring $R$.
\begin{enumerate}
\item If $J \in \rGV (R)$, then $JR_S \in \rGV (R_S)$.

\item $\big( I_u R_S \big)_u = (IR_S)_u$ for all $I \in \mathsf K (R)$.
\end{enumerate}
\end{lemma}

\begin{proof}
(1) Since $J$ is a finitely generated regular ideal,
$JR_S$ is a finitely generated regular ideal of $R_S$. Moreover,
$(JR_S)^{-1} = J^{-1}R_S$ \cite[Theorem 4.4]{3}, and hence $(JR_S)^{-1} = R_S$.
Thus, $JR_S \in \rGV (R_S)$.

(2) If $x \in I_u$, then $xJ \subseteq I$ for some $J \in \rGV (R)$.
Hence, $xJR_S \subseteq IR_S$, and since $JR_S \in \rGV (R_S)$ by (1), $x \in (IR_S)_u$.
Thus, $I_u \subseteq (IR_S)_u$, whence $I_u R_S \subseteq (IR_S)_u$.
Now, note that $IR_S \subseteq I_u R_S$, so
\[
  (IR_S)_u \subseteq \big( I_u R_S \big)_u \subseteq \big( (IR_S)_u \big)_u = (IR_S)_u \,.
\]
Thus, $(IR_S)_u = \big( I_u R_S \big)_u$.
\end{proof}

\smallskip
\begin{corollary} \label{local-s-Krull}
Let $S$ be a multiplicative set of
a general Krull ring $R$. Then
both $R_S$ and $R_{[S]}$ are general Krull rings.
\end{corollary}

\begin{proof}
We first show that $R_S$ is a general Krull ring. Let $a \in R$.
Then, by Theorem~\ref{s-krull}, $aR = (P_1 \ldots P_n)_u$ for some prime ideals $P_1, \cdots, P_n$ of $R$.
Hence, $aR_S = (P_1 \cdots P_n)_u R_S$, so by Lemma~\ref{lemma5.12},
\[
  (aR_S)_u = \big( (P_1 \cdots P_n)_u R_S \big)_u = \big( (P_1 \cdots P_n)R_S \big)_u = \big( (P_1 R_S) \cdots (P_n R_S) \big)_u \,.
\]
Thus, $R_S$ is a general Krull ring by Theorem~\ref{s-krull} again.

Next, note that $R_{[S]}$ is an overring of $R$, and since $R$ is a Krull ring,
$R_{[S]}$ is also a Krull ring. Thus, $R_{[S]}$ is a general Krull ring by Corollary \ref{krull overring}.
\end{proof}

\smallskip
\subsection{Kaplansky-type Theorem}
\smallskip

Kaplansky's theorem states that an integral domain $R$ is a UFD
if and only if every nonzero prime ideal of $R$ contains a nonzero prime element \cite[Theorem 5]{8}.
In this subsection, we study the Kaplansky-type theorems of general Krull rings,
$\pi$-rings, and UFRs.

In \cite[Theorem 3.6]{k89}, Kang gave a Kaplansky-type characterization of a Krull domain.
That is, he proved that an integral domain $R$ is a Krull domain if (and only
if) every nonzero prime ideal of $R$ contains a $t$-invertible prime ideal.
He also generalized his result to rings with zero divisors as follows:
$R$ is a Krull ring if and only if every regular prime ideal of $R$ contains a $t$-invertible regular prime ideal
\cite[Theorem 13]{7}. Our first result states the general Krull ring analog.

\medskip
\begin{theorem} \label{SKR u-invertible}
A ring $R$ is a general Krull ring if and only if every nonprincipal prime ideal of $R$ contains a $u$-invertible prime ideal.
\end{theorem}

\begin{proof}
($\Rightarrow$) Let $Q$ be a prime ideal of $R$.
If $Q \subseteq Z (R)$, then $Q$ is principal by the proof of Theorem~\ref{thm5.2}.
Now, assume that $Q \nsubseteq Z (R)$. Then $Q$ is regular.
Choose $a \in \reg (Q)$, then by Theorem \ref{s-krull}, $aR = (P_1 \cdots P_n)_u$ for some prime ideals $P_1, \ldots, P_n$ of $R$.
Clearly, each $P_i$ is $u$-invertible and $aR \subseteq Q$ implies that $P_i \subseteq Q$ for some $i$.
Thus, $Q$ contains a $u$-invertible prime ideal.

($\Leftarrow$) Let $P$ be a prime ideal of $R$.
Note that a $u$-invertible ideal is regular by Corollary~\ref{u-inv regular}, so if $P \subseteq Z (R)$, then $P$ must be principal by assumption.
In particular, $T(R)$ is a PIR, and hence $\dim (T(R)) = 0$.
Now, assume that $P \nsubseteq Z (R)$. Then $P$ is regular,
and hence $P$ contains a $u$-invertible prime ideal, which is $t$-invertible by Corollary \ref{u-inv regular}.
Hence, $R$ is a Krull ring \cite[Theorem 13]{7}.
Thus, $R$ is a general Krull ring by Theorem~\ref{s-krull}.
\end{proof}

The next corollary is the $\pi$-ring analog of Theorem \ref{SKR u-invertible},
which is an extension of Kang's result that an integral domain $D$ is a $\pi$-domain
if and only if every nonzero prime ideal of $D$ contains an invertible
ideal \cite[Theorem 4.1]{k89}. Kang also showed that $R$ is a regular $\pi$-ring
if and only if each regular prime ideal of $R$ contains an invertible ideal \cite[Lemma 2]{k91}.

\begin{corollary} \label{pi-ring}
The following statements are equivalent for a ring $R$.
\begin{enumerate}
\item $R$ is a $\pi$-ring.

\item Every nonprincipal prime ideal of $R$ contains an invertible prime ideal.

\item $R$ is a regular $\pi$-ring, each minimal prime ideal of $R$ is principal, and $\dim (T(R)) = 0$.
\item $R$ is a general Krull ring in which each minimal regular prime ideal is invertible.
\end{enumerate}
\end{corollary}

\begin{proof}
(1) $\Rightarrow$ (4) It is clear that a $\pi$-ring is a general Krull ring. Also,
a $\pi$-ring is a regular $\pi$-ring. Hence, each minimal regular prime ideal is invertible.

(4) $\Rightarrow$ (3) A general Krull ring is a Krull ring, so $R$ is a regular $\pi$-ring by (4).
Thus, the result follows from Theorem~\ref{s-krull}.

(3) $\Rightarrow$ (2) Let $Q$ be a nonprincipal prime ideal of $R$.
Then, by Theorems~\ref{s-krull} and \ref{SKR u-invertible}, $Q$ contains a $u$-invertible prime ideal,
and since a $u$-invertible ideal is regular by Corollary~\ref{u-inv regular}, $Q$ is regular.
Thus, $Q$ contains an invertible prime ideal by assumption.

(2) $\Rightarrow$ (1) By Theorem~\ref{SKR u-invertible}, $R$ is a general Krull ring, so $R$ is a finite direct product of Krull domains and SPRs.
Let $D$ be a direct summand of $R$ that is a Krull domain.
Then, by Proposition~\ref{star direct product}(8), each nonzero prime ideal of $D$ contains an invertible prime ideal,
 whence $D$ is a $\pi$-domain \cite[Theorem 4.1]{k89}.
Thus, $R$ is a $\pi$-ring.
\end{proof}

In \cite[Theorem 2.6]{am95}, Anderson and Markanda showed that $R$ is a UFR if and only if (i) every minimal prime ideal of $R$ is principal and
(ii) every prime ideal of height $>$ 0 contains a principal prime ideal of height $>$ 0.

\begin{corollary}
The following statements are equivalent for a ring $R$.
\begin{enumerate}
\item $R$ is a UFR.

\item Every nonprincipal prime ideal of $R$ contains a regular prime element.

\item $R$ is a factorial ring, each minimal prime ideal of $R$ is principal, and $\dim (T(R)) = 0$.

\item $R$ is a general Krull ring in which each minimal regular prime ideal is principal.
\end{enumerate}
\end{corollary}

\begin{proof}
Note that $R$ is a UFR (resp., factorial ring)
if and only if $R$ is a $\pi$-ring (resp., regular $\pi$-ring) with $\Cl (R) = \{0\}$ by Theorem~\ref{chracter-cg}
(resp., Theorem \ref{Krull ring property}).
Note also that $R$ is a factorial ring if and only if
every regular prime ideal of $R$ contains a regular prime element \cite[Theorem]{am95-2}.
Thus, the result follows directly from Corollary~\ref{pi-ring}.
\end{proof}

\subsection{Mori-Nagata Theorem I}

Mori-Nagata theorem states that the integral closure of a Noetherian domain is a Krull domain.
It was conjectured by Krull \cite[page 108]{k35}, then
the local case was proved
by Mori \cite[Theorem 1]{m53}, and the general case was proved by Nagata \cite[Theorem 2]{n55}.
Huckaba showed that the integral closure of a Noetherian ring is a Krull ring \cite[Corollary 2.3]{h76}.
Chang and Kang generalized Huckaba's result to r-Noetherian rings,
i.e., they showed that the integral closure of an r-Noetherian ring is a Krull ring \cite[Theorem 13]{ck02}.
However, the next example shows that the integral closure of a Noetherian ring $R$
need not be a general Krull ring even though $T(R)$ is an SPR.

\begin{example} \label{ex4.8}
Let $\mathbb Q$ be the field of rational numbers, $\mathbb Q [X]$ be the polynomial ring over $\mathbb Q$, and $A = \mathbb Q [X] / (X^{2})$; so $A$ is an SPR.
Let $\mathrm{m} = (X)/(X^2)$, $Y$ be an indeterminate over $A$, and $R = A[Y]$.
Then $R$ is a one-dimensional Noetherian ring and $T(R) = A(Y)$, so $T(R)$ is an SPR \cite[Lemma 18.7]{4}.
But, note that $N(A) = \mathrm{m}$ and $N (R) = N(A)[Y]$,
so $N(R)= \mathrm{m}[Y]$ is a prime ideal of $R$.
Hence, if $\overline{R}$ is the integral closure of $R$, then $N (\overline{R}) = \mathrm{m}T(R)$,
which is a nonzero prime ideal of $\overline{R}$, but since $N(\overline{R}) \cap R = \mathrm{m}[Y]$,
$N(\overline{R})$ is not a maximal ideal of $\overline{R}$.
Thus, $\overline{R}$ is a Krull ring \cite[Corollary 2.3]{h76} but $\overline{R}$ is not a general Krull ring.
\end{example}

By Mori-Nagata theorem, an integrally closed Noetherian domain is a Krull domain.
Hence, the next result is a partial corresponding part of general Krull rings
to Mori-Nagata theorem.

\begin{theorem} \label{icNSKR}
Let $R$ be an integrally closed Noetherian ring.
Then $R$ is a general Krull ring if and only if $T(R)$ is a PIR.
\end{theorem}

\begin{proof}
If $R$ is general Krull, then $T(R)$ is a PIR by Corollary~\ref{T(R) PIR}.
Conversely, assume that $T(R)$ is a PIR, and let $(0) = Q_1 \cap \cdots \cap Q_n$ be an irredundant primary decomposition of $(0)$ in $R$.
Then $\h (\sqrt{Q_i}) = 0$ for $i = 1, \ldots, n$ and
\[
  R \hookrightarrow R/Q_i \times \cdots \times R/Q_n \hookrightarrow T(R/Q_1) \times \cdots \times T(R/Q_n) \cong T(R) \,.
\]
Note that $R$ is integrally closed and $R/Q_1 \times \cdots \times R/Q_n$ is a finitely generated
$R$-submodule of $T(R)$, so $R = R/Q_1 \times \cdots \times R/Q_n$.
Hence, it suffices to show that each $R/Q_i$ is a general Krull ring by Corollary~\ref{SKR direct product}.
For convenience, let $Q = Q_i$ and $P = \sqrt{Q_i}$.
Note that $R/Q$ is Noetherian, $Z (R/Q) = P/Q$, and $R/Q$ is integrally closed.
If $\dim (R/Q) = 0$, then $R/Q = T(R/Q)$, and hence $R/Q$ is an SPR by assumption.

Next, assume that $\dim (R/Q) \ge 1$, and let $M$ be a maximal ideal of $R$ such that $Q \subsetneq M$ and $\h (M/Q) \ge 1$.
Put $\overline{S} = \reg (R/Q \setminus M/Q)$.
Then $(R/Q)_{\overline{S}}$ is an integrally closed Noetherian ring and
$N \big( (R/Q)_{\overline{S}} \big) = Z \big( (R/Q)_{\overline{S}} \big) = (P/Q)_{\overline{S}}$,
whence the Jacobson radical $J \big( (R/Q)_{\overline{S}} \big)$ of
$(R/Q)_{\overline{S}}$ is not contained in $Z \big( (R/Q)_{\overline{S}} \big)$.
Thus, $N \big( (R/Q)_{\overline{S}} \big) = (0)$ \cite[Theorem 11.1]{4}, which
implies that $P = Q$.
Therefore, $R/Q$ is an integrally closed Noetherian domain. Thus, $R/Q$ is a Krull domain \cite[Theorem 2]{n55}.
\end{proof}

\begin{corollary} \label{coro5.18}
Let $K$ be a field, $n \ge 2$ be an integer, $X_1, \ldots, X_n$ be indeterminates over $K$, $K[X_1, \ldots, X_n]$ be the polynomial ring over $K$,
$A$ be an ideal of $K[X_1, \ldots, X_n]$, $R = K[X_1, \ldots, X_n] / A$, and $\overline{R}$ be the integral closure
of $R$. Then $\overline{R}$ is a general Krull ring if and only if $T(R)$ is a PIR and $\overline{R}$ is a Noetherian ring.
\end{corollary}

\begin{proof}
($\Leftarrow$) Theorem~\ref{icNSKR}.
($\Rightarrow$) $T(R)$ is a PIR by Corollary~\ref{T(R) PIR}.
Then, by the proof of Theorem~\ref{icNSKR},
\[
  R \hookrightarrow R/Q_1 \times \cdots \times R/Q_n \hookrightarrow T(R)
\]
for some primary ideals $Q_1, \ldots, Q_n$ of $R$,
and hence $\overline{R} = \overline{(R/Q_1)} \times \cdots \times \overline{(R/Q_n)}$,
where $\overline{(R/Q_i)}$ is the integral closure of $R/Q_i$ for $i =1, \dots, n$.
Note that $\overline{R}$ is a general Krull ring by assumption,
so each $\overline{(R/Q_i)}$ is a Krull domain or an SPR by Theorem \ref{thm5.2}(8).
If $\overline{(R/Q_i)}$ is a Krull domain, then $R/Q_i$ is a Noetherian domain,
and hence $\overline{(R/Q_i)}$ is a Noetherian domain \cite[Theorem 3]{d62}.
Next, if $\overline{(R/Q_i)}$ is an SPR, then $\overline{(R/Q_i)}$ is Noetherian.
Thus, $\overline{R}$ is Noetherian.
\end{proof}

\begin{remark}
(1) Let $R$ be a Noetherian ring and $\overline{R}$ be the integral closure of $R$.
By Theorem \ref{icNSKR}, if $\overline{R}$ is Noetherian, then $\overline{R}$
is a general Krull ring if and only if $T(R)$ is a PIR. However, there
is a Noetherian domain $D$ whose integral closure is a non-Noetherian Krull domain \cite[Proposition 2]{n54}.
Then $R: = D \times \mathbb{Q}$ is a Noetherian ring, $\overline{R}$ is not Noetherian,
but $\overline{R}$ is a general Krull ring. Hence, Corollary \ref{coro5.18}
is not true in general.

(2) Let the notation be as in Corollary \ref{coro5.18}.
Chang and Kang showed that $\overline{R}$ is an r-Noetherian ring \cite[Theorem 3.2]{ck21-1}
and studied when $\overline{R}$ is a Noetherian ring \cite[Theorem 3.3]{ck21-1}. Among them, they proved that
$\overline{R}$ is a Noetherian ring if and only if $N(R) = N(\overline{R})$.
\end{remark}

Krull-Akizuki theorem says that every overring of a one-dimensional Noetherian domain is a Noetherian domain \cite[Theorem 33.2]{n62}.
But, in \cite{n54}, Nagata constructed (i) a two-dimensional Noetherian domain $R$ such that there is a
non-Noetherian ring between $R$ and its integral closure
and (ii) a three-dimensional Noetherian domain whose integral closure is not Noetherian.
However, the integral closure of two-dimensional Noetherian domain is Noetherian, which was proved by Mori for local rings \cite{m53}
 and generalized by Nagata \cite[Theorem 3]{n55}.
Huckaba constructed an $n$-dimensional Noetherian ring whose integral closure is not Noetherian for any integer $n \ge 1$ \cite[Proposition 3.1]{h76},
while he showed that if $R$ is a Noetherian ring with $\dim (R) \le 2$, then the integral closure of $R$ is r-Noetherian \cite[Theorem B]{h76}.
Davis and Anderson-Huckaba studied when the integral closure of a Noetherian ring of dimension $\le$ 2 is Noetherian in \cite{d62} and \cite{ah78} respectively.

\begin{corollary}
Let $R$ be a Noetherian ring with $\dim (R) \le 2$ and $\overline{R}$ be the integral closure of $R$.
Then $\overline{R}$ is a general Krull ring if and only if $T(R)$ is a PIR and $\overline{R}$ is a Noetherian ring.
\end{corollary}

\begin{proof}
($\Leftarrow$) Theorem~\ref{icNSKR}.
($\Rightarrow$) $T(\overline{R}) = T(R)$, so $T(R)$ is a PIR by Corollary~\ref{T(R) PIR}.
Now, let $P$ be a prime ideal of $\overline{R}$. If $P \subseteq Z (\overline{R})$,
then $P$ is principal by the proof of Theorem~\ref{thm5.2}.
Next, if $P \nsubseteq Z (R)$, then $P$ is regular, and hence $P$ is finitely generated
\cite[Theorem B]{h76}. Thus, by Cohen's Theorem, $\overline{R}$ is Noetherian.
\end{proof}

\subsection{Mori-Nagata Theorem II}

A general Krull ring need not be reduced (see Example \ref{ex4.8}).
The next result shows when a general Krull ring is reduced.

\begin{proposition} \label{polynomial}
The following statements are equivalent for a ring $R$.
\begin{enumerate}
\item $R$ is a reduced general Krull ring.

\item $R$ is a finite direct product of Krull domains.

\item $R[X]$ is a general Krull ring.

\item $R[X]$ is a Krull ring.
\end{enumerate}
\end{proposition}

\begin{proof}
(1) $\Rightarrow$ (2) Let $D$ be a direct summand of $R$. Then $D$ is either a Krull domain or an SPR.
If $D$ is an SPR that is not a field, then $D$ contains a nonzero nilpotent element, so $R$
also contains a nonzero nilpotent element, a contradiction. Thus, $D$ is a Krull domain.

(2) $\Rightarrow$ (1) Clear.

(2) $\Rightarrow$ (3) Let $R = D_1 \times \cdots \times D_n$ be a finite direct product of Krull domains $D_1, \ldots, D_n$. Then
$R[X] \cong D_1 [X] \times \cdots \times D_n [X]$
and each $D_i [X]$ is a Krull domain \cite[Theorem 43.11]{3}.
Thus, $R[X]$ is a general Krull ring.

(3) $\Rightarrow$ (4) Clear.

(4) $\Leftrightarrow$ (2) \cite[Theorem 5.7]{aam85}.
\end{proof}

As we noted in Example \ref{ex4.8}, the integral closure of a Noetherian
ring need not be a general Krull ring. However, we next show
that the integral closure of a reduced Noetherian ring is a general Krull ring.
We first need the following lemma.

\begin{lemma} \label{reduced_strong_Krull}
Let $R$ be a reduced ring with a finite number of minimal prime ideals.
Then $R$ is a Krull ring if and only if $R$ is a general Krull ring.
\end{lemma}

\begin{proof}
Let $P_1, \ldots, P_n$ be the minimal prime ideals of $R$.
Then $P_1 \cap \cdots \cap P_n = (0)$, $\dim (T(R)) =0$, and $R$ is integrally closed, so $R \cong R/P_1 \times \cdots \times R/P_n$
\cite[Lemma 8.14]{4}. Thus, $R$ is a Krull ring if and only if
each $R/P_i$ is a Krull domain if and only if $R$ is a general Krull ring.
\end{proof}

It is well known that the integral closure of a reduced Noetherian ring is a
finite direct product of Krull domains. This fact is also called Mori-Nagata theorem because
it is an immediate consequence of Mori-Nagata theorem.
Hence, by Proposition \ref{polynomial}, the following result is a restatement of Mori-Nagata theorem.

\begin{proposition} \label{icNoetherian}
If $R$ is a reduced Noetherian ring,
then the integral closure of $R$ is a general Krull ring.
\end{proposition}

\begin{proof}
Let $\overline{R}$ be the integral closure of $R$. Then,
by assumption, $\overline{R}$ is reduced.
Moreover, $R$ being Noetherian implies that $R$ and $\overline{R}$ have a finite number of minimal prime ideals.
Finally, $\overline{R}$ is a Krull ring \cite[Corollary 2.3]{h76}.
Thus, $\overline{R}$ is a general Krull ring by Lemma~\ref{reduced_strong_Krull}.
\end{proof}

We end this section with a corollary, which
can be applied so that we can construct an easy example
of Krull rings that are not general Krull rings.

\begin{corollary}
Let $K$ be a field, $n \ge 2$ be an integer, $X_1, \ldots, X_n$ be indeterminates over $K$,
$K[X_1, \ldots, X_n]$ be the polynomial ring over $K$, $f \in K[X_1, \ldots, X_n]$ be a nonconstant polynomial,
$R = K[X_1, \ldots, X_n] / (f)$, and $\overline{R}$ be the integral closure of $R$.
Then $\overline{R}$ is a Krull ring. Moreover, the following statements are equivalent.
\begin{enumerate}
\item $\overline{R}$ is Noetherian.

\item $R$ is reduced.

\item $\overline{R}$ is a general Krull ring.

\item $f$ is square-free.
\end{enumerate}
\end{corollary}

\begin{proof}
Clearly, $R$ is Noetherian. Thus, $\overline{R}$ is a Krull ring \cite[Corollary 2.3]{h76}.

(1) $\Leftrightarrow$ (2) $\Leftrightarrow$ (4)  \cite[Corollaries 3.5 and 3.8(5)]{ck21-1}.

(2) $\Rightarrow$ (3) Proposition~\ref{icNoetherian}.

(3) $\Rightarrow$ (4) Let $f = f^{e_1}_1 \cdots f^{e_m}_m$ be the prime factorization of $f$ in $K[X_1, \ldots, X_n]$. Then
$(f) = (f^{e_1}_1) \cap \cdots \cap (f^{e_m}_m)$,
so if we let $R_i = K[X_1, \ldots, X_n] / (f^{e_i}_i)$
and $\overline{R_i}$ be the integral closure of $R_i$
for $i = 1, \ldots, m$, then
$$\overline{R} \cong \overline{R_1} \times \cdots \times \overline{R_m} \quad \text{ and } \quad T(R) = T(R_1) \times \cdots \times T(R_n).$$
By assumption, $\overline{R}$ is a general Krull ring,
and since each $\overline{R_i}$ has a unique minimal prime, $\overline{R_i}$ is either a Krull domain or an SPR
by Theorem \ref{thm5.2}(8).
If $R_i$ is a Krull domain, then $e_i = 1$, while if $R_i$ is an SPR,
then $(f_i)$ must be a maximal ideal of $K[X_1, \ldots, X_n]$, a contradiction.
Thus, $e_1 = \cdots = e_m = 1$, and hence $f$ is square-free.
\end{proof}

\section{Nagata rings} \label{5}

Let $R$ be a ring, $X$ be an indeterminate over $R$, $R[X]$ be the
polynomial ring over $R$, $N_u = \{ f \in R[X] \mid c(f)_u = R \}$,
$N_w = \{ f \in R[X] \mid c(f)_w = R \}$, and $N_v = \{ f \in R[X] \mid f \neq 0$ and $c(f)_v = R \}$,
so $N_u \subseteq N_w \subseteq N_v$. It is clear that if $R$ is an integral domain,
then $N_u = N_w = N_v$ and $R[X]_{N_u}$
is the ($t$-)Nagata ring of $R$.
In this section, we first show that $N_u$ is a saturated multiplicative set of
$R[X]$. We then study the ideal theory of $R[X]_{N_u}$,
which includes the general Krull ring analog
of \cite[Theorem 2.2]{g70} that $R$ is a Krull domain if and only if $R[X]_{N_u}$ is a PID
and the $u$-Noetherian ring analog of \cite[Theorem 2.2]{c05} that $R$ is a $u$-Noetherian
domain if and only if $R[X]_{N_u}$ is a Noetherian domain.

\begin{lemma}
The following statements are satisfied for a ring $R$.
\begin{enumerate}
\item $N_u$ is a saturated multiplicative set of $R[X]$.
\item Each element of $N_u$ is regular. Hence, $R[X]_{N_u}$
is an overring of $R[X]$.
\end{enumerate}
\end{lemma}

\begin{proof}
(1) Let $f, g \in R[X]$. If $f, g \in N_u$, then $c(f)^{n+1} c(g) = c(f)^{n} c(fg)$
for some integer $n \ge 1$ \cite[Theorem 28.1]{3} and $c(f)_u = c(g)_u = R$. Hence, $$(c(f)^{n+1} c(g))_u = (c(f)^{n} c(fg))_u,$$
by which $c(fg)_u = R$. Thus, $fg \in N_u$.
Conversely, if $fg \in N_u$, then
$R = c(fg)_u \subseteq (c(f) c(g))_u \subseteq R$, and hence $(c(f) c(g))_u = R$.
Thus, $c(f)_u = c(g)_u = R$.

(2) Let $f \in N_u$, and assume that $f$ is a zero divisor of $R[X]$.
Then, by McCoy's theorem, there is a nonzero $a \in R$ such that $af = 0$.
Hence, $$(0) = (0)_u = (a c(f))_u \supseteq a c(f)_u = aR,$$ so $aR = (0)$, and hence $a = 0$, a contradiction.
Thus, $f$ is regular.
\end{proof}

\begin{remark} \label{remark6.2}
(1) Let $N_v^r(R) = \{f \in R[X] \mid c(f)$ is regular and $c(f)_v = R\}$.
Then $N_v^r(R) \subseteq N_u$ by Corollary \ref{u-maxiaml}. Moreover, if
$f \in N_u$, then $c(f)_u = R$, and hence $c(f)$ is regular by Lemma \ref{u-inv regular},
so $f \in N_v^r(R)$. Thus, $N_v^r(R) = N_u$. In \cite{7},
Kang studied a couple of ring-theoretic properties of $R[X]_{N_v^r(R)}$.
Kang also used $R[X]_{N_v^r(R)}$ to show that if every regular prime ideal
of $R$ contains a $t$-invertible regular prime ideal, then $R$ is a Krull ring \cite[Theorem 12]{7}.

(2) Let $S = \{f \in R[X] \mid c(f)_v = R\}$. It is worthwhile to note that
(i) $S$ is a saturated multiplicative set of $R[X]$ \cite[Proposition 1]{7},
(ii) $S = R[X]$ if and only if $R= T(R)$, (iii) $N_v \subseteq S$, and equality holds if and only if $R \neq T(R)$, and (iv) $N_v$ is not a multiplicative set if and only if $R =T(R)$ and $R$ is not a field.
\end{remark}

The next result is a $u$-operation analog of the fact that
 that if $R$ is an integral domain, then $\Max (R[X]_{N_u}) = \{P[X]_{N_u} \mid P \in u$-$\Max(R)\}$
and $(IR[X]_{N_u})_t = I_tR[X]_{N_u}$ for a nonzero fractional ideal $I$ of $R$ \cite[Propositions 2.1 and 2.2]{k89-1}.
By Remark \ref{remark6.2}(1), $R[X]_{N_u} = R[X]_{N_v^r(R)}$, so
the second property of Proposition \ref{maximal} is an immediate consequence of \cite[Proposition 5]{7}
that $(IR[X]_{N_v^r(R)})_t = I_tR[X]_{N_v^r(R)}$ for a regular fractional
ideal $I$ of $R$.

\begin{proposition} \label{maximal}
The following statements are satisfied for a ring $R$.
\begin{enumerate}
\item $\Max (R[X]_{N_u}) = \{ P[X]_{N_u} \mid P \in $u-$\Max (R) \}$.

\item If $P \in u$-$\Max (R)$ is regular, then $P[X]_{N_u}$ is a $t$-ideal of $R[X]_{N_u}$.
\end{enumerate}
\end{proposition}

\begin{proof}
(1) ($\supseteq$) Let $P \in u$-$\Max (R)$, and assume that $Q$ is a prime ideal of $R[X]$ with $P[X] \subsetneq Q$.
Choose $f \in Q \setminus P[X]$. Then $(P + c(f))_u = R$, and since $u$ is of finite type,
$((a_0, \ldots, a_n) + c(f))_u = R$ for some $(a_0, \ldots, a_n) \subseteq P$.
Now, let $$g = a_0 + a_1 X + \cdots + a_n X^{n} + X^{n+1}f.$$
Then $g \in Q$ and $c(g)_u = R$, and therefore $Q \cap N_u \neq \emptyset$.
Clearly, $P[X] \cap N_u = \emptyset$. Thus, $P[X]_{N_u} \in \Max (R[X]_{N_u})$.

($\subseteq$) Let $A \in \Max (R[X]_{N_u})$, so $A = M_{N_u}$ for some prime ideal $M$ of $R[X]$.
Let $I = \sum_{f \in M} c(f)$; then $I$ is an ideal of $R$.
If $I_u = R$, then there are some $f_1, \ldots, f_k \in M$ such that $(c(f_1) + \cdots + c(f_k))_u = R$.
Next, let $n_i = \deg (f_i)$ and
\[
  f = f_1 + f_2 X^{n_1 + 1} + f_3 X^{n_1 + n_2 + 2} + \cdots + f_k X^{n_1 + \ldots + n_{k-1} + (k-1)} \,.
\]
Then $f \in M$ and $c(f) = c(f_1) + \cdots + c(f_k)$, so $f \in M \cap N_u$, a contradiction.
Hence, $I_u \subsetneq R$, so there is a maximal $u$-ideal of $R$, say, $P$, such that
$I \subseteq P$. Note that $M \subseteq  I[X] \subseteq P[X]$,
$P[X] \cap N_u = \emptyset$, and $M_{N_u}$ is maximal. Thus, $M_{N_u} = P[X]_{N_u}$.

(2) Let $J$ be a regular fractional ideal of $R$.
It is clear that $J^{-1}R[X]_{N_u} \subseteq \big( JR[X]_{N_u} \big)^{-1}$.
For the reverse containment, let $x \in \big( JR[X]_{N_u} \big)^{-1}$.
Then $xJ \subseteq x \big( JR[X]_{N_u} \big) \subseteq R[X]_{N_u}$, and since $J$ is regular, $x \in T(R)[X]_{N_u}$.
Hence, $x =  \frac{f}{g}$ for some $f \in T(R)[X]$ and $g \in N_u$.
Now, let $a \in J$.
Then $ax = \frac{fa}{g} \in R[X]_{N_u}$, whence $\frac{fa}{g} = \frac{k}{h}$ for some $k \in R[X]$ and $h \in N_u$.
Then $afh = gk$, so by Dedekind-Mertens lemma \cite[Theorem 28.1]{3},
$$a c(f) \subseteq (a c(f))_u = c(afh)_u = c(gk)_u = c(k)_u \subseteq R.$$
Hence, $c(f) \subseteq J^{-1}$, so $x = \frac{f}{g} \in J^{-1}R[X]_{N_u}$.
Therefore, $\big( JR[X]_{N_u} \big)^{-1} \subseteq J^{-1}R[X]_{N_u}$.
Thus, $\big( JR[X]_{N_u} \big)^{-1} = J^{-1}R[X]_{N_u}$.

Now, note that since $P$ is regular, $P \in t$-$\Max (R)$ by Corollary \ref{u-maxiaml}.
Hence, it suffices to show that if $I \subseteq P$ is a finitely generated regular ideal of $R$,
then $\big( IR[X]_{N_u} \big)_v = I_v R[X]_{N_u}$.
But, note that if $I$ is a fractional ideal of $R$, then $I^{-1}$ is always a regular fractional ideal of $R$.
Thus, by the first paragraph, $\big( IR[X]_{N_u} \big)_v = I_v R[X]_{N_u}$.
\end{proof}

\begin{corollary} \label{picard group}
Every finitely generated locally principal ideal of $R[X]_{N_u}$ is principal. Hence, $\Pic (R[X]_{N_u}) = \{0\}$.
In particular, if $R$ satisfies Property(A), then $\Cl (R[X]_{N_u}) = \{ 0 \}$.
\end{corollary}

\begin{proof}
The proof is the same as that of \cite[Theorem 2]{7} with Proposition~\ref{maximal}.
In particular, if $A$ is a $t$-invertible regular ideal of $R[X]_{N_u}$,
then $AA^{-1} \nsubseteq P[X]_{N_u}$ for all $P \in u$-$\Max (R)$ by Property(A) and Proposition~\ref{maximal}(2),
whence $A$ is invertible.
Thus, $\Cl ( R[X]_{N_u}) = \Pic (R[X]_{N_u}) = \{ 0 \}$.
\end{proof}

The following lemma is a very useful tool for studying when $R[X]_{N_u}$ is a (regular) PIR or a Noetherian ring
 (see \cite[Lemma 2.1]{c05} for a domain case).

\begin{lemma} \label{u-closure}
The following statements hold for an ideal $I$ of a ring $R$.
\begin{enumerate}
\item $IR[X]_{N_u} \cap T(R) = I_u$.

\item $I_u R[X]_{N_u} = IR[X]_{N_u}$.

\item $I_u$ is of finite type if and only if $IR[X]_{N_u}$ is finitely generated.
\end{enumerate}
\end{lemma}

\begin{proof}
(1) ($\supseteq$) Let $a \in I_u$. Then there exists a $J \in \textnormal{rGV}(R)$ such that $aJ \subseteq I$.
It is clear that $JR[X] \cap N_u \neq \emptyset$, so $a \in aR[X]_{N_u} = aJR[X]_{N_u} \subseteq IR[X]_{N_u}$.
Thus, $I_u \subseteq IR[X]_{N_u} \cap T(R)$.

($\subseteq$) Let $b \in IR[X]_{N_u} \cap T(R)$. Then $b = \frac{f}{g}$ for some $f \in IR[X]$ and $g \in N_u$.
Hence, $bg = f$ by which we have
\[
  b \in bR = b c(g)_u \subseteq (b c(g))_u = c(f)_u \subseteq I_u \,.
\]
Thus, $IR[X]_{N_u} \cap T(R) \subseteq I_u$.

(2) This follows directly from (1).

(3) Assume that $I_u$ is of finite type, and let $J$ be a finitely generated ideal of $R$ such that
$I_u = J_u$. Then, by (2), $$IR[X]_{N_u} = I_uR[X]_{N_u} = J_uR[X]_{N_u} = JR[X]_{N_u}.$$
Thus, $IR[X]_{N_u}$ is finitely generated. Conversely, suppose that
$IR[X]_{N_u}$ is finitely generated. Then
$IR[X]_{N_u} = (f_1, \dots , f_n)R[X]_{N_u}$ for some $f_1, \dots , f_n \in IR[X]$.
Let $J = c(f_1) + \cdots + c(f_n)$. Then $J$ is
a finitely generated ideal of $R$ and $IR[X]_{N_u} = JR[X]_{N_u}$.
Hence, by (1), $$I_u = IR[X]_{N_u} \cap T(R) = JR[X]_{N_u} \cap T(R) = J_u.$$
Thus, $I_u$ is of finite type.
\end{proof}

A Krull ring does not satisfy Property(A) in genral. For example, let $D$ be a Krull domain, $S$ be a
total quotient ring that does not satisfy Property(A) (see, for example, \cite[Example 2 on page 174]{4}),
and $R=D \times S$. Then $R$ is a Krull ring but $R$ does not satisfy Property(A)
by Corollary \ref{direct product Property(A)}.
The following theorem gives a new characterization of Krull rings with Property(A)
via the ring $R[X]_{N_u}$.

\begin{theorem} \label{Krull ring Property(A)}~
The following statements are equivalent for a ring $R$ with Property(A).
\begin{enumerate}
\item $R$ is a Krull ring.

\item $R[X]_{N_u}$ is a Krull ring.

\item $R[X]_{N_u}$ is a regular $\pi$-ring.

\item $R[X]_{N_u}$ is a factorial ring.

\item $R[X]_{N_u}$ is a Dedekind ring.

\item $R[X]_{N_u}$ is a regular PIR.
\end{enumerate}
\end{theorem}

\begin{proof}
(1) $\Rightarrow$ (6) Let $A_{N_u}$ be a regular ideal of $R[X]_{N_u}$ and $I = \sum_{f \in A} c(f)$
for an ideal $A$ of $R[X]$.
Then $I$ is regular because $A \subseteq IR[X]$, $A$ is regular, and $R$ satisfies Property(A).
Hence, $I$ is $u$-invertible by (1) and Corollary~\ref{coro4.7}.
Thus, $I_u = J_u$ for some finitely generated regular ideal $J$ of $R$ with $J \subseteq I$
by Proposition \ref{prop1.5}.
Since $J$ is finitely generated, there are some polynomials
$f_1, \ldots, f_m \in A$ so that $J \subseteq c(f_1) + \cdots + c(f_m)$.
Now, let $k_i = \deg (f_i)$ and
\[
  f = f_1 + f_2 X^{k_1 + 1} + \cdots + f_m X^{k_1 + k_2+ \cdots + k_{m-1} + (m-1)} \,,
\]
so $c(f) = c(f_1) + \cdots + c(f_m)$. Then, by Lemma~\ref{u-closure},
\[
  A_{N_u} \subseteq IR[X]_{N_u} = I_u R[X]_{N_u} = J_u R[X]_{N_u} = JR[X]_{N_u} \subseteq c(f) R[X]_{N_u} \,.
\]
Note that $c(f)$ is a regular ideal of $R$, so $c(f)$ is $u$-invertible by (1)
and Corollary \ref{coro4.7}. Hence, $c(f)R_P$ is a regular principal ideal
of $R_P$, which implies $$(fR[X]_{N_u})_{P[X]_{N_u}} = (c(f)R[X]_{N_u})_{P[X]_{N_u}},$$
for all $P \in u$-Max$(R)$. Thus, by Proposition~\ref{maximal}(1),
$c(f) R[X]_{N_u} = fR[X]_{N_u} \subseteq A_{N_u},$ and hence
$fR[X]_{N_u} = A_{N_u}$.
Therefore, $R[X]_{N_u}$ is a regular PIR.

(6) $\Rightarrow$ (5) $\Rightarrow$ (3) $\Rightarrow$ (2) Clear.

(6) $\Rightarrow$ (4) $\Rightarrow$ (2) Clear.

(2) $\Rightarrow$ (1) Let $I$ be a regular ideal of $R$.
Then $IR[X]_{N_u}$ is a regular ideal of $R[X]_{N_u}$, and hence $IR[X]_{N_u}$ is $t$-invertible by assumption.
Note that $(IR[X]_{N_u})^{-1} = I^{-1}R[X]_{N_u}$ by the proof of Proposition~\ref{maximal}(2), so
$$(IR[X]_{N_u})(IR[X]_{N_u})^{-1} = (II^{-1})R[X]_{N_u}.$$
Moreover, $II^{-1}$ is regular.
Hence, by Proposition~\ref{maximal}, $II^{-1} \nsubseteq P$ for all $P \in u$-$\Max (R)$.
Since the $u$-operation is of finite type, $(II^{-1})_u = R$ by Proposition \ref{prop1.2}(2).
Thus, $R$ is a Krull ring by Corollary \ref{coro4.7}.
\end{proof}

\begin{corollary}
If $R$ is the integral closure of a Noetherian ring, then
$R[X]_{N_u}$ is a regular PIR.
\end{corollary}

\begin{proof}
A Noetherian ring satisfies Property(A) \cite[Theorem 82]{8}
and each overring of a ring with Property(A) satisfies Property(A) \cite[Corollary 2.6]{4}.
Hence, $R$ satisfies Property(A).
Moreover, $R$ is a Krull ring \cite[Corollary 2.3]{h76}.
Thus, $R[X]_{N_u}$ is a regular PIR by Theorem~\ref{Krull ring Property(A)}.
\end{proof}

Let $N = \{ f \in R[X] \mid c(f) = R \}$ and $W = \{ f \in R[X] \mid f$ is a monic polynomial$\}$.
Then $N$ and $W$ are both multiplicative sets of $R[X]$ and $W \subseteq N \subseteq N_u$.
Hence, we have the following two types of quotient rings;
\[
  R(X) = R[X]_N \quad \mbox{ and } \quad R \langle X \rangle = R[X]_W \,.
\]

\noindent
For some divisibility properties of $R(X)$ and $R \langle X \rangle$,
see \cite[Section 18]{4}.

We next state an interesting characterization of general Krull rings,
which is motivated by \cite[Theorem 18.8]{4} and \cite[Theorem 2.2]{g70} that (i) $R$ is a general ZPI-ring
if and only if $R(X)$ is a PIR and (ii)
$R$ is a Krull domain if and only if $R[X]_{N_u}$ is a PID, respectively.

\begin{theorem} \label{SKR[X]}
The following statements are equivalent for a ring $R$.
\begin{enumerate}
\item $R$ is a general Krull ring.

\item $R[X]_{N_u}$ is a general Krull ring.

\item $R[X]_{N_u}$ is a $\pi$-ring.

\item $R[X]_{N_u}$ is a UFR.

\item $R[X]_{N_u}$ is a general ZPI-ring.

\item $R[X]_{N_u}$ is a PIR.

\item $R \langle X \rangle$ is a general Krull ring.

\item $R(X)$ is a general Krull ring.
\end{enumerate}
\end{theorem}

\begin{proof}
(1) $\Rightarrow$ (6) Since $R$ is a general Krull ring,
$R = R_1 \times \cdots \times R_n$ for some
Krull domains and SPRs $R_1, \dots, R_n$.
Let $N_{u_i} = \{f \in R_i[X] \mid c(f)_u = R_i\}$ for $i =1, \dots , n$.
Then, by Lemma~\ref{u-diect product},
\[
  R[X]_{N_u} = R_1 [X]_{N_{u_1}} \times \cdots \times R_n [X]_{N_{u_n}}  \,
\]
and $R_i [X]_{N_{u_i}}$ is a PID (when $R_i$ is a Krull domain) \cite[Theorem 2.2]{g70} or an SPR
(when $R_i$ is an SPR) \cite[Lemma 18.7]{4}.
Thus, $R[X]_{N_u}$ is a PIR.

(6) $\Rightarrow$ (4) $\Rightarrow$ (3) $\Rightarrow$ (2) Clear.

(6) $\Rightarrow$ (5) $\Rightarrow$ (3) Clear.

(2) $\Rightarrow$ (1) A general Krull ring is a Krull ring and $R[X]_{N_u} \cap T(R) = R$,
so $R$ is integrally closed and there exists an irredundant primary decomposition of $(0)$ in $R$,
say, $(0) = Q_1 \cap \cdots \cap Q_n$,
such that each $\sqrt{Q_i}$ is a minimal prime ideal of $R$.
Hence, $R \cong R/Q_1 \times \cdots \times R/Q_n.$ Moreover,
$\dim \big( T(R/Q_i) \big) = 0$ because $Q_i$ is primary for $i = 1, \dots, n$,
so $\dim (T(R)) = 0$ by Lemma~\ref{direct product}.
Hence, $R$ satisfies Property(A) \cite[Corollaries 2.6 and 2.12]{4}.

Now, let $Q$ be a prime ideal of $R[X]_{N_u}$.
If $Q$ is minimal, then $Q \cap R$ is a minimal prime ideal of $R$ and $Q = (Q \cap R)[X]_{N_u}$.
Next, assume that $Q$ is not minimal.
Then $Q$ is regular, $Q \subseteq P[X]_{N_u}$ for some $P \in u$-$\Max (R)$ by Proposition~\ref{maximal},
and $P$ is regular because $R$ satisfies Property(A).
Note that $R[X]_{N_u}$ is general Krull by assumption
and $P[X]_{N_u} \in X^1_r(R[X]_{N_u})$ by Proposition \ref{maximal} and Corollary \ref{coro4.7},
hence ht$(P[X]_{N_u}) =1$ by Theorem \ref{thm5.2}(5).
This implies that $Q = P[X]_{N_u}$ and $\h Q = 1$.
In particular, $R[X]_{N_u}$ is a general ZPI-ring by Theorem \ref{chracter-cg}(2).
Let $a \in R$.
Then $aR[X]_{N_u}$ is a finite direct product of prime ideals, and hence
\[
  aR[X]_{N_u} = \big( P_1 [X]_{N_u} \big) \cdots \big( P_n [X]_{N_u} \big) = (P_1 \cdots P_n)R[X]_{N_u}
\]
for some prime ideals $P_1, \ldots, P_n$ of $R$. Thus, by Lemma~\ref{u-closure},
\[
  (aR)_u = \big( aR[X]_{N_u} \big) \cap T(R) = (P_1 \cdots P_n)R[X]_{N_u} \cap T(R) = (P_1 \cdots P_n)_u \,.
\]
Hence, $R$ is a general Krull ring by Theorem~\ref{s-krull}.

(1) $\Rightarrow$ (7) If $R$ is a general Krull ring, then $R = R_1 \times \cdots \times R_n$ is a finite direct product of Krull domains and SPRs.
Note that $R[X] = R_1 [X] \times \cdots \times R_n [X]$, so $R \langle X \rangle = R_1 \langle X \rangle \times \cdots \times R_n \langle X \rangle$.
If $R_i$ is a Krull domain, then $R_i \langle X \rangle$ is a Krull domain \cite[Theorem 5.2]{aam85}.
If $R_i$ is an SPR, then $R_i \langle X \rangle$ is an SPR \cite[Lemma 18.7]{4}.
Thus, $R \langle X \rangle$ is a general Krull ring.

(7) $\Rightarrow$ (8) Note that $W \subseteq N$, so $R(X) = R[X]_N = (R[X]_W)_N = (R \langle X \rangle)_N$.
Since $R \langle X \rangle$ is a general Krull ring, the result follows by Corollary~\ref{local-s-Krull}.

(8) $\Rightarrow$ (2) Note that $N \subseteq N_u$, so $R[X]_{N_u} = (R[X]_N)_{N_u} = (R(X))_{N_u}$.
Thus, the assertion follows directly from Corollary~\ref{local-s-Krull}.
\end{proof}

As in \cite{wm97}, we say that a $w$-Noetherian domain is a {\it strong Mori domain}
(an SM domain). Hence, a $u$-Noetherian domain is just an SM domain.
It is known that an integral domain $R$ is an SM domain if and only if $R[X]$ is an SM domain,
if and only if $R[X]_{N_u}$ is a Noetherian domain \cite[Theorem 2.2]{c05}.
A PIR is a Noetherian ring, so if $R$ is a general Krull ring, then $R[X]_{N_u}$ is a Noetherian ring by Theorem \ref{SKR[X]}.
We next study when $R[X]$ and $R[X]_{N_u}$ are Noetherian.

\begin{theorem} \label{u-Noetherian}
The following statements are equivalent for a ring $R$.
\begin{enumerate}
\item $R$ is a $u$-Noetherian ring.

\item $R$ is a $w$-Noetherian ring.

\item $R[X]$ is a $u$-Noetherian ring.

\item $R[X]$ is a $w$-Noetherian ring.

\item $R[X]_{N_u}$ is a Noetherian ring.

\item $R[X]_{N_w}$ is a Noetherian ring.
\end{enumerate}
In this case, $u = w$ on $R$.
\end{theorem}

\begin{proof}
(1) $\Rightarrow$ (2) This follows because $u \le w$ by Proposition \ref{u-w-w'-oper}.

(2) $\Rightarrow$ (1) If $R$ is $w$-Noetherian, then $R$ satisfies Property(A) \cite[Corollary 6.8.24]{wk16},
and hence $u = w$ on $R$ by Proposition~\ref{u-w-w'-oper}.
Thus, $R$ is $u$-Noetherian.

(2) $\Leftrightarrow$ (4) $\Leftrightarrow$ (6) \cite[Theorem 6.8.8]{wk16}.

(3) $\Leftrightarrow$ (4) This follows directly from the equivalence of (1) and (2) above.

(4) $\Rightarrow$ (5) If $R[X]$ is $w$-Noetherian, then $R$ is $w$-Noetherian
by the equivalence of (2) and (4) above,
whence $u = w$ on $R$ and $N_u = N_w$. Thus, $R[X]_{N_u}$ is Noetherian
by the equivalence of (4) and (6).

(5) $\Rightarrow$ (1) Let $I$ be a $u$-ideal of $R$.
Then $IR[X]_{N_u}$ is finitely generated by assumption.
Thus, $I$ is of finite type by Lemma~\ref{u-closure}(3).
\end{proof}

A $w$-Noetherian ring is a $u$-Noetherian ring by Theorem \ref{u-Noetherian}.
Hence, a finite direct product of $u$-Noetherian rings is a $u$-Noetherian ring
\cite[Proposition 4.8]{ywzc11}. The next corollary shows that the converse of \cite[Proposition 4.8]{ywzc11}
is also true.

\begin{corollary} \label{coro6.10}
Let $R= R_1 \times R_2$ be a direct product of rings $R_1$ and $R_2$.
Then $R$ is a $u$-Noetherian ring if and only if $R_i$ is a $u$-Noetherian ring for $i=1,2$.
\end{corollary}

\begin{proof}
Let $N_{u_i} = \{f \in R_i[X] \mid c(f)_u = R_i\}$ for $i=1, 2$.
Then $$R[X]_{N_u} = R_1[X]_{N_{u_1}} \times R_2[X]_{N_{u_2}}$$ (see the proof of Theorem \ref{SKR[X]}). Thus,
$R$ is $u$-Noetherian if and only if $R[X]_{N_u}$ is Noetherian by Theorem \ref{u-Noetherian},
if and only if each $R_i[X]_{N_{u_i}}$ is Noetherian, if and only if
each $R_i$ is $u$-Noetherian.
\end{proof}

The following corollary is an analog of Theorem \ref{SKR[X]}, which characterizes
general Krull rings via the $w$-Nagata ring $R[X]_{N_w}$.

\begin{corollary} \label{coro6.11}
The following statements are equivalent for a ring $R$.
\begin{enumerate}
\item $R$ is a general Krull ring.

\item $R[X]_{N_w}$ is a general Krull ring.

\item $R[X]_{N_w}$ is a $\pi$-ring.

\item $R[X]_{N_w}$ is a UFR.

\item $R[X]_{N_w}$ is a general ZPI-ring.

\item $R[X]_{N_w}$ is a PIR.
\end{enumerate}
\end{corollary}

\begin{proof}
(1) $\Rightarrow$ (6) By Corollary~\ref{T(R) PIR},
$R$ is a $u$-Noetherian ring, and hence $u = w$ on $R$ by Theorem~\ref{u-Noetherian}.
Thus, $R[X]_{N_w}$ is a PIR by Theorem~\ref{SKR[X]}.

(6) $\Rightarrow$ (5) $\Rightarrow$ (2) Clear.

(6) $\Rightarrow$ (4) $\Rightarrow$ (3) $\Rightarrow$ (2) Clear.

(2) $\Rightarrow$ (1) Note that each ideal of $R[X]_{N_w}$ is a $w$-ideal \cite[Theorem 6.6.18]{wk16},
so $d = w = u$ on $R[X]_{N_w}$.
Hence, by Corollary~\ref{T(R) PIR}, $R[X]_{N_w}$ is Noetherian,
and thus $u = w$ on $R$ by Theorem~\ref{u-Noetherian}.
Thus, $R$ is a general Krull ring by Theorem~\ref{SKR[X]}.
\end{proof}

It is known that an integral domain $R$ is a strong Mori domain
if and only if $R_P$ is Noetherian for all $P \in w$-$\Max (R)$
and each nonzero element of $R$ is contained in only finitely
 many maximal $w$-ideals of $R$ \cite[Theorem 1.9]{wm99}.

\begin{proposition}
A ring $R$ is a $u$-Noetherian ring if and only if $R$ satisfies the following three properties;
\begin{enumerate}
\item [(i)] $R_P$ is a Noetherian ring for all $P \in u$-$\Max (R)$,
\item [(ii)] every prime $Z$-ideal of $R$ is of finite type, and
\item [(iii)] each regular element of $R$ is contained in only finitely many maximal $u$-ideals of $R$.
\end{enumerate}
\end{proposition}

\begin{proof}
$(\Rightarrow)$ By Lemma \ref{local-properties}(1), (i) is satisfied. (ii) is clear.
For $a \in \reg(R)$, let $\Lambda = \{P \in u\textnormal{-} \hspace{-2pt}\Max(R) \mid a \in P\}$.
If $\Lambda$ is infinite, then we can choose a countably infinite subset $\{P_i \mid i =1, 2, \dots \}$ of $\Lambda$.
Now, let $I_k = P_1 \cap \cdots \cap P_k$ for all integers $k \geq 1$. Then
$a \in I_k$, $(I_k)_v = I_k$, and $I_{k+1} \subsetneq I_{k}$
 by Corollary \ref{u-maxiaml} and the assumption that
$R$ is $u$-Noetherian.
Hence, $a(I_{k})^{-1} \subsetneq a(I_{k+1})^{-1} \subsetneq R$ for all $k \geq 1$,
which is an infinite ascending chain of $u$-ideals of $R$,
a contradiction. Thus, (iii) is satisfied.

$(\Leftarrow)$ By (ii), Proposition \ref{regLocal}(1), and Cohen's theorem,
$T(R)$ is a Noetherian ring. Hence, $T(R)$ satisfies Property(A).
Thus, $R$ satisfies Property(A) \cite[Corollary 2.6]{4} and
$u = w$ on $R$ by Proposition \ref{u-w-w'-oper}(2). Hence, it suffices to show that
each prime $u$-ideal of $R$ is of finite type \cite[Theorem 6.8.5]{wk16}.

Let $P$ be a maximal $u$-ideal of $R$.
Then $R_P$ is a Noetherian ring by (i), and hence $R_P(X)$ is a Noetherian ring.
Thus, $R[X]_{N_u}$ is a locally Noetherian ring by Proposition \ref{maximal}(1).
Now, let $I$ be a prime ideal of $R$. If $I$ is a regular ideal of $R$, then $IR[X]_{N_u}$
is contained in only finitely many maximal ideals of $R[X]_{N_u}$ by (iii) and Proposition \ref{maximal},
and since $R[X]_{N_u}$ is locally Noetherian, $IR[X]_{N_u}$ is finitely generated.
Thus, $I_u$ is of finite type by Lemma \ref{u-closure}. Next, assume that $I \subseteq Z(R)$.
Then $I$ is a $u$-ideal of finite type by (ii) and Proposition \ref{regLocal}(2).
\end{proof}

\section{Almost Dedekind rings} \label{6}

A general Krull ring is a $u$-Noetherian ring by Corollary \ref{T(R) PIR}(2)
or by Theorems \ref{SKR[X]} and \ref{u-Noetherian}. It is known that
an integral domain is a Krull domain
if and only if it is an integrally closed $u$-Noetherian ring \cite[Theorem 2.8]{wm99}.
In this section, we study when a $u$-Noetherian ring is a general Krull ring.

\begin{theorem} \label{u-Noetherian strong Krull}
The following statements are equivalent for a ring $R$.
\begin{enumerate}
\item $R$ is a general Krull ring.

\item $R$ is a $u$-Noetherian ring such that $R_P$ is a DVR or an SPR for all maximal $u$-ideals $P$ of $R$.

\item $R$ is a $w$-Noetherian ring such that $R_P$ is a DVR or an SPR for all maximal $w$-ideals $P$ of $R$.
\end{enumerate}
\end{theorem}

\begin{proof}
(1) $\Rightarrow$ (2) Clearly, $R$ is a $u$-Noetherian ring.
Now, let $Q$ be a maximal $u$-ideal of $R$ and $R = R_1 \times \cdots \times R_n$ be a finite direct product of Krull domains and SPRs
$R_1, \dots, R_n$. Then, by Lemma \ref{direct product}(3), there exists a prime ideal $P_i$ of $R_i$ such that $Q = R_1 \times \cdots \times P_i \times \cdots \times R_n$.
Recall that $Q = R_1 \times \cdots \times (P_i)_u \times \cdots \times R_n$
by Lemma \ref{u-diect product}, so $(P_i)_u = P_i$.
Thus, if $R_i$ is a Krull domain, then $R_Q = (R_i)_{P_i}$ is a DVR.
Next, if $R_i$ is an SPR, then $R_Q = (R_i)_{P_i} = R_i$ is an SPR.

(2) $\Rightarrow$ (1) By Corollary~\ref{T(R)_Noetherian}, $T(R)$ is Noetherian, so there is an irredundant primary decomposition of
$(0)$ in $T(R)$, say, $(0) = Q_1 \cap \cdots \cap Q_n$.
If $Q_i + Q_j \subseteq M$ for some maximal ideal $M$ of $T(R)$,
then $M \cap R$ is a $u$-ideal of $R$ by Proposition~\ref{regLocal}(2) and $T(R)_M = R_{M \cap R}$.
Hence, $T(R)_M$ is a DVR or an SPR, so $(Q_i)_M$ and $(Q_j)_M$ are comparable under inclusion.
Thus, $Q_i$ and $Q_j$ are comparable, a contradiction. Hence, $Q_1, \ldots, Q_n$ are comaximal, $\dim (T(R)) =0$, and
\begin{eqnarray*}
R &\hookrightarrow& R / (Q_1 \cap R) \times \cdots \times R / (Q_n \cap R) \\
  &\hookrightarrow& T(R) / Q_1 \times \cdots \times T(R) / Q_n \\
 &=& T(R).
\end{eqnarray*}
Now, let $x \in T(R)$ be integral over $R$. Note that, for each $P \in u$-$\Max (R)$,
$$R_P \hookrightarrow T(R)_P \hookrightarrow T(R_P)$$ and $x$ is integral over $R_P$.
Since $R_P$ is a DVR or an SPR by assumption, $x \in R_P$ for all $P \in u$-$\Max (R)$.
Next, let $I = \{ a \in R \mid ax \in R \}$. Then $I_u = I$ and $I \nsubseteq P$ for all $P \in u$-$\Max (R)$, so $I = R$.
Thus, $x \in R$.

Clearly, $R/ (Q_1 \cap R) \times \cdots \times R / (Q_n \cap R)$ is a finitely generated $R$-module,
and since $R$ is integrally closed by the previous paragraph,
\[
  R = R / (Q_1 \cap R) \times \cdots \times R / (Q_n \cap R) \,.
\]
Let $P_i = \sqrt{Q_i \cap R}$ for $i=1, \dots , n$.
Assume that $P_i$ is not a maximal ideal of $R$,
and let $M$ be a maximal ideal of $R$ with $P_i \subsetneq M$. Choose $a \in M \setminus P_i$,
and let $P$ be a minimal prime ideal of $(Q_i \cap R) +aR$ that is
contained in $M$. Then $P$ is a minimal regular prime ideal of $R$,
so $P$ is a prime $u$-ideal of $R$. By assumption, $R_P$ is a DVR, which implies that $Q_i \cap R$ is a prime ideal.
In this case, $R/ (Q_i \cap R)$ is an integrally closed $u$-Noetherian domain
by assumption and Corollary \ref{coro6.10}. Thus, $R/ (Q_i \cap R)$ is a Krull domain.
Next, if $P_i$ is a maximal ideal of $R$, then $R/ (Q_i \cap R) \cong R_{P_i}/{(Q_i \cap R)R_{P_i}}$,
and since $R_{P_i}$ is an SPR by assumption, $R / (Q_i \cap R)$ is an SPR.
Thus, $R$ is a finite direct product of Krull domains and SPRs.

(2) $\Leftrightarrow$ (3) This follows from Theorem~\ref{u-Noetherian}.
\end{proof}

A Noetherian ring is a $u$-Noetherian ring, so by Theorem \ref{u-Noetherian strong Krull},
we have

\begin{corollary}
A Noetherian ring $R$ is a general Krull ring if and only if
$R_P$ is a DVR or an SPR for all maximal $u$-ideals $P$ of $R$.
\end{corollary}

The next corollary is a general ZPI-ring analog
of a Dedekind domain that an integral domain $R$ is a Dedekind domain if and only if $R$
is Noetherian and $R_M$ is a DVR for all $M \in$ Max$(R)$ \cite[Theorem 13.1]{f73}.

\begin{corollary} \label{general zpi}
A ring $R$ is a general ZPI-ring if and only if $R$ is Noetherian and $R_M$
is a DVR or an SPR for all $M \in \Max(R)$.
\end{corollary}

\begin{proof}
A general ZPI-ring is a general Krull ring with (Krull) dimension $\leq 1$
and a $u$-Noetherian ring with (Krull) dimension $\leq 1$ is a Noetherian ring.
Thus, the result follows from Theorem \ref{u-Noetherian strong Krull}.
\end{proof}

We will say that $R$ is an {\it almost Dedekind ring} (resp., a {\it $u$-almost Dedekind ring})
if $R_M$ is a DVR or an SPR for all maximal ideals (resp., maximal $u$-ideals) $M$ of $R$.
Hence, by Corollary \ref{general zpi} (resp., Theorem \ref{u-Noetherian strong Krull}),
$R$ is a general ZPI-ring (resp., general Krull ring) if and only if $R$ is a Noetherian
almost Dedekind ring  (resp., a $u$-Noetheria $u$-almost Dedekind ring).
Moreover, it is easy to see that an integral domain is an almost Dedekind ring
(resp., a $u$-almost Dedekind ring) if and only if it is an almost Dedekind domain \cite[Section 36]{3}
(resp., a $t$-almost Dedekind domain \cite[Section IV]{k89-1});
an almost Dedekind ring is a $u$-almost Dedekind ring;
and a $u$-almost Dedekind ring is integrally
closed by the proof of Theorem \ref{u-Noetherian strong Krull}.

\begin{lemma} \label{lemma-74}
Let $R = R_1 \times \cdots \times R_n$ be a finite direct product of rings $R_1, \dots , R_n$.
Then $R$ is an almost Dedekind ring $($resp., a $u$-almost Dedekind ring$)$
if and only if $R_i$ is an almost Dedekind ring $($resp., a $u$-almost Dedekind ring$)$ for $i =1, \dots, n$.
\end{lemma}

\begin{proof}
This follows directly from the following three observations:
(i) $Q$ is a prime ideal of $R$ if and only if
$Q = R_1 \times \cdots \times P_i \times \cdots \times R_n$ for some prime ideal $P_i$ of $R_i$;
(ii) $R_Q = (R_i)_{P_i}$ by Lemma \ref{direct product}(3),
and (iii) $Q$ is a $u$-ideal if and only if $P_i$ is a $u$-ideal by Lemma \ref{u-diect product}(2).
\end{proof}

The next proposition is an almost Dedekind ring analog of a general ZPI-ring.
The (3) $\Leftrightarrow$ (5) $\Rightarrow$ (2) of Proposition \ref{almost dedekind}
holds without the assumption that $R$ is an almost Dedekind ring (see \cite[Theorem 10]{ac11}).

\begin{proposition} \label{almost dedekind}
Let $R$ be an almost Dedekind ring. Then the following statements are equivalent.
\begin{enumerate}
\item $R$ is a finite direct product of almost Dedekind domains and SPRs.
\item $R$ has a finite number of minimal prime ideals.
\item The zero element of $R$ is a finite product of prime elements.
\item $R$ has few zero divisors.
\item Each minimal prime ideal of $R$ is principal.
\end{enumerate}
\end{proposition}

\begin{proof}
(1) $\Rightarrow$ (3) and (5).
It is easy to check that it is indeed true of a ring
that is a finite direct product of integral domains and SPRs (cf. the proof of Theorem \ref{thm5.2}).

(3) $\Rightarrow$ (2) Clear.

(2) $\Rightarrow$ (1) and (4). Let $P_1, \dots , P_n$ be the minimal prime ideals of $R$.
If $x \in R \setminus \bigcup_{i=1}^nP_i$, then $x$ is regular in $R_M$,
because $R_M$ is a DVR or an SPR, for all
maximal ideals $M$ of $R$. Hence, if we let $I = \{a \in R \mid ax =0\}$,
then $IR_M = (0)$ for all maximal ideals $M$ of $R$, which implies that $I = (0)$
and $x$ is regular in $R$. Hence, $Z(R) = \bigcup_{i=1}^nP_i$. Note that
$R_{P_i}$ is either a field or an SPR,
so $P_i = P_i^2$ (resp., $P_i^{k_i-1} \supsetneq P_i^{k_i} = P_i^{k_i+1}$ for some $k_i \geq 2$)
if $R_{P_i}$ is a field (resp., not a field).
Without loss of generality, we may assume that
$R_{P_1}, \dots, R_{P_m}$ are fields and $R_{P_{m+1}}, \dots, R_{P_n}$ are not fields. Then
$P_iR_{P_i} = (0)$ for $i =1, \dots , m$ and $P_i^{k_i}R_{P_i} = (0)$ for $i= m+1, \dots , n$,
and hence
\[
  P_1 \cap \cdots \cap P_m \cap P_{m+1}^{k_{m+1}} \cap \cdots \cap P_n^{k_n} = (0) \, \mbox{ in } \, R \,.
\]
It is clear that each pair of distinct minimal prime ideals of $R$ is comaximal. Hence, by Chinese Remainder Theorem,
$$R = R/P_1 \times \cdots \times R/P_m \times R/P_{m+1}^{k_{m+1}} \times \cdots \times R/P_{n}^{k_{n}} \,.$$
Then, by Lemma \ref{lemma-74}, $R/P_i$ is an almost Dedekind domain for $i =1, \dots, m$
and $R/P_{i}^{k_{i}}$ is an SPR for $i = m+1, \dots , n$.
Thus, the proof is completed.

(4) $\Rightarrow$ (2) Let $Q_1, \dots, Q_n$ be prime ideals of $R$ such that
$Z(R) = \bigcup_{i=1}^nQ_i$. If $P$ is a minimal prime ideal of $R$, then
$P \subseteq Q_i$ for some $i$, and since $R$ is an almost Dedekind ring,
either $P=Q_i$ or $P$ is a unique prime
ideal of $R$ with $P \subsetneq Q_i$. Thus, the number of minimal prime
ideals of $R$ is at most $n < \infty$.

(5) $\Rightarrow$ (2) This follows from \cite[Theorem 1.6]{gh93} that
if each minimal prime ideal of an ideal $I$ is the radical of a finitely
generated ideal, then $I$ has finitely many minimal prime ideals.
\end{proof}

It is known that an integral domain $R$ is an almost Dedekind domain (resp., a $t$-almost Dedekind domain)
if and only if  $R(X)$ (resp., $R[X]_{N_v}$)
is an almost Dedekind domain \cite[Proposition 36.7]{3} (resp., \cite[Theorem 4.4]{k89-1}).

\begin{proposition} \label{prop7.7}
A ring $R$ is an almost Dedekind ring (resp., a $u$-almost Dedekind ring)
if and only if $R(X)$ (resp., $R[X]_{N_u}$) is an almost Dedekind ring.
\end{proposition}

\begin{proof}
This follows directly from the following observation:
(i) $\Max (R[X]_{N_u}) = \{PR[X]_{N_u} \mid P \in u$-$\Max(R)\}$ and $\Max (R(X)) = \{MR(X) \mid M \in \Max (R)\}$ \cite[Proposition 33.1]{3},
(ii) $R(X)_{MR(X)}$ = $R_M(X)$ and $(R[X]_{N_u})_{PR[X]_{N_u}} = R_P(X)$,
(iii) $R_P$ is a DVR if and only if $R_P(X)$ is a DVR, and
(iv) $R_P$ is an SPR if and only if $R_P(X)$ is an SPR \cite[Lemma 18.7]{4}.
\end{proof}

A general Krull ring is a $u$-almost Dedekind ring with
a finite number of minimal prime ideals by Theorem \ref{u-Noetherian strong Krull}.
We next characterize the $u$-almost Dedekind ring that has a finite number of minimal prime ideals.
This is a generalization of the result of Proposition \ref{almost dedekind} to $u$-almost Dedekind rings.

\begin{corollary} \label{u-almost dedekind}
Let $R$ be a $u$-almost Dedekind ring.
Then the following statements are equivalent.
\begin{enumerate}
\item $R$ is a finite direct product of $t$-almost Dedekind domains and SPRs.
\item $R$ has a finite number of minimal prime ideals.
\item The zero element of $R$ is a finite product of prime elements.
\item $R$ has few zero divisors.
\item Each minimal prime ideal of $R$ is principal.
\end{enumerate}
\end{corollary}

\begin{proof}
(1) $\Rightarrow$ (3) $\Rightarrow$ (2)
and (1) $\Rightarrow$ (5) $\Rightarrow$ (2).
These can be proved by the same arguments
as the proofs of the counterparts in Proposition \ref{almost dedekind}.

(2) $\Rightarrow$ (1) Let $\{P_1, \dots, P_m, P_{m+1}, \dots , P_n\}$ be the set of minimal
prime ideals of $R$ such that
$R_{P_i}$ is a field for $i=1, \dots, m$ and $R_{P_i}$ is an SPR for $i = m+1, \dots, n$.
Then, by an argument similar to the proof of Proposition \ref{almost dedekind},
$$R \hookrightarrow R/P_1 \times \cdots \times R/P_m \times R/P_{m+1}^{k_{m+1}} \times \cdots \times R/P_{n}^{k_{n}}
\hookrightarrow T(R)$$
for some integers $k_i \geq 2$, $i=m+1, \dots , n$.
Note that $R$ is integrally closed and $R/P_1 \times \cdots \times R/P_{n}^{k_{n}}$
is a finitely generated $R$-module, so
$$R = R/P_1 \times \cdots \times R/P_m \times R/P_{m+1}^{k_{m+1}} \times \cdots \times R/P_{n}^{k_{n}}.$$
Hence, by Lemma \ref{lemma-74},
$R/P_i$ is a $t$-almost Dedekind domain for $i=1, \dots, m$.
Next, assume that $P_i$ is not maximal for $i =m+1, \dots, n$.
Then $P_i$ must be contained in a maximal $u$-ideal of $R$, say, $Q$,
but $R_Q$ is neither a DVR nor an SPR, a contradiction. Thus,
$R/P_{i}^{k_{i}} = T(R)/P_{i}^{k_{i}}T(R)$, which is an SPR.

(2) $\Rightarrow$ (4) By Proposition \ref{prop7.7}, $R[X]_{N_u}$ is an almost Dedekind ring.
Note that if $\{P_1, \dots , P_n\}$ is the set of minimal prime ideals of $R$,
then $\{P_i[X]_{N_u} \mid i=1, \dots , n\}$ is the set of minimal prime
ideals of $R[X]_{N_u}$. Hence, by the proof of (2) $\Rightarrow$ (1) in
Proposition \ref{almost dedekind}, $Z(R[X]_{N_u}) = \bigcup_{i=1}^nP_i[X]_{N_u}$.
Thus, $Z(R) =  \bigcup_{i=1}^nP_i$.

(4) $\Rightarrow$ (2) Let $P$ be a prime $Z$-ideal of $R$. Then $P$ is a $u$-ideal
by Proposition \ref{regLocal}. Thus, the proof is completed by the
same argument as the proof of (4) $\Rightarrow$ (2) in Proposition \ref{almost dedekind}.
\end{proof}

It is easy to see that an integral domain is a Krull domain if and only if
it is a $t$-almost Dedekind domain in which each principal ideal has a finite
number of minimal prime ideals. The next result is a general Krull ring analog, which
can be used as the definition of general Krull rings as in the case of Krull domains.

\begin{corollary} \label{char of strong Krull}
The following statements are equivalent for a ring $R$.
\begin{enumerate}
\item $R$ is a general Krull ring.

\item $R$ satisfies the following conditions.
      \begin{enumerate}
      \item $R = \underset{P \in X^{1}_r (R)}{\bigcap} R_{[P]}$.

      \item $R_P$ is a DVR for all $P \in X^{1}_r (R)$ and $R_P$ is
       an SPR for all prime $Z$-ideals $P$ of $R$.

      \item Each principal ideal of $R$ has a finite number of minimal prime ideals.
      \end{enumerate}
\item $R$ is a $u$-almost Dedekind ring in which each principal ideal has a finite number of minimal prime ideals.
\end{enumerate}
\end{corollary}

\begin{proof}
(1) $\Rightarrow$ (2) By Theorem \ref{Krull and strong Krull}, (a) holds.
Note that $u$-$\Max(R) = X^1_r(R) \cup \{P \in \Max(R) \mid P \mbox{ is minimal} \}$ by Theorem \ref{thm5.2}(5)-(6),
so (b) holds by Theorem \ref{u-Noetherian strong Krull}.
For (c), let $a \in R$ be a nonunit. If $a$ is regular, then each
minimal prime ideal of $aR$ is in $X_r^1(R)$, and since $R$ is a Krull ring,
such prime ideals are finite. Next, assume $a \in Z(R)$. Then $a=bc$ for some $b \in \reg(R)$ and $c \in Z(R)$ that
is a finite product of prime elements of $R$ in $Z(R)$ by Theorem \ref{thm5.2}(2).
Note that $bR$ has a finite number of minimal prime ideals, which are prime ideals in $X_r^1(R)$.
Note also that if $P$ is a minimal prime ideal of $bcR$,
then $P$ is minimal over either $bR$ or $cR$. Thus, $aR = bcR$ has a finite number of minimal prime ideals.

(2) $\Rightarrow$ (1) By Theorem \ref{Krull and strong Krull}, it suffices
to show that $T(R)$ is a PIR. By (b), $\dim(T(R)) =0$ and $T(R)$ is a locally PIR.
Moreover, by (c), $T(R)$ has a finite number of maximal ideals. Thus, $T(R)$ is a PIR.

(1) $\Rightarrow$ (3) $R$ is a $u$-almost Dedekind ring by Theorem \ref{u-Noetherian strong Krull}.
Thus, the proof is completed by the proof of (1) $\Rightarrow$ (2).

(3) $\Rightarrow$ (1) Clearly, $R$ has a finite number of minimal prime ideals.
Hence, $R$ is a finite direct product of $t$-almost Dedekind domains and SPRs by Corollary \ref{u-almost dedekind}.
Moreover, if $D$ is a direct summand of $R$ that is a $t$-almost Dedekind domain, then each nonzero
element of $D$ is contained in only finitely many height-one prime ideals of $D$.
Hence, $D$ is a Krull domain. Thus, $R$ is a general Krull ring.
\end{proof}

It is reasonable to say that $R$ is a regular almost Dedekind ring
(resp., regular $t$-almost Dedekind ring) if $(R_{[P]}, [P]R_{[P]})$
is a rank-one DVR for all regular maximal ideals (resp., regular maximal $t$-ideals) $P$ of $R$.
It is easy to see that (i) Dedekind rings (resp., Krull rings) are regular almost Dedekind ring (resp., regular $t$-almost Dedekind ring);
(ii) regular almost Dedekind rings are regular $t$-almost Dedekind rings and its regular dimension is at most one;
(iii) regular $t$-almost Dedekind rings are completely integrally closed;
and (iv) almost Dedekind rings (resp.,$u$-almost Dedekind rings) are
regular almost Dedekind rings (resp., regular $t$-almost Dedekind rings) \cite[Theorem 1]{ck02}.

\begin{proposition}
Let $R$ be a regular almost Dedekind ring (resp., regular $t$-almost Dedekind ring).
Then $R$ is a Dedekind ring (resp., Krull ring) if and only if $R$ is an r-Noetherian ring
(resp., $R$ satisfies the ascending chain condition on integral regular $t$-ideals).
\end{proposition}

\begin{proof}
(Krull ring case) This case follows from \cite[Proposition 2.2]{9} and \cite[Theorem 5]{m82}
that $R$ is a Krull ring if and only if $R$ satisfies the ascending chain condition on integral regular $v$-ideals
and is completely integrally closed.

(Dedekind ring case) It is clear that a Dedekind ring is r-Noetherian. Conversely,
assume that $R$ is an r-Noetherian ring. Then $R$ satisfies the ascending chain condition on integral regular $v$-ideals.
Thus, $R$ is a Krull ring \cite[Theorem 13]{7}. Moreover, since $\rdim (R) \leq 1$ by assumption,
$R$ is a Dedekind ring by Theorem \ref{Krull ring property}.
\end{proof}

Let $*$ be a star operation on a ring $R$. We say that $R$ is a {\it $*$-multiplication ring}
if, for all ideals $A$ and $B$ of $R$ with $A \subseteq B$, there exists an ideal $C$ of $R$ such that $A_* = (BC)_*$.
Hence, $d$-multiplication rings are just multiplication rings.
It is known that an integral domain is a Dedekind domain if and only if it is a multiplication ring \cite[Theorem 37.1]{3},
a general ZPI-ring is a multiplication ring, but a multiplication ring
need not be a general ZPI-ring \cite[Exercise 8 on p.224]{10}.

\begin{remark} \label{remark7.9}
A ring $R$ is an almost multiplication ring if $R_P$ is a multiplication ring
for all prime ideals $P$ of $R$ \cite[Theorem 2.7]{bp63}.
It is known that $R$ is an almost multiplication ring
if and only if $R_P$ is a general ZPI-ring for all prime ideals $P$ of $R$ \cite[Theorem 2.0]{bp63}.
Note that a general ZPI-ring with a unique maximal ideal is a DVR or
an SPR \cite[page 83]{a51}. Hence, the almost Dedekind ring of this paper is just
the almost multiplication ring (cf. \cite[page 761]{a76}). Thus,
the $R(X)$ of Proposition \ref{prop7.7} comes from \cite[Theorem 8(3)]{a76} and
the equivalence of (1), (2) and (4) in Proposition \ref{almost dedekind}
is from \cite[Theorem 12]{mott69} and \cite[Theorem 5]{a76}.
\end{remark}

It is known that a multiplication
ring is an almost multiplication ring \cite[Lemma 2.4]{bp63}
and a Noetherian almost multiplication ring is a multiplication ring \cite[Exercise 10 on page 224]{10}.
Hence, by Corollary \ref{general zpi},  $R$ is a general ZPI-ring if and only if $R$ is a Noetherian multiplication ring.

\begin{lemma} \label{u-multi}
\begin{enumerate}
\item A $u$-multiplication ring is a $u$-almost Dedekind ring.
\item A finite direct product of $u$-multiplication rings is a $u$-multiplication ring.
\end{enumerate}
\end{lemma}

\begin{proof}
(1) Let $R$ be a $u$-multiplication ring and $P$ be a maximal $u$-ideal of $R$.
Now, let $A$ and $B$ be ideals of $R_P$ with $A \subseteq B$.
Then $A = IR_P$ and $B = JR_P$ for some ideals $I$ and $J$ of $R$ with $I \subseteq J$ \cite[Theorems 4.3 and 4.4]{3}.
Hence, $I_u = (JC)_u$ for some ideal $C$ of $R$, so by Lemma~\ref{local-properties}(1),
\[
  A = IR_P = I_u R_P = (JC)_u R_P = (JC)R_P = (JR_P)(CR_P) = B(CR_P) \,.
\]
Thus, $R_P$ is a multiplication ring.
Then, by Remark \ref{remark7.9}, $R_P$ is either a DVR or an SPR.

(2) This follows directly from Lemmas \ref{direct product}(2) and \ref{u-diect product}(2).
\end{proof}

The next result is a general Krull ring analog of \cite[Proposition 39.4]{3}
that $R$ is a general ZPI-ring if and only if $R$ is a Noetherian multiplication ring.

\begin{corollary}
A ring $R$ is a general Krull ring if and only if
 $R$ is a $u$-Noetherian $u$-multiplication ring.
\end{corollary}

\begin{proof}
Suppose that $R$ is a general Krull ring.
Then $R$ is a $u$-Noetherian ring by Theorem~\ref{u-Noetherian strong Krull}
and $R$ is a finite direct product of Krull domains and SPRs.
Clearly, an SPR is a $u$-multiplication ring.
Now, let $D$ be a Krull domain and
$A \subseteq B$ be ideals of $D$. If $A = (0)$, then $(0) = (0)_u = (BA)_u$, so we may assume that $A \neq (0)$.
Then $B$ is $u$-invertible by assumption, and hence
if we let $C = B^{-1}A$, then $C$ is an ideal of $D$ such that $A_u = (BC)_u$.
Hence, $D$ is a $u$-multiplication ring.
Therefore, $R$ is a finite direct product of $u$-multiplication rings.
Thus, by Lemmas \ref{u-multi}(2), $R$ is a $u$-multiplication ring.
The converse follows by Theorem~\ref{u-Noetherian strong Krull} and Lemma~\ref{u-multi}(1).
\end{proof}

We end this paper with the
following diagram, which summarizes the exact relationship of the $15$ different types of rings which are
studied in this paper.

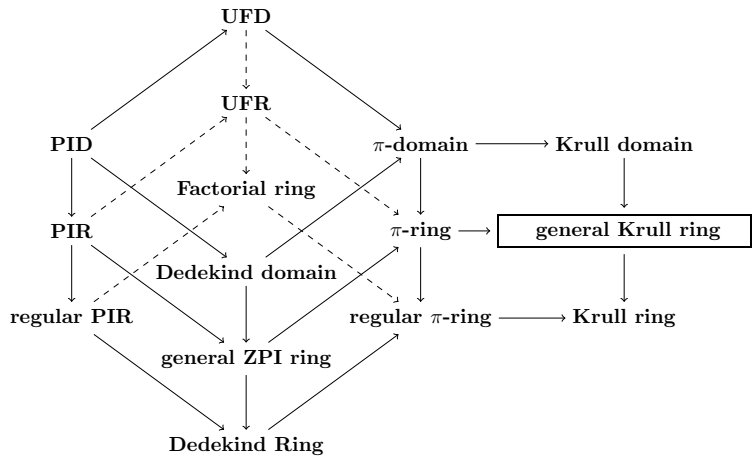
\begin{figure}[h]
\resizebox{.8\textwidth}{!}{%
\begin{tikzpicture}[node distance=2cm]
\title{Untergruppenverband der}
\node(PID)    at (1.5,1.5)     {\bf PID};
\node(PIR)    at (1.5,0)       {\bf PIR};
\node(rPIR)   at (1.5,-1.5)    {\bf regular PIR};

\node(UFD)    at (4.5,3.7)     {\bf UFD};
\node(UFR)    at (4.5,2.2)     {\bf UFR};
\node(F)      at (4.5,0.7)     {\bf Factorial ring};
\node(D)      at (4.5,-0.7)    {\bf Dedekind domain};
\node(ZPI)    at (4.5,-2.2)    {\bf general ZPI ring};
\node(DR)     at (4.5,-3.7)    {\bf Dedekind Ring};

\node(PiD)    at (7.5,1.5)     {\bf $\pi$-domain};
\node(PiR)    at (7.5,0)       {\bf $\pi$-ring};
\node(rPi)    at (7.5,-1.5)    {\bf regular $\pi$-ring};

\node(KD)     at (11,1.5)      {\bf Krull domain};
\node(SKR)    at (11,0)        {\bf $\framebox{ \quad  general Krull ring\quad}$};
\node(KR)     at (11,-1.5)     {\bf Krull ring};

\draw[->]         (PID)  -- (UFD);
\draw[->]         (PID)  -- (D);
\draw[->]         (UFD)  -- (PiD);
\draw[->]         (D)    -- (PiD);
\draw[->]         (PiD)  -- (KD);

\draw[->]         (PID)  -- (PIR);
\draw[dashed, ->] (UFD)  -- (UFR);
\draw[->]         (D)    -- (ZPI);
\draw[->]         (PiD)  -- (PiR);
\draw[->]         (KD)   -- (SKR);

\draw[dashed, ->] (PIR)  -- (UFR);
\draw[->]         (PIR)  -- (ZPI);
\draw[dashed, ->] (UFR)  -- (PiR);
\draw[->]         (ZPI)  -- (PiR);
\draw[->]         (PiR)  -- (SKR);

\draw[->]         (PIR)  -- (rPIR);
\draw[dashed, ->] (UFR)  -- (F);
\draw[->]         (ZPI)  -- (DR);
\draw[->]         (PiR)  -- (rPi);
\draw[->]         (SKR)  -- (KR);

\draw[dashed, ->] (rPIR) -- (F);
\draw[->]         (rPIR) -- (DR);
\draw[dashed, ->] (F)    -- (rPi);
\draw[->]         (DR)   -- (rPi);
\draw[->]         (rPi)  -- (KR);
\end{tikzpicture}
}
\end{figure}

\noindent
It is well known and easy to see that none of the implications are reversible and the three counter parts of each type of rings are the same when the rings are integral domains.

\begin{remark}
While correcting this paper according to the referee's comments, we found Juett's recent work \cite{j22}
where he used the $w$-operation to introduce the class of
general $w$-ZPI rings. Let $R$ be a ring. Juett called $R$ a general $w$-ZPI ring if every proper $w$-ideal of $R$ is a finite
$w$-product of prime $w$-ideals. Then, among other things, he proved that
$R$ is a general $w$-ZPI ring if and only if $R$ is a finite direct product of Krull domains and SPRs,
if and only if $R[X]_{N_w}$ is an Euclidean ring. Therefore, Juett's general $w$-ZPI ring is
exactly the general Krull ring of this paper (see Corollary \ref{coro5.7}).

Let $S$ be a commutative cancellative additive monoid and $R[S]$ be the semigroup ring of $S$ over $R$.
Juett has studied several factorization properties of $R[S]$ including Dedekind rings, $\pi$-rings, and UFRs (see, for example, \cite{jmn21, jmr21}).
In \cite{j22}, Juett also studied when $R[S]$ is a general $w$-ZPI ring.
\end{remark}

\vspace{.2cm}
\noindent
\textbf{Acknowledgements.}
The authors would like to thank the anonymous referee for his/her helpful comments.
This work was supported by Basic Science Research Program through the National Research Foundation of Korea (NRF) funded by the Ministry of Education (2017R1D1A1B06029867).

\end{document}